\documentclass[reqno,10pt]{amsart}
\usepackage{graphicx,epsfig}
\usepackage{epstopdf}
\usepackage{pdfpages}
\usepackage{amssymb}
\usepackage{empheq}
\usepackage{cases}
\usepackage{amsthm,amsmath}
\usepackage{caption,lipsum}
\usepackage{stmaryrd}
\usepackage{tabularx}
\usepackage{color}
\usepackage{empheq,float}
\DeclareGraphicsExtensions{.eps}
\usepackage{amssymb,amsmath,graphicx,amsfonts,euscript,mathrsfs}
\usepackage{hyperref}

\usepackage{booktabs,multirow}  
\usepackage{enumerate} 
\usepackage{enumitem} 
\allowdisplaybreaks[3]  
\usepackage{subcaption} 
\usepackage[numbers,sort]{natbib}  

\newtheorem{theorem}{Theorem}[section]
\newtheorem{lemma}[theorem]{Lemma}
\newtheorem{proposition}[theorem]{Proposition}
\newtheorem{corollary}[theorem]{Corollary}
\theoremstyle{definition}

\newtheorem{example}[theorem]{Example}

\theoremstyle{remark}
\newtheorem{remark}{Remark}

\newcommand{\eps}{\varepsilon}
\newcommand{\bR}{{\mathbb R}}

\newcommand{\bx}{\mathbf{x}}

\numberwithin{equation}{section}

\begin{document}
\title[Improved SAV for original energy stability]{Improved scalar auxiliary variable schemes for original energy stability of gradient flows}

\author[R. Chen]{Rui Chen}
\address{\hspace*{-12pt}R.~Chen: School of Science \& Key Laboratory of Mathematics and Information Networks (Ministry of Education), Beijing University of Posts and Telecommunications, Beijing, 100876, China}
\email{ruichen@bupt.edu.cn}


\author[T. Wang]{Tingfeng Wang}
\address{\hspace*{-12pt}T.~Wang: School of Mathematics and Statistics, Wuhan University, Wuhan, 430072, China}
\email{tingfengwang@whu.edu.cn}

\author[X. Zhao]{Xiaofei Zhao}
\address{\hspace*{-12pt}X.~Zhao: School of Mathematics and Statistics \& Computational Sciences Hubei Key Laboratory, Wuhan University, Wuhan, 430072, China}
\email{matzhxf@whu.edu.cn}




\maketitle

\begin{abstract}\noindent
Scalar auxiliary variable (SAV) methods are a class of linear schemes for solving gradient flows that are known for the stability of a `modified' energy. In this paper, 
we propose an improved SAV (iSAV) scheme that not only retains the complete linearity  but also ensures rigorously the stability of the original energy. The convergence and optimal error bound are rigorously established for the iSAV scheme and discussions are made for its high-order extension. Extensive numerical experiments are done to validate the convergence, robustness and energy stability of iSAV, and some comparisons are made. 


\noindent\textbf{AMS(2010) subject classifications.} 65M12, 35K20, 35K35, 35K55, 65Z05
\end{abstract}

\section{Introduction}

The gradient flow that consists of a driving free energy and a dissipative mechanism, widely occurs in mathematical modelling of natural science and engineering, e.g., the interface dynamics, crystallization, thin films, liquid crystals et al. \cite{Anderson1998Diffuse, Gurtin1996TwoPhase, Lowengrub1998QuasiIncompressible, Elder2002Modeling, Fraaije1993dynamic, Fraaije2003Model, Leslie1979Theory, Larson1990Arrested}. 
In this paper, we consider a free energy functional  $\mathcal{E}[\phi]$ defined for functions on a bounded  spatial domain $\Omega\subset\bR^d$ ($d=1,2,3$) and the  general  gradient flow of the following form:  
\begin{align}
	\frac{\partial\phi}{\partial t} = -\mathcal{G}\mu,\quad \mu = \delta \mathcal{E}/\delta \phi, \label{gradient flow 0}
\end{align}
where $\mathcal{G}$ is a nonnegative symmetric linear operator that determines the dissipative mechanism. The common dissipative mechanisms include the $L^{2}$ gradient flow (i.e., $\mathcal{G} = I$) and the $H^{-\alpha}$ gradient flow (i.e., $\mathcal{G} = (-\Delta)^{\alpha}$ with $0<\alpha\leq 1$).  We supplement 
\eqref{gradient flow 0} with boundary conditions satisfying the integration-by-parts, resulting in the elimination of all boundary terms for simplicity. Typically, periodic boundary conditions, homogeneous Neumann and Dirichlet boundary conditions are all applicable.  
With $\langle \phi,\psi \rangle := \int_{\Omega}{ \phi\psi }\,{\rm d}{\bf x}$ denoting the $L^{2}$ inner product, the  energy dissipation law of (\ref{gradient flow 0}) reads
\begin{align*}
	\frac{{\rm d} \mathcal{E}[\phi]}{{\rm d} t} = \left\langle \frac{\delta \mathcal{E}}{\delta \phi},\frac{\partial\phi}{\partial t} \right\rangle = -\langle \mu,\mathcal{G}\mu \rangle\leq 0.
\end{align*}

For accurately and efficiently solving (\ref{gradient flow 0}),  extensive numerical methods have been proposed. The classical approaches are the fully implicit schemes, e.g., the backward Euler, Crank-Nicolson and the convex splitting  \cite{Shen2010Numerical,Chen2019SecondOrder, Wise2010Unconditionally, Shen2012SecondOrder, Elliott1993GlobalDynamics, Eyre1998Unconditionally}. They can provide the decrease of the energy at discrete time level, but  nonlinear solvers are needed at each time step,  which is often costly. Concerning strong stability issues of explicit schemes, the semi-implicit schemes  which need only linear solvers at each time step, have become popular, particularly when some tailored stabilization techniques   \cite{Shen2010Numerical, Shen2015Efficient, Shen2015Decoupled, Chen2015Decoupled, Shen2014Decoupled, Li2017SecondOrder} can be proposed to overcome the stability constraint from explicitly treating  the nonlinearity. 
Another major  class of methods are the exponential time difference methods \cite{Du21review, Ju2018EnergyStability, Wang2016EfficientStable}, which can be designed to meet the energy stability and/or the maximum bound principle. For more complete and detailed  reviews, we refer to \cite{Du21review,Du2020phase,Shen2019NewClass} and the references therein.


Among the semi-implicit methods, the modern approach is to introduce some auxiliary variable to the problem which induces the state-of-art: invariant energy quadratization (IEQ) methods \cite{Chen2017SecondOrder, Yang2016LinearFirst, Yang2017EfficientLinear, Yang2017Linear, Yang2017Numerical, Yang2017NumericalThreeComponent} and the scalar auxiliary variable (SAV) methods \cite{Cheng2020LagrangeMultiplier,Shen2019NewClass, Huang2020HighlyEfficient,Shen2018ConvergenceSAV}. Both IEQ and SAV  provide framework that 
enables to construct linear and  high-order schemes, and  their main advantage is to provide the unconditional stability of an approximated/modified energy. 
Compared to IEQ, SAV is more explicit in the sense that the matrix for inversion 
at each time step is a constant, whereas that of IEQ is varying in the time iteration. We shall focus on the SAV
method in this work. 

Ever since the born  of the original SAV \cite{Shen2019NewClass}, the improvements on it have been started. For instance, \cite{Huang2020HighlyEfficient, Cheng2021GeneralizedSAV, Yang2020Roadmap, Liu2020ExponentialSAV, Huang2022new, Takhirov2024quad} considered the elimination of restrictions on the lower bound of the free energy, or the positivity-preservation of the  auxiliary variable.  Among all possible directions of improvements, the main target is for the stability of the  original energy, because the modified energy dissipation law in SAV (and IEQ) cannot guarantee  the  dissipation of the original model in theory and counterexamples can be observed in practical computing. To this aim,  effective stabilization terms have been designed and added to SAV methods \cite{zhang20,Zhang2020non,chen2019fast} (also stabilized-IEQ \cite{xu19}) to enhance the practical performance. \cite{Jiang2022ImprovingSAV, Liu2023RelaxedSAV} proposed relaxation techniques to improve the correlation between the modified energy and the original energy. Cheng et al. \cite{Cheng2020LagrangeMultiplier} proposed to use  the  Lagrange multiplier for the original energy-stability,   while the price is to solve a nonlinear algebraic equation per time level.

In this paper, we are going to propose an improved SAV (iSAV) scheme for solving (\ref{gradient flow 0}) particularly aiming for the original energy stability. It is linear with only the constant-coefficient matrix for inversion, and more importantly, it provides the dissipation of the original energy at discrete level as (see  Theorem \ref{Energy Law of iSAV-BE}): 
$$ \frac{1}{\tau}\left( \mathcal{E}[\phi^{n+1}]-\mathcal{E}[\phi^{n}] \right)\leq -\left\langle 
\mu^{n+1},\mathcal{G}\mu^{n+1} \right\rangle,\quad n\geq0.$$
The key ideas consist of: 1) replacing the numerical value of the scalar variable $r^n$ in the backward Euler discretization by the original functional $r[\phi^n]$ of the scalar variable; 2) introducing  a stabilization term. 
They are inspired from the work for the IEQ method \cite{Chen2022NovelLinearSchemes} but which is restricted to the double-well potential case. The proposed iSAV method here  works for general potential functions bounded from below, and we shall rigorously  establish its first-order error bound (see Theorem \ref{Theo eroor estimate}). A possible second-order extension will also be discussed. The theoretical energy-stability and accuracy will be supported by numerical experiments. Comparisons with the original SAV will be made to illustrate the improvement.

The rest of this paper is organized as follows. In Section \ref{sec:method}, we present the iSAV schemes and establish the  original energy stability result. In Section \ref{sec:analysis}, we provide an error estimate for the proposed scheme. The numerical results and comparisons are given in Section \ref{sec:result}, and some conclusions are drawn in Section \ref{sec:conclusion}. 
The following are some notations that will be frequently used in the paper. The standard $L^p=L^p(\Omega)$ space and the Sobolev space $H^s=H^s(\Omega)$ with $s>0$ are utilized. 
$L^p(0,T; V)$ denotes the $L^p$ space on the time interval $(0, T)$ with values in the functional space $V$. 
The norm of the space $V$ is denoted by $\|\cdot\|_{V}$ and specifically, the $L^2$ norm is denoted without a subscript. 
For a nonnegative symmetric linear operator, such as $\mathcal{G}$, we denote $\left\langle \phi, \mathcal{G}\phi \right\rangle^{\frac{1}{2}}=\|\mathcal{G}^{\frac{1}{2}}\phi\|$ for short. The notation $a=\mathcal{O}(b)$ or $a\lesssim b$ means that there exists a  constant $C>0$, whose value may vary from line to line, but is independent of the numerical discrete parameters (e.g., the step size $\tau$ or the time level $n$), such that $\left| a \right|\leq Cb$.


\section{Improved SAV scheme}\label{sec:method}
In this section, we shall present the main numerical schemes for solving the gradient flows. To be precise, let us clearly state our setup of the problem. We  focus on the following form of free energy:
\begin{align*}
    \mathcal{E}[\phi] = \frac{1}{2}\left\| \mathcal{L}^{\frac{1}{2}}\phi \right\|^2 + \int_{\Omega}{ F(\phi) }\,{\rm d}{\bf x}, \quad \Omega\subset\bR^d,\ d=1,2,3,
\end{align*}
and its gradient flow (\ref{gradient flow 0}) then reads as
\begin{subequations}\label{gradient flow}\begin{align}
    &\frac{\partial\phi}{\partial t} = -\mathcal{G}\mu, \quad t>0,\  \bx\in\Omega,\label{gradient flow 1} \\
    &\mu = \mathcal{L}\phi + f(\phi), \label{gradient flow 2}
\end{align}\end{subequations}
where $\mathcal{L}$ is a symmetric nonnegative linear operator that is independent of $\phi=\phi(\bx,t)$, 
 and $F(\phi)$ denotes the bulk part of the free energy density which is usually nonlinear 
(e.g., a double well form $F(\phi)=\frac{1}{4\varepsilon^2}(\phi^2-1)^2$) with $f(\phi):=F'(\phi)$. From the application point of view \cite{Shen2019NewClass},  $\int_{\Omega}{ F(\phi) }\,{\rm d}{\bf x}$ is often bounded from below. So without loss of generality, we assume that $$\int_{\Omega}{ F(\phi) }\,{\rm d}{\bf x} \geq C_0 > 0,$$ otherwise we can add a constant to $F(\phi)$ without changing the gradient flow (\ref{gradient flow}). The boundary of the bounded domain $\Omega$ is set to be  periodic for (\ref{gradient flow}) for simplicity. 

\subsection{The SAV framework} We begin with a   brief review of the SAV method introduced in \cite{Shen2019NewClass}. 
The first key step for the SAV approach is to introduce a scalar variable:  
$$ \mathcal{F}[\phi] := \int_{\Omega} F(\phi) \, d\mathbf{x},\quad r(t):=r[\phi] :=\sqrt{\mathcal{F}[\phi]}, $$ 
and then (\ref{gradient flow}) can be rewritten as a system:
\begin{subequations}\label{SAV Gradient flow}\begin{align}
    &\frac{\partial \phi}{\partial t} = -\mathcal{G}\mu, \quad t>0,\  \bx\in\Omega,\label{SAV Gradient flow 1}\\
    &\,\mu = \mathcal{L}\phi+\frac{r}{\sqrt{\mathcal{F}[\phi]}}f(\phi), \label{SAV Gradient flow 2} \\
    &\frac{{\rm d} r}{{\rm d} t} = \frac{1}{2\sqrt{\mathcal{F}[\phi]}}\int_{\Omega} f(\phi)\frac{\partial \phi}{\partial t}  \,{\rm d}{\bf x},\quad t>0. \label{SAV Gradient flow 3}
\end{align}\end{subequations}
Taking the $L^2$ inner products of (\ref{SAV Gradient flow 1}) and (\ref{SAV Gradient flow 2}) with $\mu$ and $\frac{\partial \phi}{\partial t}$, respectively,  multiplying (\ref{SAV Gradient flow 3}) by $2r$, and then by adding them together, the energy dissipative law of the gradient flow (\ref{gradient flow}) can be seen
\begin{align}\label{rate ex}
    \frac{{\rm d}\mathcal{E}[\phi]}{{\rm d}t}=\frac{{\rm d}}{{\rm d}t}\left(  \frac{1}{2}\left\| \mathcal{L}^{\frac{1}{2}}\phi \right\|^2 + r^{2} \right) = -\| \mathcal{G}^{\frac{1}{2}}\mu \|^2,\quad t>0.
\end{align}

Numerical discretizations can then be considered for (\ref{SAV Gradient flow}). 
By the backward Euler type of finite difference discretization with explicit treatment of  the nonlinear terms in  (\ref{SAV Gradient flow}), a first-order SAV scheme (SAV-BE) was constructed in \cite{Shen2019NewClass}: for $n\geq0$,
\begin{subequations}\label{SAV-BE scheme}\begin{align}
		&\frac{1}{\tau}(\phi^{n+1}-\phi^{n}) = -\mathcal{G}\mu^{n+1}, \label{SAV-BE scheme 1}\\
		&\mu^{n+1}= \mathcal{L}\phi^{n+1}+\frac{r^{n+1}}{\sqrt{\mathcal{F}[\phi^{n}]}}f(\phi^{n}), \label{SAV-BE scheme 2} \\
		&r^{n+1}-r^{n}=\frac{1}{2\sqrt{\mathcal{F}[\phi^{n}]}}\int_{\Omega}f(\phi^{n})(\phi^{n+1}-\phi^{n})\,{\rm d}{\bf x}  \label{SAV-BE scheme 3}.
\end{align}\end{subequations}
The first advantage of the scheme is its linearity, where $\phi^{n+1}$ can be obtained from  (\ref{SAV-BE scheme})  by solving two linear systems. We refer the readers to \cite{Shen2019NewClass} for the detailed update procedure. Moreover, by 
taking the $L^2$ inner products of (\ref{SAV-BE scheme 1}) and (\ref{SAV-BE scheme 2})  with $\mu^{n+1}$ and $\frac{1}{\tau}(\phi^{n+1}-\phi^{n})$, respectively,   multiplying (\ref{SAV-BE scheme 3}) by $2\frac{r^{n+1}}{\tau}$, and then adding them together, one can obtain a discrete energy dissipative law for SAV-BE (\ref{SAV-BE scheme}) as
\begin{align}
    \frac{1}{\tau}\left( E^{n+1} - E^{n} \right) \leq -\left\| \mathcal{G}^{\frac{1}{2}}\mu^{n+1} \right\|^2, \quad E^{n} := \frac{1}{2}\left\| \mathcal{L}^{\frac{1}{2}}\phi^{n} \right\|^2 + \left( r^{n} \right)^2,\quad n\geq0,\label{modified energy of SAV-BE}
\end{align}
which gives the `energy' stability. However, $r^{n}\neq r[\phi^{n}]$ for $n\geq 1$, and this may lead to oscillations in the evolution of the original energy $\mathcal{E}[\phi^{n}]$ of the gradient flow and the occurrence of numerical instability. 

In fact, for the behaviour of the original energy in the SAV-BE scheme (\ref{SAV-BE scheme}), we can give the following formal estimate (derivation given  in Appendix \ref{Appendix 0}): 
\begin{align}
    \frac{1}{\tau}\left( \mathcal{E}[\phi^{n+1}] - \mathcal{E}[\phi^{n}] \right) \leq -\| \mathcal{G}^{\frac{1}{2}}\mu^{n+1} \|^2 + \mathcal{O}(\tau), \quad n\geq0,\label{Energy Law of SAV-BE}
\end{align}
which indicates that the original energy stability needs the time step to be small enough.  This can get worse  
when the  nonlinearity is stiff, e.g., $F(\phi)=\frac{1}{\varepsilon^2}(\phi^2-1)^2$ for some parameter $0<\eps<1$, where (\ref{Energy Law of SAV-BE}) will at least deteriorate to:
\begin{align*}
    \frac{1}{\tau}\left( \mathcal{E}[\phi^{n+1}] - \mathcal{E}[\phi^{n}] \right) \leq -\| \mathcal{G}^{\frac{1}{2}}\mu^{n+1} \|^2 + \mathcal{O}(\tau/\varepsilon^2),\quad n\geq0.
\end{align*}
So if $\varepsilon\ll1$ while $\tau$ is not sufficiently small, the original energy of the SAV schemes could easily be unstable and the oscillations could get wilder as $\eps$ reduces. Higher-order SAV schemes \cite{Shen2019NewClass,Shen2018ConvergenceSAV} exhibit such  problem as well.  Suitable  stabilization techniques need to be considered for SAV \cite{Shen2019efficient} for the better practical performance.  

\subsection{Improved SAV scheme} Concerning the pity of SAV schemes, we aim to suggest a patch for it. 
We begin with the work on the first-order scheme. Reviewing the SAV-BE scheme,  $r^{n}$ is used to update $r^{n+1}$ in  (\ref{SAV-BE scheme 3}), which indeed leads to a growing disparity between the values of $r^{n}$ and $r[\phi^{n}]$. Therefore, we consider to replace $r^{n}$ by $r[\phi^{n}]$ in (\ref{SAV-BE scheme 3}) and add a stabilization term in (\ref{SAV-BE scheme 2}). This results in the following scheme that we name as the first-order \emph{improved SAV-BE scheme (iSAV-BE)}: for $n\geq0$,
\begin{subequations}\label{iSAV-BE scheme}\begin{align}
    &\frac{1}{\tau}(\phi^{n+1}-\phi^{n}) = -\mathcal{G}\mu^{n+1},\label{iSAV-BE scheme 1}\\
    &\mu^{n+1} = \mathcal{L}\phi^{n+1}+\frac{\tilde{r}^{n+1}}{\sqrt{\mathcal{F}[\phi^{n}]}}f(\phi^n)+S(\phi^{n+1}-\phi^{n}),\label{iSAV-BE scheme 2} \\
    &\tilde{r}^{n+1}-r[\phi^{n}] = \frac{1}{2\sqrt{\mathcal{F}[\phi^{n}]}}\int_{\Omega}{ f(\phi^{n})(\phi^{n+1}-\phi^{n}) }\,{\rm d}{\bf x}, \label{iSAV-BE scheme 3}
\end{align}\end{subequations}
where $S\geq 0$ is a stabilizing parameter to be determined. 

For the proposed iSAV-BE (\ref{iSAV-BE scheme}), firstly, it is still a linear scheme and $\phi^{n+1}$ can be  obtained efficiently as follows. Note that the two auxiliary variables $\mu^{n+1}$ and $\tilde{r}^{n+1}$ can be eliminated by combining (\ref{iSAV-BE scheme 1})$\sim$(\ref{iSAV-BE scheme 3}), and we can find 
\begin{align*}
    \frac{1}{\tau}\left( \phi^{n+1}-\phi^{n} \right) =&  -\mathcal{G}\mathcal{L}\phi^{n+1} - \frac{ \mathcal{G}f(\phi^{n}) }{ \sqrt{\mathcal{F}[\phi^{n}]} }\left( r[\phi^{n}] + \frac{1}{2}\left\langle \frac{ f(\phi^{n}) }{ \sqrt{\mathcal{F}[\phi^{n}]} },\phi^{n+1}-\phi^{n} \right\rangle \right) - S\mathcal{G}(\phi^{n+1}-\phi^{n}).
\end{align*}
Denoting $b^{n} := f(\phi^{n})/\sqrt{\mathcal{F}[\phi^{n}]}$, the above equation can be rewritten as
\begin{align}
    &\left[ I+\tau\mathcal{G}\left( \mathcal{L} + S \right) \right]\phi^{n+1} + \frac{\tau}{2}\left\langle b^{n}, \phi^{n+1} \right\rangle\mathcal{G}b^{n} \nonumber\\
     =& \left( I+\tau S\mathcal{G} \right)\phi^{n} - \tau\left( r[\phi^{n}] - \frac{1}{2}\left\langle b^{n}, \phi^{n} \right\rangle \right)\mathcal{G}b^{n} =: g^{n}. \label{iSAV linear solve 1}
\end{align}
Since $\mathcal{G}$ and $\mathcal{L}$ are nonnegative operators, the operator $I+\tau\mathcal{G}(\mathcal{L} + S)$ is invertible and we denote $$\mathcal{A}^{-1} := \left[ I+\tau\mathcal{G}(\mathcal{L} + S) \right]^{-1}.$$ Then, by multiplying (\ref{iSAV linear solve 1}) on both sides with $\mathcal{A}^{-1}$ and taking the $L^2$ inner product with $b^{n}$, we can obtain
\begin{align}\label{isav-be solver eq1}
    \langle b^{n},\phi^{n+1}\rangle + \frac{\tau}{2}\langle b^{n},\phi^{n+1}\rangle\langle b^{n},\mathcal{A}^{-1}\mathcal{G} b^{n}\rangle = \langle b^{n},\mathcal{A}^{-1}g^{n}\rangle,\quad n\geq0.
\end{align}
Based on the above two equations, we can compute $\phi^{n+1}$ (semi-discretized numerical solution) from (\ref{iSAV-BE scheme}) through the  following steps: for $n\geq0$,
\begin{enumerate}[label=(\roman*),leftmargin=*]
    \item compute $\langle b^{n},\mathcal{A}^{-1}\mathcal{G} b^{n}\rangle$ and $\langle b^{n},\mathcal{A}^{-1}g^{n}\rangle$, where $\mathcal{A}^{-1}\mathcal{G} b^{n}$ and $\mathcal{A}^{-1}g^{n}$ can be calculated by solving two linear systems with constant coefficients;
    \item compute $\langle b^{n},\phi^{n+1}\rangle$ from (\ref{isav-be solver eq1}) as 
    \begin{align*}
        \langle b^{n},\phi^{n+1}\rangle = \langle b^{n},\mathcal{A}^{-1}g^{n}\rangle/\left( 1 + \frac{\tau}{2}\langle b^{n},\mathcal{A}^{-1}\mathcal{G} b^{n}\rangle \right),
    \end{align*}
    where $1+\frac{\tau}{2}\langle b^{n},\mathcal{A}^{-1}\mathcal{G} b^{n}\rangle\geq 1$  clearly;
    \item obtain $\phi^{n+1}$ from (\ref{iSAV linear solve 1}) as
    $$ \phi^{n+1} = -\frac{\tau}{2}\langle b^{n},\phi^{n+1}\rangle\mathcal{A}^{-1}\mathcal{G} b^{n} + \mathcal{A}^{-1}g^{n}. $$
\end{enumerate}

The invertibility of $\mathcal{A}$ also indicates the uniqueness of the numerical solution from the iSAV-BE scheme (\ref{iSAV-BE scheme}). We shall rigorously prove in the next section that it is a first order accurate method in time.  For the practical implementation, the spatial discretizations can be done by the Fourier pseudo-spectral method \cite{Shen2011} under the periodic setup, and the involved linear systems in above procedure (i)\&(iii) are then automatically diagonalized by the fast Fourier transform. Thus, the whole scheme can work efficiently in a matrix-free way.  
For other spatial discretizations in general,  
since $\mathcal{A}^{-1}$ is constant in the temporal iteration, it can be pre-computed and stored once for all. 

Next, we will give the main improvement from the iSAV-BE scheme that is the  stability of the original energy of the gradient flow. 

\begin{theorem}[Original energy stability]\label{Energy Law of iSAV-BE}Suppose (\ref{gradient flow}) is solved by the iSAV-BE scheme (\ref{iSAV-BE scheme}) till some $0<T<\infty$, and we choose
   the stabilizing parameter $S$  to satisfy
    \begin{equation}\label{S1 condition}
        S \geq S_0:=\frac{1}{2}\max_{0\leq n \leq T/\tau - 1}\;\max_{\mathbf{x}\in\Omega,\,\rho
        \in(0,1)}\left\{ f'\left( \left( 1-\rho \right)\phi^{n} + \rho\phi^{n+1} \right) \right\}.
    \end{equation}
    Then, the iSAV-BE scheme (\ref{iSAV-BE scheme}) has the following discrete energy disspative law,
    \begin{align} 
        \frac{1}{\tau}\left( \mathcal{E}[\phi^{n+1}]-\mathcal{E}[\phi^{n}] \right)\leq -\| \mathcal{G}^\frac{1}{2}\mu^{n+1} \|^2,\quad 
    0\leq n\leq T/\tau-1. \label{Original Energy Law of iSAV-BE}
    \end{align}
    After Corollary \ref{cor h2 bound}, we will further show that $S_0$ can be a fixed finite number, and so an $S\in[S_0,\infty)$ can be chosen.  
\end{theorem}
\begin{proof}
    By taking the $L^2$ inner products of (\ref{iSAV-BE scheme 1}) and (\ref{iSAV-BE scheme 2}) with $\mu^{n+1}$ and $(\phi^{n+1}-\phi^{n})/\tau$,  respectively, and multiplying (\ref{iSAV-BE scheme 3}) by $\frac{2\tilde{r}^{n+1}}{\tau}$, we can obtain the following equations:
    \begin{subequations}\label{oes eq}
    \begin{align}
        &\frac{1}{\tau}\left\langle \phi^{n+1}-\phi^{n},\mu^{n+1} \right\rangle = -\left\| \mathcal{G}^\frac{1}{2}\mu^{n+1} \right\|^2,
        \label{oes eq1}\\
        &\frac{1}{\tau}\left\langle \mu^{n+1},\phi^{n+1}-\phi^{n} \right\rangle = \frac{1}{\tau}\left\langle \mathcal{L}\phi^{n+1},\phi^{n+1}-\phi^{n} \right\rangle + \frac{\tilde{r}^{n+1}}{\tau}\left\langle \frac{ f(\phi^{n}) }{ \sqrt{\mathcal{F}[\phi^{n}]} }, \phi^{n+1}-\phi^{n} \right\rangle + \frac{S}{\tau}\left\| \phi^{n+1} - \phi^{n} \right\|^2 \nonumber \\
        &\hspace{3.14cm} = \frac{1}{2\tau}\left( \left\| \mathcal{L}^\frac{1}{2}\phi^{n+1} \right\|^2 - \left\| \mathcal{L}^\frac{1}{2}\phi^{n} \right\|^2 + \left\| \mathcal{L}^\frac{1}{2}(\phi^{n+1}-\phi^{n}) \right\|^2 \right) + \frac{S}{\tau}\left\| \phi^{n+1} - \phi^{n} \right\|^2  \nonumber\\
        &\hspace{3.7cm} + \frac{\tilde{r}^{n+1}}{\tau}\left\langle \frac{ f(\phi^{n}) }{ \sqrt{\mathcal{F}[\phi^{n}]} }, \phi^{n+1}-\phi^{n} \right\rangle, \label{oes eq2} \\
        &\frac{1}{\tau}\left[ \left( \tilde{r}^{n+1} \right)^2 - \left( r[\phi^{n}] \right)^2 + \left( \tilde{r}^{n+1} - r[\phi^{n}] \right)^2 \right]  = \frac{\tilde{r}^{n+1}}{\tau}\left\langle \frac{ f(\phi^{n}) }{ \sqrt{\mathcal{F}[\phi^{n}]} }, \phi^{n+1}-\phi^{n} \right\rangle.\label{oes eq3}
    \end{align}
    \end{subequations}
    Plugging (\ref{oes eq1}) into the left hand side of (\ref{oes eq2}) and plugging (\ref{oes eq3}) into the right hand side of (\ref{oes eq2}), 
    we get
    \begin{align*}
    &-\frac{1}{\tau}\left[ \left( \tilde{r}^{n+1} \right)^2 - \left( r[\phi^{n}] \right)^2 + \left( \tilde{r}^{n+1} - r[\phi^{n}] \right)^2 \right]\\
    =&    \left\| \mathcal{G}^\frac{1}{2}\mu^{n+1} \right\|^2
    +\frac{1}{2\tau}\left( \left\| \mathcal{L}^\frac{1}{2}\phi^{n+1} \right\|^2 - \left\| \mathcal{L}^\frac{1}{2}\phi^{n} \right\|^2 + \left\| \mathcal{L}^\frac{1}{2}(\phi^{n+1}-\phi^{n}) \right\|^2 \right) + \frac{S}{\tau}\left\| \phi^{n+1} - \phi^{n} \right\|^2. 
    \end{align*}
    Plugging (\ref{iSAV-BE scheme 3}) into the left hand side of the above, we can find
    \begin{align}
        \frac{1}{\tau}\left( \tilde{E}^{n+1} - \mathcal{E}[\phi^{n}] \right) = -\left\| \mathcal{G}^{\frac{1}{2}}\mu^{n+1} \right\|^2 - \frac{S}{\tau}\left\| \phi^{n+1} - \phi^{n} \right\|^2 - \mathcal{N}_{1}^{n}, \label{20240111-03}
    \end{align}
    where $\tilde{E}^{n+1}$ and $\mathcal{N}_{1}^{n}$ denote
    \begin{align*}
        &\tilde{E}^{n+1} = \frac{1}{2}\left\| \mathcal{L}^{\frac{1}{2}}\phi^{n+1} \right\|^2 + \left( \tilde{r}^{n+1} \right)^2, \\
        &\mathcal{N}_{1}^{n} = \frac{1}{2\tau}\left\| \mathcal{L}^{\frac{1}{2}}\left( \phi^{n+1} - \phi^{n} \right) \right\|^2 + \frac{1}{4\tau}\frac{1}{\mathcal{F}[\phi^{n}]}\left( \int_{\Omega}{ f(\phi^{n})\left( \phi^{n+1} - \phi^{n} \right) }\,{\rm d}{\bf x} \right)^2.
    \end{align*}
    
    Next, we proceed to derive the difference between $\tilde{E}^{n+1}$ and $\mathcal{E}[\phi^{n}]$. Note that by the Taylor expansion, 
    \begin{align*}
        &\mathcal{F}\left[\phi^{n} + \rho\left(\phi^{n+1} - \phi^{n}\right)\right] \nonumber\\
        =& \mathcal{F}\left[{\phi^{n}}\right] + \frac{{\rm d}}{{\rm d}\rho}\mathcal{F}\left[\phi^{n} + \rho\left(\phi^{n+1} - \phi^{n}\right)\right]\Big|_{\rho=0}\cdot\rho + \frac{{\rm d}^2}{{\rm d}\rho^2}\mathcal{F}\left[\phi^{n} + \rho\left(\phi^{n+1} - \phi^{n}\right)\right]\Big|_{\rho=\rho_{0}}\cdot\frac{\rho^2}{2} \nonumber\\
        =& \mathcal{F}\left[\phi^{n}\right] + \rho\int_{\Omega}{ f\left(\phi^{n}\right)\left(\phi^{n+1} - \phi^{n}\right) }\,{\rm d}{\bf x} + \frac{\rho^2}{2}\int_{\Omega}{ f'\left( (1-\rho_{0})\phi^{n} + \rho_{0}\phi^{n+1} \right)\left(\phi^{n+1} - \phi^{n}\right)^2 }\,{\rm d}{\bf x},
    \end{align*}
    for some $\rho_{0}\in(0,\rho)$, and so  with $\rho=1$ we have 
    \begin{align}
        \left( r[\phi^{n+1}] \right)^2 - \left( r[\phi^{n}] \right)^2 &= \mathcal{F}[\phi^{n+1}] - \mathcal{F}[\phi^{n}] \nonumber\\
        &= \int_{\Omega}{ f(\phi^{n})\left( \phi^{n+1} - \phi^{n} \right) }\,{\rm d}{\bf x} + \frac{1}{2}\int_{\Omega}{ f'\left( (1-\rho_{0})\phi^{n} + \rho_{0}\phi^{n+1} \right)(\phi^{n+1} - \phi^{n})^2 }\,{\rm d}{\bf x}. \label{20240111-01}
    \end{align}
    Additionally from (\ref{iSAV-BE scheme 3}): $\tilde{r}^{n+1} =r[\phi^{n}]+ \frac{1}{2\sqrt{\mathcal{F}[\phi^{n}]}}\int_{\Omega}{ f(\phi^{n})(\phi^{n+1}-\phi^{n}) }\,{\rm d}{\bf x}$, then by taking the square on both sides,  we obtain 
    \begin{align}
        &\left( \tilde{r}^{n+1} \right)^2 - \left( r[\phi^{n}] \right)^2 = \int_{\Omega}{ f(\phi^{n})\left( \phi^{n+1} - \phi^{n} \right) }\,{\rm d}{\bf x} + \frac{1}{4\mathcal{F}[\phi^{n}]}\left( \int_{\Omega}{ f(\phi^{n})\left( \phi^{n+1} - \phi^{n} \right) }\,{\rm d}{\bf x} \right)^2.  \label{20240111-02}
    \end{align}
    Subtracting (\ref{20240111-02}) from (\ref{20240111-01}), we have
    \begin{align}
        \left( r[\phi^{n+1}] \right)^2 - \left( \tilde{r}^{n+1} \right)^2 =& \frac{1}{2}\int_{\Omega}{ f'\left( (1-\rho_{0})\phi^{n} + \rho_{0}\phi^{n+1} \right)(\phi^{n+1} - \phi^{n})^2 }\,{\rm d}{\bf x} \nonumber\\
        & - \frac{1}{4\mathcal{F}[\phi^{n}]}\left( \int_{\Omega}{ f(\phi^{n})\left( \phi^{n+1} - \phi^{n} \right) }\,{\rm d}{\bf x} \right)^2. \label{20240111-05}
    \end{align}
    Then by (\ref{20240111-03}), 
    \begin{align*}
        \frac{1}{\tau}\left( \mathcal{E}[\phi^{n+1}] - \mathcal{E}[\phi^{n}] \right) =& \frac{1}{\tau}\left( \tilde{E}^{n+1} - \mathcal{E}[\phi^{n}] \right) + \frac{1}{\tau} \left( \mathcal{E}[\phi^{n+1}] - \tilde{E}^{n+1} \right) \nonumber\\
        =& -\left\| \mathcal{G}^{\frac{1}{2}}\mu^{n+1} \right\|^2 - \frac{S}{\tau}\left\| \phi^{n+1} - \phi^{n} \right\|^2 - \mathcal{N}_{1}^{n}+
    \frac{1}{\tau} \left( \mathcal{E}[\phi^{n+1}] - \frac{1}{2}\left\| \mathcal{L}^{\frac{1}{2}}\phi^{n+1} \right\|^2 - \left( \tilde{r}^{n+1} \right)^2\right),
    \end{align*}
    and by plugging (\ref{20240111-05}) into the above, we get
    \begin{align*}
        \frac{1}{\tau}\left( \mathcal{E}[\phi^{n+1}] - \mathcal{E}[\phi^{n}] \right) 
        =& \frac{1}{2\tau}\int_{\Omega}{ f'\left( (1-\rho_{0})\phi^{n} + \rho_{0}\phi^{n+1} \right)(\phi^{n+1} - \phi^{n})^2 }\,{\rm d}{\bf x}- \frac{S}{\tau}\left\| \phi^{n+1} - \phi^{n} \right\|^2\nonumber\\
    &-\left\| \mathcal{G}^{\frac{1}{2}}\mu^{n+1} \right\|^2 - \mathcal{N}_{2}^{n} ,
    \end{align*}
    with 
    $$ \mathcal{N}_{2}^{n} = \frac{1}{2\tau}\left\| \mathcal{L}^{\frac{1}{2}}\left( \phi^{n+1} - \phi^{n} \right) \right\|^2 + \frac{1}{2\tau}\frac{1}{\mathcal{F}[\phi^{n}]}\left( \int_{\Omega}{ f(\phi^{n})\left( \phi^{n+1} - \phi^{n} \right) }\,{\rm d}{\bf x} \right)^2\geq0. $$
    With the condition on the stabilizing parameter (\ref{S1 condition}), we can immediately conclude the original energy dissipative law (\ref{Original Energy Law of iSAV-BE}).
\end{proof}

\begin{remark}
    The stabilizing term $S(\phi^{n+1} - \phi^{n})$ ensures rigorous stability in the dissipation of original energy, and this term is of $\mathcal{O}(\tau)$.  Thus, the inclusion of it in the scheme does not affect the overall convergence order. 
\end{remark}

\begin{remark}\label{Remark r-r BE}
    Assuming that $\phi^{n},f(\phi^{n}),f'(\phi^{n})\;(n=0,1,\cdots)$ and $\phi_{t}$ are bounded, we can formally see that $r[\phi^{n+1}] - \tilde{r}^{n+1} = \mathcal{O}(\tau^2)$ from (\ref{20240111-05}) for iSAV-BE scheme. On the other hand, for the SAV-BE scheme (\ref{SAV-BE scheme}), we  have
    \begin{align*}
        r[\phi^{n+1}] - r^{n+1} = r[\phi^{n}] - r^{n} + \mathcal{O}(\tau^2),
    \end{align*} 
    which is derived in Appendix \ref{Appendix 0}. That means after the accumulation in the temporal iterations, the SAV-BE scheme in general offers only  $r[\phi^{n+1}] - r^{n+1} = \mathcal{O}(\tau)$. This will be shown numerically later.
\end{remark}

\begin{remark}\label{stable tech}
    In practice, the SAV schemes will need to use some  stabilization strategy. Taking the SAV-BE (\ref{SAV-BE scheme}) scheme as an example, a stabilization term is often extracted from $f(\phi^{n})$ and (\ref{SAV-BE scheme 2}) would be modified as
    \begin{align*}
        \mu^{n+1} = S\phi^{n+1} +  \mathcal{L}\phi^{n+1}+\frac{r^{n+1}}{\sqrt{\mathcal{F}[\phi^{n}]}}\left( f(\phi^n) - S\phi^{n} \right).
    \end{align*}
    Such kind of strategy works well in the literature \cite{Shen2019NewClass,Shen2019efficient} for numerical simulations,  while unfortunately it lacks of rigorous analysis and cannot have the theoretical guarantee for  the original energy stability. Moreover, 
    the size of the stabilization parameter for SAV  \cite{Shen2019efficient} is sometimes unknown in advance. For iSAV, such condition is to some extent more explicitly given in (\ref{S1 condition}).
\end{remark}

\subsection{Possible extension}
The SAV schemes can have high-order versions. For instance,   a second-order scheme can be obtained by using the backward differentiation formula (SAV-BDF) \cite{Shen2019NewClass}: for $n\geq1,$
\begin{subequations}\label{SAV-BDF scheme} \begin{align}
    &\frac{1}{2\tau}\left( 3\phi^{n+1} - 4\phi^{n} + \phi^{n-1} \right) = -\mathcal{G}\mu^{n+1}, \label{SAV-BDF scheme 1} \\
    &\mu^{n+1} = \mathcal{L}\phi^{n+1} + \frac{r^{n+1}}{\sqrt{\mathcal{F}[\phi_{*}^{n+1}]}}f(\phi_{*}^{n+1}), \label{SAV-BDF scheme 2} \\
    &3r^{n+1} - 4r^{n} + r^{n-1} = \frac{1}{2\sqrt{\mathcal{F}[\phi_{*}^{n+1}]}}\int_{\Omega}{ f(\phi_{*}^{n+1})\left( 3\phi^{n+1} - 4\phi^{n} + \phi^{n-1} \right) }\,{\rm d}{\bf x}.\label{SAV-BDF scheme 3}
\end{align}\end{subequations}
The extension of iSAV to high order is, however, not easy unfortunately and it is still under going. Here we discuss some possibility.   

Borrowing the modification used for iSAV-BE, we can introduce similar patch to the above SAV-BDF to write down the following \emph{iSAV-BDF scheme}:
\begin{subequations}\label{iSAV-BDF scheme} \begin{align}
    &\frac{1}{2\tau}\left( 3\phi^{n+1} - 4\phi^{n} + \phi^{n-1} \right) = -\mathcal{G}\mu^{n+1}, \label{iSAV-BDF scheme 1} \\
    &\mu^{n+1} = \mathcal{L}\phi^{n+1} + \frac{\tilde{r}^{n+1}}{\sqrt{\mathcal{F}[\phi_{*}^{n+1}]}}f(\phi_{*}^{n+1}) + S(\phi^{n+1} - 2\phi^{n} + \phi^{n-1}), \label{iSAV-BDF scheme 2} \\
    &3\tilde{r}^{n+1} - 4r[\phi^{n}] + r[\phi^{n-1}] = \frac{1}{2\sqrt{\mathcal{F}[\phi_{*}^{n+1}]}}\int_{\Omega}{ f(\phi_{*}^{n+1})\left( 3\phi^{n+1} - 4\phi^{n} + \phi^{n-1} \right) }\,{\rm d}{\bf x}, \label{iSAV-BDF scheme 3}
\end{align}\end{subequations}
where $\phi_{*}^{n+1}:=2\phi^{n}-\phi^{n-1}$ is the second-order extrapolation for $\phi^{n+1}$. The iSAV-BDF scheme is still linear for solving. 
Although by numerical experiments later, we shall show that the iSAV-BDF (\ref{iSAV-BDF scheme}) can remarkably improve the stability for the original energy compared to the SAV-BDF (\ref{SAV-BDF scheme}), 
 we are encountering difficulty to rigorously prove its  original energy stability. By assuming enough  smoothness and boundedness for $f,\phi$ and $\phi^{n}\;(n=0,1,\cdots)$, what we are able to show for 
 the iSAV-BDF scheme (\ref{iSAV-BDF scheme}) is
\begin{align}
    \frac{1}{\tau}\left( \mathcal{E}_{2}[\phi^{n+1},\phi^{n}] - \mathcal{E}_{2}[\phi^{n},\phi^{n-1}] \right) \leq -\left\| \mathcal{G}^{\frac{1}{2}}\mu^{n+1} \right\|^2 + \mathcal{O}(\tau^2), \label{Energy Law of iSAV-BDF}
\end{align}
where $\mathcal{E}_{2}$ is defined as
\begin{align*}
    \mathcal{E}_{2}[\phi^{n}, \phi^{n-1}] :=& \frac{1}{4}\left( \left\| \mathcal{L}^{\frac{1}{2}}\phi^{n} \right\|^2 + \left\| \mathcal{L}^{\frac{1}{2}}\left( 2\phi^{n}-\phi^{n-1} \right) \right\|^2 \right) + \frac{1}{2}\left[ \left( r[\phi^{n}] \right)^2 + \left( 2r[\phi^{n}]-r[\phi^{n-1}] \right)^2 \right] \\
    & + \frac{S}{2}\left\| \phi^{n} - \phi^{n-1} \right\|^2.
\end{align*}
We leave the derivation of (\ref{Energy Law of iSAV-BDF}) to Appendix \ref{Appendix 1}. 

\begin{remark}\label{Remark r-r BDF}
    Similarly as in Remark \ref{Remark r-r BE},  we can derive the difference between the auxiliary variables $r^{n}$ and  $r[\phi^n]$ in SAV-BDF and iSAV-BDF  as
    \begin{align*}
        &r[\phi^{n}] - \tilde{r}^{n} = \mathcal{O}(\tau^3), \quad   \text{for iSAV-BDF (\ref{iSAV-BDF scheme})}; \\
        &r[\phi^{n}] - r^{n} = \mathcal{O}(\tau^2), \quad  \text{for SAV-BDF (\ref{SAV-BDF scheme})}. 
    \end{align*}
    The detailed derivations are omitted here for brevity.
\end{remark}

\section{Convergence analysis}\label{sec:analysis}
In this section, we carry out  rigorous convergence analysis for the proposed iSAV-BE scheme (\ref{iSAV-BE scheme}).  The main convergence theorem is stated in Section \ref{subsec3 main}, and the rest of the section is devoted for the proof. In the analysis, a large number of constants $C$ will appear whose value may vary from line to line, and  we will use the notation $C(\cdot)$ to indicate its dependence when needed.

\subsection{Main convergence result}\label{subsec3 main}
We restrict the spatial dimension to be no more than three dimensions, i.e., $d\leq 3$, and consider only one typical  case of the gradient flow for analysis: the Cahn-Hilliard equation, i.e., $\mathcal{L}=-\Delta$ and $\mathcal{G}=-\Delta$ in \eqref{gradient flow}. Assume that $F\in C^{5}(\mathbb{R})$ and $f$ satisfies 
\begin{subequations}\label{f condition}\begin{align}
    &|f'(z)|\leq C\cdot\big( |z|^p + 1 \big), \quad\ p>0 \;\;\ {\rm arbitrary\;\;if} \;\; d = 1,2; \quad 0<p<4 \;\; {\rm if} \;\; d = 3, \label{f condition 1}\\
    &|f''(z)|\leq C\cdot\big( |z|^{p'} + 1 \big), \quad {p'}>0 \;\; {\rm arbitrary\;\;if} \;\; d = 1,2; \quad 0<{p'}<3 \;\; {\rm if} \;\; d = 3, \label{f condition 2}
\end{align}\end{subequations}
and the exact solution $\phi$ of \eqref{gradient flow} up to  $0\leq t\leq T$ with some $0<T<\infty$ satisfies \cite{Shen2018ConvergenceSAV}: 
\begin{align}
    \phi\in L^{\infty}(0,T;H^{2})\cap L^{\infty}(0,T;W^{1,\infty}), \quad \phi_{t}\in L^{\infty}(0,T;H^{-1})\cap L^{2}(0,T;H^{1}), \quad \phi_{tt}\in L^{2}(0,T;H^{-1}). \label{Assumption of exact solution 0}
\end{align}
The main convergence result of the iSAV-BE scheme (\ref{iSAV-BE scheme}) is stated as the following theorem. 
\begin{theorem}[Convergence and error bound]\label{Theo eroor estimate}
    Denote $e_{\phi}^{n} = \phi^{n} - \phi(t_{n})$, $e_{\mu}^{n} = \mu^{n} - \mu(t_{n})$, $\tilde{q}^{n+1} = \tilde{r}^{n+1} - r(t_{n+1})$ with $\phi^{n},\mu^{n},\tilde{r}^{n+1}$ obtained from iSAV-BE (\ref{iSAV-BE scheme}) for any $S\geq0$, and let both (\ref{f condition})  and (\ref{Assumption of exact solution 0}) hold. Assume that $\phi^{0}\in H^{8}(\Omega)$ and denote $\left\| \phi \right\|_{L^{\infty}(0,T;H^{1})}^2=M$,  then there exist some constants $\tau_{0},\,C_1,C_{2}>0$ independent of $n,\tau$ such that for any $0<\tau<\tau_{0}$, 
    \begin{align}\label{uniform bound}
        \left\| \phi^{n} \right\|_{H^{1}}^2 \leq M+1, \quad \left\| \phi^{n} \right\|_{L^\infty} \leq C_1, \quad 0\leq n \leq T/\tau,
    \end{align} 
    and 
    \begin{align} 
        \frac{1}{2}\left\| e_{\phi}^{n} \right\|_{H^{1}}^2 +  \left| \tilde{q}^{n} \right|^2 + \frac{\tau}{5}\sum_{k=1}^{n}\left\| \nabla e_{\mu}^{k} \right\|^2 \leq C_2\tau^2,\quad 0\leq n \leq T/\tau,\label{Error estimation formula} 
    \end{align}
    Here in particular, $C_1$ depends  on $d, T, M, C_{0}$ and the norms of $\phi^{0},f$ on $\Omega$ but is  independent of $S$. 
\end{theorem}

The uniform $L^\infty$ bound of the numerical solution in (\ref{uniform bound}) can immediately lead to the following result that accomplishes Theorem \ref{Energy Law of iSAV-BE}.
\begin{corollary}\label{cor h2 bound}
  Denote $\displaystyle S_1=\frac12\max_{|\rho|\leq C_1}\left\{ f'\left(\rho \right) \right\}<\infty$. Let the assumptions and conditions in Theorem \ref{Theo eroor estimate} be satisfied,  then
    $$S_0=\frac{1}{2}\max_{0\leq n \leq T/\tau - 1}\;\max_{\mathbf{x}\in\Omega,\,\rho
        \in(0,1)}\left\{ f'\left( \left( 1-\rho \right)\phi^{n} + \rho\phi^{n+1} \right) \right\}\leq S_1,$$ and so by chosen  $S\in[S_1,\infty)$, the iSAV-BE scheme (\ref{iSAV-BE scheme}) with the original energy-stability (\ref{Original Energy Law of iSAV-BE}) is well-defined. 
\end{corollary}

The convergence for other gradient flows, e.g., the Allen-Cahn equation ($\mathcal{G}=I,\mathcal{L}=-\Delta$), can be established similarly, and we omit the results here for brevity. We will then prove Theorem \ref{Theo eroor estimate} in the next two subsections. 
\begin{remark}
    Compared to the traditional SAV-BE scheme (see Lemma 2.4 in \cite{Shen2018ConvergenceSAV}), the assumption for iSAV-BE in Theorem \ref{Theo eroor estimate} enhances the regularity requirement of $\phi^{0}$ from $H^{4}$ to $H^{8}$. This is only for the technical  analysis reason,  where an $H^{4}$-bound of the numerical solution is called in the proof (in (\ref{estimate A7})). It should be noted that such condition (even $\phi^{0}\in H^{4}$ for SAV-BE) is not necessary in practical computing. Relaxing the regularity requirement or finding the optimal condition is a future analytical task. 
\end{remark}

\subsection{Prior estimates of numerical solution}
To give the proof, we need to first establish some prior estimates for the numerical solution of iSAV-BE and we shall need a tool lemma which we directly quote from \cite{Shen2018ConvergenceSAV}. 

\begin{lemma}[\cite{Shen2018ConvergenceSAV}] \label{f estimate lem}
     Assume that $\left\| u \right\|_{H^{1}}^2\leq M+1$ and (\ref{f condition}) holds. Then for any $u\in H^{4}$, there exist some $0\leq \sigma\leq 1$ and a constant $C(d,M)>0$ dependent on $d,M$ such that 
    \begin{align*}
        \left\| \Delta f(u) \right\|^2\leq C(d,M)\left(1+\left\| \Delta^2 u \right\|^{2\sigma}\right).
    \end{align*} 
\end{lemma}

\begin{proposition}[$H^2$-boundedness]\label{H2 boundness}
    Assume that (\ref{f condition}) holds, $\phi^{0}\in H^{4}(\Omega)$ and $\left\| \phi^{n} \right\|_{H^{1}}^2\leq M+1$ for all $0\leq n \leq m$  with some $1\leq m < T/\tau$. Then, there exists a constant $\tau_{1}>0$ independent of $n,m,\tau$ such that for any $0<\tau\leq \tau_{1}$, 
    \begin{align}
        &\left\| \Delta\phi^{n+1} \right\|^2 + \frac{\tau}{2}\sum_{k=0}^{n+1}\left\| \Delta^2\phi^{k} \right\|^2 \leq C\left( d, T, M, C_{0},\phi^0 \right),  \quad \forall\,0\leq n\leq m,\label{H2 bound}
    \end{align}
   where $C_{0}=\inf_{\phi}\mathcal{F}[\phi]$ and the constant $C\left( d, T, M, C_{0},\phi^0 \right)>0$  depends  on $d, T, M, C_{0}$ and the norms of $\phi^{0}$ on $\Omega$.
\end{proposition}
\begin{proof}
    Plugging  (\ref{iSAV-BE scheme 2}) into (\ref{iSAV-BE scheme 1}), we obtain
    \begin{align}
        \frac{1}{\tau}\left( \phi^{n+1} - \phi^{n} \right) = -\Delta^2\phi^{n+1} + \frac{\tilde{r}^{n+1}}{\sqrt{\mathcal{F}[\phi^{n}]}}\Delta f(\phi^{n}) + S\Delta\left( \phi^{n+1} - \phi^{n} \right), \label{20240327-01}
    \end{align}
    and by the regularity of elliptical equations, we can have $\phi^{n}\in H^{4}$ for all $n\geq0$. Now we aim for the desired bound (\ref{H2 bound}).   
    Taking the $L^{2}$ inner product of (\ref{20240327-01}) with $\tau\Delta^{2}\phi^{n+1}$, we obtain
    \begin{align*}
        &\frac{1}{2}\left( \left\| \Delta\phi^{n+1} \right\|^2 - \left\| \Delta\phi^{n} \right\|^2 + \left\| \Delta\left( \phi^{n+1} - \phi^{n} \right) \right\|^2 \right) \nonumber\\
        =& \,\tau\left\langle -\Delta^2\phi^{n+1} + \frac{\tilde{r}^{n+1}}{\sqrt{\mathcal{F}[\phi^{n}]}}\Delta f(\phi^{n}) + S\Delta\left( \phi^{n+1} - \phi^{n} \right), \Delta^2\phi^{n+1}  \right\rangle \nonumber\\
        =& -\tau\left\| \Delta^2\phi^{n+1} \right\|^2 + \tau\frac{\tilde{r}^{n+1}}{\sqrt{\mathcal{F}[\phi^{n}]}}\left\langle \Delta f(\phi^{n}), \Delta^2\phi^{n+1} \right\rangle - \frac{\tau}{2} S\left( \left\| \nabla\Delta\phi^{n+1} \right\|^2 - \left\| \nabla\Delta\phi^{n} \right\|^2 + \left\| \nabla\Delta\left( \phi^{n+1}-\phi^{n} \right) \right\|^2 \right).
    \end{align*}
    From (\ref{20240111-03}), we can derive that 
    \begin{align*}
        \frac{\left( \tilde{r}^{n+1} \right)^2}{\mathcal{F}[\phi^{n}]} \leq \frac{ \frac{1}{2}\left\| \phi^{n} \right\|_{H^{1}}^2 + \left( r\left[\phi^{n}\right] \right)^2 }{\mathcal{F}[\phi^{n}]} \leq \frac{M}{2C_{0}} + 1,
    \end{align*}
    and then by using Lemma \ref{f estimate lem} and Young's inequality, we have
    \begin{align*}
        &\frac{1}{2}\left( \left\| \Delta\phi^{n+1} \right\|^2 - \left\| \Delta\phi^{n} \right\|^2 \right) + \tau\left\| \Delta^2\phi^{n+1} \right\|^2 + \frac{\tau}{2}S\left( \left\| \nabla\Delta\phi^{n+1} \right\|^2 - \left\| \nabla\Delta\phi^{n} \right\|^2 \right) \nonumber\\
        \leq& \tau\frac{\tilde{r}^{n+1}}{\sqrt{\mathcal{F}[\phi^{n}]}}\left\langle \Delta f(\phi^{n}), \Delta^2\phi^{n+1} \right\rangle \leq \tau C(M,C_{0})\left\| \Delta f(\phi^{n}) \right\|^2 + \frac{\tau}{2}\left\| \Delta^2\phi^{n+1} \right\|^2 \nonumber\\
        \leq& \tau C(d,M,C_{0}) \left( 1 + \left\| \Delta^2\phi^{n} \right\|^{2\sigma} \right) +  \frac{\tau}{2}\left\| \Delta^2\phi^{n+1} \right\|^2 \leq \tau C(d,M,C_{0})  + \frac{\tau}{4}\left\| \Delta^2\phi^{n} \right\|^2 + \frac{\tau}{2}\left\| \Delta^2\phi^{n+1} \right\|^2.
    \end{align*}
    Thus, we have for all $0\leq k\leq n$,
    $$\left\| \Delta\phi^{k+1} \right\|^2 - \left\| \Delta\phi^{k} \right\|^2  + \tau\left\| \Delta^2\phi^{k+1} \right\|^2+ \tau S\left( \left\| \nabla\Delta\phi^{k+1} \right\|^2 - \left\| \nabla\Delta\phi^{k} \right\|^2 \right) \leq \tau C(d,M,C_{0})+\frac{\tau}{2}
    \left\| \Delta^2\phi^{k} \right\|^2. $$
    By summing $k$ up from $0$ to $n$ and the condition $0<\tau\leq\tau_{1}$ with  $\tau_1 S\lesssim1$, we obtain  (\ref{H2 bound}).
\end{proof}

Unlike the  SAV-BE method, for convergence analysis of iSAV-BE, we will further need an  $H^4$-boundedness of the numerical solution.
\begin{proposition}[$H^4$-boundedness]\label{H4 boundness}
    Assume that (\ref{f condition}) holds, $\phi^{0}\in H^{8}(\Omega)$ and $\left\| \phi^{n} \right\|_{H^{1}}^2\leq M+1$ for all $0\leq n \leq m$  with some $1\leq m < T/\tau$. Then for any $0<\tau\leq\tau_{1}$ (the $\tau_1$ from Proposition \ref{H2 boundness}), it holds that
    \begin{align}
        &\left\| \Delta^2\phi^{n+1} \right\|^2 + \frac{\tau}{2}\sum_{k=0}^{n+1}\left\| \Delta^3\phi^{k} \right\|^2 \leq C\left( d, T, M, C_{0},\phi^0, f\right),  \quad \forall\,0\leq n\leq m,\label{H4 bound}
    \end{align}
    for some constant $C\left( d, T, M, C_{0},\phi^0 ,f\right)>0$ that depends only on $d, T, M, C_{0}$ and the norms of $\phi^{0},f$ on $\Omega$.
\end{proposition}
\begin{proof}
    Firstly, from (\ref{20240327-01}) we have $\phi^{n}\in H^{8}$ for all $n$ by the regularity of elliptical equations. Now by taking the $L^{2}$ inner product of (\ref{20240327-01}) with $\tau\Delta^{4}\phi^{n+1}$, we have
    \begin{align*}
        &\frac{1}{2}\left( \left\| \Delta^2\phi^{n+1} \right\|^2 - \left\| \Delta^2\phi^{n} \right\|^2 + \left\| \Delta^2\left( \phi^{n+1} - \phi^{n} \right) \right\|^2 \right) \nonumber\\
        =& -\tau\left\| \Delta^3\phi^{n+1} \right\|^2 + \tau\frac{\tilde{r}^{n+1}}{\sqrt{\mathcal{F}[\phi^{n}]}}\left\langle \Delta f(\phi^{n}), \Delta^4\phi^{n+1} \right\rangle \nonumber\\
        & - \frac{\tau}{2} S\left( \left\| \nabla\Delta^2\phi^{n+1} \right\|^2 - \left\| \nabla\Delta^2\phi^{n} \right\|^2 + \left\| \nabla\Delta^2\left( \phi^{n+1}-\phi^{n} \right) \right\|^2 \right). 
    \end{align*}
    Similarly as before, with $\left( \tilde{r}^{n+1} \right)^2/\mathcal{F}\left[\phi^{n}\right]\leq M/2C_{0}+1$ and Young's inequality, we have
    \begin{align}
        &\frac{1}{2}\left( \left\| \Delta^2\phi^{n+1} \right\|^2 - \left\| \Delta^2\phi^{n} \right\|^2 + \left\| \Delta^2\left( \phi^{n+1} - \phi^{n} \right) \right\|^2 \right) + \tau\left\| \Delta^3\phi^{n+1} \right\|^2 \nonumber\\
        & +\frac{\tau}{2}S\left( \left\| \nabla\Delta^2\phi^{n+1} \right\|^2 - \left\| \nabla\Delta^2\phi^{n} \right\|^2 + \left\| \nabla\Delta^2\left( \phi^{n+1} - \phi^{n} \right) \right\|^2 \right) \nonumber\\
        =& \tau \frac{\tilde{r}^{n+1}}{\sqrt{\mathcal{F}[\phi^{n}]}}\left\langle \Delta^2 f(\phi^{n}), \Delta^3\phi^{n+1} \right\rangle \leq \tau C(M,C_{0}) \left\| \Delta^2 f(\phi^{n}) \right\|^2 + \frac{\tau}{2}\left\| \Delta^3\phi^{n+1} \right\|^2. \label{20240227-01}
    \end{align}
    Next, we need to handle the term $\left\| \Delta^2 f(\phi^{n}) \right\|^2$. From the Proposition \ref{H2 boundness} which gives the uniform boundedness of $\phi^{n}$ in $H^2$ for $0\leq n\leq m$ and the Sobolev embedding $H^{2}\subseteq L^{\infty}$, we have 
    \begin{align*}
        |f(\phi^{n})|,\; |f'(\phi^{n})|,\; |f''(\phi^{n})|,\; |f^{(3)}(\phi^{n})|,\; |f^{(4)}(\phi^{n})| \leq C(d,T,M,C_{0},\phi^{0},f) \quad {\rm for\;all} \quad 0\leq n\leq m, 
    \end{align*}
    By directly calculating,
    \begin{align*}
        \Delta^2 f(\phi^{n}) =& f^{(4)}(\phi^{n})\left| \nabla\phi^{n} \right|^{4} + 6f^{(3)}(\phi^{n})\left| \nabla\phi^{n} \right|^{2}\Delta\phi^{n} + 4f''(\phi^{n})\nabla\Delta\phi^{n}\cdot\nabla\phi^{n} \nonumber\\
        &+ 3f''(\phi^{n})\left| \Delta\phi^{n} \right|^2 + f'(\phi^{n})\Delta^2\phi^{n},
    \end{align*}
    and we can derive that
    \begin{align}
        \left\| \Delta^2 f(\phi^{n}) \right\|^2 \leq& C\Big[\left\| f^{(4)}(\phi^{n}) \right\|_{L^{\infty}}^2\left\| \nabla\phi^{n} \right\|_{L^{8}}^8 + \left\| f^{(3)}(\phi^{n}) \right\|_{L^{\infty}}^2\left( \left\| \nabla\phi^{n} \right\|_{L^{8}}^8 + \left\| \Delta\phi^{n} \right\|_{L^{4}}^4 \right) \nonumber\\
        & + \left\| f''(\phi^{n}) \right\|_{L^{\infty}}^2\left( \left\| \nabla\Delta\phi^{n} \right\|_{L^{4}}^4 + \left\| \nabla\phi^{n} \right\|_{L^{4}}^4 + \left\| \Delta\phi^{n} \right\|_{L^{4}}^4 \right) + \left\| f'(\phi^{n}) \right\|_{L^{\infty}}^2\left\| \Delta^2\phi^{n} \right\|^2\Big] \nonumber\\
        \leq& C\left( d,T,M,C_{0},\phi^0,f \right)\left[ \left\| \nabla\phi^{n} \right\|_{L^{8}}^8 + \left\| \nabla\Delta\phi^{n} \right\|_{L^{4}}^4 + \left\| \Delta\phi^{n} \right\|_{L^{4}}^4 + \left\| \nabla\phi^{n} \right\|_{L^{4}}^4 + \left\| \Delta^2\phi^{n} \right\|^2 \right].\label{H4bound eq1}
    \end{align}
     We can estimate the terms on the right hand side of (\ref{H4bound eq1}) by the interpolation inequality  (see Chapter II, Section 2.1 in \cite{Temam2012infinite}) and Sobolev's embedding together with the result of Proposition \ref{H2 boundness} as follows:
    \begin{align*}
        &\left\| \nabla\phi^{n} \right\|_{L^{8}}^8 \lesssim \left\| \nabla\phi^{n} \right\|_{H^{3d/8}}^8 \lesssim \left\| \nabla\phi^{n} \right\|^{8-3d/5}\left\| \Delta^3\phi^{n} \right\|^{3d/5} \lesssim\left\| \Delta^3\phi^{n} \right\|^{3d/5}, \nonumber\\
        &\left\| \nabla\Delta\phi^{n} \right\|_{L^{4}}^4 \lesssim \left\| \nabla\Delta\phi^{n} \right\|_{H^{d/4}}^4 \lesssim \left\| \Delta\phi^{n} \right\|^{3-d/4}\left\| \Delta^3\phi^{n} \right\|^{1+d/4} \lesssim\left\| \Delta^3\phi^{n} \right\|^{1+d/4}, \nonumber\\
        &\left\| \Delta\phi^{n} \right\|_{L^{4}}^4 \lesssim \left\| \Delta\phi^{n} \right\|_{H^{d/4}}^4 \lesssim \left\| \Delta\phi^{n} \right\|^{4-d/4}\left\| \Delta^3\phi^{n} \right\|^{d/4} \lesssim \left\| \Delta^3\phi^{n} \right\|^{d/4}, \nonumber\\
        &\left\| \nabla\phi^{n} \right\|_{L^{4}}^4 \lesssim \left\| \nabla\phi^{n} \right\|_{H^{d/4}}^4 \lesssim \left\| \nabla\phi^{n} \right\|^{4-d/5}\left\| \Delta^3\phi^{n} \right\|^{d/5} \lesssim \left\| \Delta^3\phi^{n} \right\|^{d/5}, \nonumber\\
        &\left\| \Delta^2\phi^{n} \right\|^2 \lesssim \left\| \Delta\phi^{n} \right\|\left\| \Delta^3\phi^{n} \right\| \lesssim \left\| \Delta^3\phi^{n} \right\|.
    \end{align*}
By substituting the above estimates into (\ref{H4bound eq1}) and with some $\sigma\in(0,1)$,  we can write 
    $$\left\| \Delta^2 f(\phi^{n}) \right\|^2 \leq
    C(d,T,M,C_{0},\phi^0,f)\left\| \Delta^3\phi^{n} \right\|^{2\sigma}.$$
    Then by Young's inequality,  (\ref{20240227-01}) further becomes
    \begin{align*}
        &\frac{1}{2}\left( \left\| \Delta^2\phi^{n+1} \right\|^2 - \left\| \Delta^2\phi^{n} \right\|^2 + \left\| \Delta^2\left( \phi^{n+1} - \phi^{n} \right) \right\|^2 \right) + \tau\left\| \Delta^3\phi^{n+1} \right\|^2 \nonumber\\
        & +\frac{\tau}{2}S\left( \left\| \nabla\Delta^2\phi^{n+1} \right\|^2 - \left\| \nabla\Delta^2\phi^{n} \right\|^2 + \left\| \nabla\Delta^2\left( \phi^{n+1} - \phi^{n} \right) \right\|^2 \right) \nonumber\\
        \leq& \tau C(d,T,M,C_{0},\phi^0,f)\left\| \Delta^3\phi^{n} \right\|^{2\sigma} + \frac{\tau}{2}\left\| \Delta^3\phi^{n+1} \right\|^2 \leq \tau C(d,T,M,C_{0},\phi^0,f) + \frac{\tau}{4}\left\| \Delta^3\phi^{n} \right\|^2 + \frac{\tau}{2}\left\| \Delta^3\phi^{n+1} \right\|^2,
    \end{align*}
    and so for any $0\leq k\leq n$,
   \begin{align*} &\left\| \Delta^2\phi^{k+1} \right\|^2 - \left\| \Delta^2\phi^{k} \right\|^2+\tau\left\| \Delta^3\phi^{k+1} \right\|^2+ 
    \tau S\left( \left\| \nabla\Delta^2\phi^{k+1} \right\|^2 - \left\| \nabla\Delta^2\phi^{k} \right\|^2  \right)\\
    \leq& \tau C(d,T,M,C_{0},\phi^0,f) + \frac{\tau}{2}\left\| \Delta^3\phi^{k} \right\|^2.
    \end{align*}
    By summing the above up for $k$ from $0$ to $n$ and the condition $0<\tau<\tau_{1}$, we can obtain the estimate (\ref{H4 bound}). 
\end{proof}

\subsection{Proof for convergence theorem}
We now turn to the truncation errors of the scheme.  
Denoting $q^{n} := r[\phi^{n}] - r(t_{n})$ and subtracting (\ref{SAV Gradient flow}) from (\ref{iSAV-BE scheme}) at $t_{n+1}$, we can obtain the error equations
\begin{subequations}\label{Error equation}\begin{align}
    &\frac{1}{\tau}\left( e_{\phi}^{n+1} - e_{\phi}^{n} \right) = \Delta e_{\mu}^{n+1} - \eta_{1}^{n}, \label{Error equation 1}\\
    &e_{\mu}^{n+1} = -\Delta e_{\phi}^{n+1} + \frac{\tilde{q}^{n+1}}{\sqrt{\mathcal{F}[\phi^{n}]}}f(\phi^{n}) + S(e_{\phi}^{n+1}-e_{\phi}^{n}) + \eta_{2}^{n} + S\eta_{3}^{n}, \label{Error equation 2}\\
    &\tilde{q}^{n+1} - q^{n} = \frac{1}{2}\int_{\Omega}{ \frac{f(\phi^{n})}{\sqrt{\mathcal{F}[\phi^{n}]}}\left( e_{\phi}^{n+1}- e_{\phi}^{n} \right) }\,{\rm d}{\bf x} - R_{1}^{n} + R_{2}^{n} + R_{3}^{n}, \label{Error equation 3} 
\end{align}\end{subequations}
where the truncation errors are defined as
\begin{align*}
    &\eta_{1}^{n} := \frac{1}{\tau}\int_{t_{n}}^{t_{n+1}}{ \left( t_{n}-s \right)\phi_{tt}(s) }\,{\rm d}{s}, \quad \eta_{2}^n := r(t_{n+1})\left( \frac{f(\phi^{n})}{\sqrt{\mathcal{F}[\phi^{n}]}} - \frac{f(\phi(t_{n+1}))}{\sqrt{\mathcal{F}[\phi(t_{n+1})]}} \right), \\
    &\eta_{3}^n := \int_{t_{n}}^{t_{n+1}}{ \phi_{t}(s) }\,{\rm d}{s}, \quad R_{1}^{n} := \int_{t_{n}}^{t_{n+1}}{ (s-t_{n})r_{tt}(s) }\,{\rm d}{s}, \\
    &R_{2}^n := \frac{1}{2}\int_{\Omega}{ \left( \frac{f(\phi^{n})}{\sqrt{\mathcal{F}[\phi^{n}]}} - \frac{f(\phi(t_{n+1}))}{\sqrt{\mathcal{F}[\phi(t_{n+1})]}} \right)\left( \phi(t_{n+1})-\phi(t_{n}) \right) }\,{\rm d}{\bf x}, \\
    &R_{3}^{n} := \frac{1}{2}\int_{\Omega}{ \frac{f(\phi(t_{n+1}))}{\sqrt{\mathcal{F}[\phi(t_{n+1})]}}\int_{t_{n}}^{t_{n+1}}{ (s-t_{n})\phi_{tt}(s) }\,{\rm d}{s} }\,{\rm d}{\bf x}.
\end{align*}
Their estimates are stated in the following proposition. 
\begin{proposition}[Local truncation error]
    Assume that (\ref{f condition}) holds, $\phi^{0}\in H^{4}(\Omega)$, the exact solution $\phi$ satisfies (\ref{Assumption of exact solution 0}) and $\left\| \phi^{n} \right\|_{H^{1}}^2 \leq M+1$ for all $0\leq n \leq m< T/\tau$. Then for any $0<\tau\leq \tau_{1}$ (the $\tau_1$ from Proposition \ref{H2 boundness}) and all $0\leq n \leq m$, we have
    \begin{subequations}\label{estimate truncation error}\begin{align}
        &\left| R_{1}^{n} \right|^2 \leq \tau^3 C\left\| \nabla\phi \right\|_{L^{\infty}(0,T;L^{\infty})}^2\int_{t_{n}}^{t_{n+1}}{ \left( \left\| \phi_{t}(s) \right\|_{H^{1}}^2 + \left\| \phi_{tt}(s) \right\|_{H^{-1}}^2 \right) }\,{\rm d}{s}, \label{R1}\\
        &\left| R_{2}^{n} \right|^2 \leq \tau^2 C\left\| \phi_{t} \right\|_{L^{\infty}(0,T;H^{-1})}^2\left( \left\| e_{\phi}^{n} \right\|^{2} + \left\| \nabla e_{\phi}^{n} \right\|^{2} \right) + \tau^3 C\left\| \phi_{t} \right\|_{L^{\infty}(0,T;H^{-1})}^2\int_{t_{n}}^{t_{n+1}}{ \left\| \phi_{t}(s) \right\|_{H^{1}}^2 }\,{\rm d}{s}, \label{R2} \\
        &\left| R_{3}^{n} \right|^2 \leq \tau^3 C\left\| \nabla\phi \right\|_{L^{\infty}(0,T;L^{\infty})}^2\int_{t_{n}}^{t_{n+1}}{ \left\| \phi_{tt}(s) \right\|_{H^{-1}}^2 }\,{\rm d}{s}, \label{R3} \\
        &\left\| (-\Delta)^{-\frac{1}{2}}\eta_{1}^{n} \right\|^2\leq \tau C\int_{t_{n}}^{t_{n+1}}{ \left\| \phi_{tt}(s) \right\|_{H^{-1}}^2 }\,{\rm d}{s},  \\
        &\left\| \eta_{2}^{n} \right\|_{H^{1}}^2 \leq C\left( \left\| e_{\phi}^{n} \right\|^{2} + \left\| \nabla e_{\phi}^{n} \right\|^{2} \right) + \tau C\int_{t_{n}}^{t_{n+1}}{ \left\| \phi_{t}(s) \right\|_{H^{1}}^2 }\,{\rm d}{s}, \\
        &\left\| \eta_{3}^{n} \right\|_{H^{1}}^2 \leq \tau C\int_{t_{n}}^{t_{n+1}}{ \left\| \phi_{t}(s) \right\|_{H^{1}}^2 }\,{\rm d}{s}. 
    \end{align}\end{subequations}
    The constants $C>0$ in (\ref{estimate truncation error}) depend on the $d$, $T$, $M$, $C_{0}$, and the norms of $\phi^{0}$, $\phi$ and $f$ on $\Omega$.
\end{proposition}
\begin{proof}
    From Proposition \ref{H2 boundness} and the embedding $H^{2}\subseteq L^{\infty}$ for $d=1,2,3$, we find 
    \begin{equation}\label{inf norm}
        \|\phi^{n}\|_{L^\infty}\leq C\left( d, T, M, C_{0},\phi^0 \right)=:C_1,\quad 0\leq n\leq m.
    \end{equation}
    Together with  (\ref{Assumption of exact solution 0}) and the smoothness of $f$, we have 
    \begin{align}
        |f(\phi)|,\; |f'(\phi)|,\; |f''(\phi)|,\; |f(\phi^{n})|,\; |f'(\phi^{n})|,\; |f''(\phi^{n})|\leq C, \quad {\rm for\;all} \quad 0\leq n\leq m. \label{f bound}
    \end{align}
    with some $C>0$ uniform in $m,n,\tau,S$. For the $R_{1}^{n}$, by H\"older's inequality, we can obtain
    \begin{align}
        \left| R_{1}^{n} \right|^2 \leq \tau^2\left(\int_{t_{n}}^{t_{n+1}}{ |r_{tt}(s)| }\,{\rm d}{s} \right)^2 \leq \tau^3\int_{t_{n}}^{t_{n+1}}{ \left| r_{tt}(s) \right|^2 }\,{\rm d}{s}. \label{R1-0}
    \end{align}
    Calculating directly, we see
    \begin{align*}
        r_{tt} = -\frac{1}{4}\left( \mathcal{F}[\phi] \right)^{-\frac{3}{2}}\left( \int_{\Omega}{ f(\phi)\phi_{t} }\,{\rm d}{\bf x} \right)^2 + \frac{1}{2}\left( \mathcal{F}[\phi] \right)^{-\frac{1}{2}}\int_{\Omega}{ \left( f'(\phi)\phi_{t}^{2} + f(\phi)\phi_{tt} \right) }\,{\rm d}{\bf x}.
    \end{align*}
    With (\ref{Assumption of exact solution 0}) and (\ref{f bound}), we can have
    \begin{align*}
        \int_{t_{n}}^{t_{n+1}}{ |r_{tt}(s)|^2 }\,{\rm d}{s} &\leq C(C_{0},\phi,f)\int_{t_{n}}^{t_{n+1}}{ \left( \left\| \phi_{t} \right\|_{L^{4}}^4 + \left\| f'(\phi)\nabla\phi \right\|^2_{L^{\infty}}\left\| \phi_{tt} \right\|^2_{H^{-1}} \right) }\,{\rm d}{s} \nonumber\\
        &\leq C(C_{0},\phi,f)\left\| \nabla\phi \right\|_{L^{\infty}(0,T;L^{\infty})}^2\int_{t_{n}}^{t_{n+1}}{ \left( \left\| \phi_{t} \right\|_{H^{1}}^2  + \left\| \phi_{tt} \right\|_{H^{-1}}^2 \right) }\,{\rm d}{s},
    \end{align*}
    which combining with (\ref{R1-0}) lead to   (\ref{R1}). 
    
    For the $R_{2}^{n}$, firstly we write
    \begin{align*}
        B:=\frac{f(\phi^{n})}{\sqrt{\mathcal{F}[\phi^{n}]}} - \frac{f(\phi(t_{n+1}))}{\sqrt{\mathcal{F}[\phi(t_{n+1})]}} =& \frac{f(\phi^{n})}{\sqrt{\mathcal{F}[\phi^{n}]}} - \frac{f(\phi(t_{n}))}{\sqrt{\mathcal{F}[\phi(t_{n})]}} + \frac{f(\phi(t_{n}))}{\sqrt{\mathcal{F}[\phi(t_{n})]}} - \frac{f(\phi(t_{n+1}))}{\sqrt{\mathcal{F}[\phi(t_{n+1})]}} \\
        =&B_1+B_2+\frac{f(\phi(t_{n}))}{\sqrt{\mathcal{F}[\phi(t_{n})]}} - \frac{f(\phi(t_{n+1}))}{\sqrt{\mathcal{F}[\phi(t_{n+1})]}},
    \end{align*}
    with 
    \begin{align*}
        B_1= \frac{f(\phi^{n})-f(\phi(t_{n}))}{\sqrt{\mathcal{F}[\phi^{n}]}},\quad B_2= \frac{ f(\phi(t_{n}))\left( \mathcal{F}[\phi(t_{n})] - \mathcal{F}[\phi^{n}] \right) }{ \sqrt{\mathcal{F}[\phi^{n}]\mathcal{F}[\phi(t_{n})]}\left( \sqrt{\mathcal{F}[\phi(t_{n})]} + \sqrt{\mathcal{F}[\phi^{n}]} \right) }.
    \end{align*}
    Then, we can derive that
    \begin{align*}
        &\left\| B_{1} \right\| + \left\| B_{2} \right\| \leq C(d,T,M,C_{0},\phi^{0},\phi,f)\left\| e_{\phi}^{n} \right\|, \\
        &\left\| \nabla B_{1} \right\| \leq C(C_{0})\left\| \nabla f(\phi^{n})- \nabla f(\phi(t_{n})) \right\| \leq C(C_{0})\left\| \left( f'(\phi^{n})-f'(\phi(t_{n})) \right)\nabla \phi(t_{n}) \right\| + C(C_{0})\left\| f'(\phi^{n})\nabla e_{\phi}^{n} \right\| \nonumber\\
        &\hspace{1.15cm} \leq C(d,T,M,C_{0},\phi^{0},\phi,f)\left( \left\| \nabla\phi(t_{n}) e_{\phi}^{n} \right\| + \left\| \nabla e_{\phi}^{n} \right\| \right) \nonumber\\
        &\hspace{1.15cm} \leq C(d,T,M,C_{0},\phi^{0},\phi,f)\left( \left\| \nabla\phi(t_{n}) \right\|_{L^{3}}\left\| e_{\phi}^{n} \right\|_{L^{6}} + \left\| \nabla e_{\phi}^{n} \right\| \right) \nonumber\\
        &\hspace{1.15cm} \leq C(d,T,M,C_{0},\phi^{0},\phi,f)\left( \left\| \phi(t_{n}) \right\|_{W^{1,\infty}} \left\| e_{\phi}^{n} \right\|_{H^1} + \left\| \nabla e_{\phi}^{n} \right\| \right)  \nonumber\\
        &\hspace{1.15cm} \leq C(d,T,M,C_{0},\phi^{0},\phi,f)\left( \left\| e_{\phi}^{n} \right\| + \left\| \nabla e_{\phi}^{n} \right\| \right), \nonumber\\
        &\left\|\nabla B_{2}\right\| \leq C(d,T,M,C_{0},\phi^{0},\phi,f)\left\| \nabla f(\phi(t_{n})) \right\|\left\| e_{\phi}^{n} \right\|  \leq C(d,T,M,C_{0},\phi^{0},\phi,f)\left\| \nabla\phi(t_{n}) \right\|\left\| e_{\phi}^{n} \right\| \nonumber\\
        &\hspace{1.15cm} \leq C(d,T,M,C_{0},\phi^{0},\phi,f)\left\| e_{\phi}^{n} \right\|,
    \end{align*}
    which immediately give that
    \begin{align}
        \left\| \frac{f(\phi^{n})}{\sqrt{\mathcal{F}[\phi^{n}]}} - \frac{f(\phi(t_{n}))}{\sqrt{\mathcal{F}[\phi(t_{n})]}} \right\|_{H^{1}}^{2} \leq C(d,T,M,C_{0},\phi^{0},\phi,f)\left( \left\| e_{\phi}^{n} \right\|^{2} + \left\| \nabla e_{\phi}^{n} \right\|^{2} \right). \label{estimate f-f}
    \end{align}
    Similarly, we can derive that
    \begin{align}
        \left\| \frac{f(\phi(t_{n}))}{\sqrt{\mathcal{F}[\phi(t_{n})]}} - \frac{f(\phi(t_{n+1}))}{\sqrt{\mathcal{F}[\phi(t_{n+1})]}} \right\|_{H^1}^2 \leq& C(C_{0},\phi,f)\left(  \left\| \phi(t_{n}) - \phi(t_{n+1}) \right\|^2 + \left\| \nabla\left( \phi(t_{n}) - \phi(t_{n+1}) \right) \right\|^2  \right) \nonumber\\
        \leq& \tau C(C_{0},\phi,f)\int_{t_{n}}^{t_{n+1}}{ \left\| \phi_{t}(s) \right\|_{H^{1}}^2 }\,{\rm d}{s}. \label{estimate f-f 1}
    \end{align}
    From (\ref{estimate f-f}) and (\ref{estimate f-f 1}), we obtain 
    \begin{align*}
        \left\| B \right\|_{H^{1}}^2 \leq C(d,T,M,C_{0},\phi^{0},\phi,f)\left( \left\| e_{\phi}^{n} \right\|^2 + \left\| \nabla e_{\phi}^{n} \right\|^2 \right) + \tau C(C_{0},\phi,f)\int_{t_{n}}^{t_{n+1}}{ \left\| \phi_{t}(s) \right\|_{H^{1}}^2 }\,{\rm d}{s}.
    \end{align*}
    By substituting it into the definition of $R_{2}^{n}$, we find 
    \begin{align}
        \left| R_{2}^{n} \right|^2 \leq& \frac{1}{2}\left| \int_{\Omega}{ B\left( \phi(t_{n+1}) - \phi(t_{n}) \right) }\,{\rm d}{\bf x} \right|^2 \leq \tau^2C\left\| \phi_{t} \right\|_{L^{\infty}(0,T;H^{-1})}^2\left\| B \right\|_{H^{1}}^2 \nonumber\\
        \leq& \tau^2 C(d,T,M,C_{0},\phi^{0},\phi,f)\left\| \phi_{t} \right\|_{L^{\infty}(0,T;H^{-1})}^2\left( \left\| e_{\phi}^{n} \right\|^{2} + \left\| \nabla e_{\phi}^{n} \right\|^{2} \right) \nonumber\\
        & + \tau^3 C(C_{0},\phi,f)\left\| \phi_{t} \right\|_{L^{\infty}(0,T;H^{-1})}^2\int_{t_{n}}^{t_{n+1}}{ \left\| \phi_{t}(s) \right\|_{H^{1}}^2 }\,{\rm d}{s}.
    \end{align}
    
    For the $R_{3}^{n}$, we have 
    \begin{align*}
        \left| R_{3}^{n} \right|^2 \leq& \tau^2C(C_{0})\left( \int_{\Omega}{ \int_{t_{n}}^{t_{n+1}}{ |f(\phi(t_{n+1}))\phi_{tt}(s) |}\,{\rm d}{s} }\,{\rm d}{\bf x} \right)^2 \leq \tau^2C(C_{0})\left| \int_{t_{n}}^{t_{n+1}}{ \left\| f(\phi(t_{n+1})) \right\|_{H^{1}}\left\| \phi_{tt}(s) \right\|_{H^{-1}} }\,{\rm d}{s} \right|^2 \nonumber\\
        \leq& \tau^3C(C_{0})\left\| f(\phi(t_{n+1})) \right\|_{H^1}^2 \int_{t_{n}}^{t_{n+1}}{ \left\| \phi_{tt}(s) \right\|_{H^{-1}}^2 }\,{\rm d}{s} \nonumber\\
        \leq& \tau^3C(C_{0},\phi,f)\left\| \phi \right\|_{L^{\infty}(0,T;H^{1})}^2\int_{t_{n}}^{t_{n+1}}{ \left\| \phi_{tt}(s) \right\|_{H^{-1}}^2 }\,{\rm d}{s}.
    \end{align*}
    
    For the $\eta_{1}^{n}$ and $\eta_{3}^{n}$, we have
    \begin{align*}
        &\left\| (-\Delta)^{-\frac{1}{2}}\eta_{1}^{n} \right\|^2 \leq C \left\| \int_{t_{n}}^{t_{n+1}}{ \left|(-\Delta)^{-\frac{1}{2}}\phi_{tt}(s)\right| }\,{\rm d}{s} \right\|^2\leq \tau C\int_{t_{n}}^{t_{n+1}}{ \| \phi_{tt}(s) \|_{H^{-1}}^2 }\,{\rm d}{s},  \\
        &\| \eta_{3}^{n} \|_{H^{1}}^2 = \left\| \int_{t_{n}}^{t_{n+1}}{ \phi_{t}(s) }\,{\rm d}{s} \right\|_{H^{1}}^2 \leq \tau C\int_{t_{n}}^{t_{n+1}}{ \| \phi_{t}(s) \|_{H^{1}}^2 }\,{\rm d}{s}.
    \end{align*}
    
    For the $\eta_{2}^{n}$, we write it as 
    \begin{align*}
        \eta_{2}^{n} = r(t_{n+1})\left( \frac{f(\phi^{n})}{\sqrt{\mathcal{F}[\phi^{n}]}} - \frac{f(\phi(t_{n}))}{\sqrt{\mathcal{F}[\phi(t_{n})]}} \right) + r(t_{n+1})\left( \frac{f(\phi(t_{n}))}{\sqrt{\mathcal{F}[\phi(t_{n})]}}-\frac{f(\phi(t_{n+1}))}{\sqrt{\mathcal{F}[\phi(t_{n+1})]}} \right),
    \end{align*}
    and from   (\ref{estimate f-f}) and (\ref{estimate f-f 1}), we can obtain
    \begin{align*}
        \left\| \eta_{2}^{n} \right\|_{H^{1}}^2 \leq C(d,T,M,C_{0},\phi^{0},\phi,f)\left( \left\| e_{\phi}^{n} \right\|^{2} + \left\| \nabla e_{\phi}^{n} \right\|^{2} \right) + \tau C(C_{0},\phi,f)\int_{t_{n}}^{t_{n+1}}{ \left\| \phi_{t}(s) \right\|_{H^{1}}^2 }\,{\rm d}{s}.
    \end{align*}
    All the estimations in (\ref{estimate truncation error}) are now complete.
\end{proof}

Now with the boundedness results and the local error estimates, we are ready to prove Theorem \ref{Theo eroor estimate}.

\emph{Proof of Theorem \ref{Theo eroor estimate}.}
We proceed by using a mathematical induction on the bound  
\begin{equation}\label{induction}
\left\| \phi^{m} \right\|_{H^{1}}^2 \leq M+1,\quad 0\leq m\leq T/\tau,
\end{equation} which clearly holds for $m=0$.

{\bf Step 1.} Assuming that $\left\| \phi^{n} \right\|_{H^{1}}^2\leq M+1$ for all $0\leq n\leq m$ with some $0\leq m<T/\tau$, we need to show that it holds for $n=m+1$. Firstly, taking the $L^2$ inner product of (\ref{Error equation 1}) with $\tau e_{\phi}^{n+1}$, we have
\begin{align}
    \frac{1}{2}\left( \left\| e_{\phi}^{n+1} \right\|^2 - \left\| e_{\phi}^{n} \right\|^2 + \left\| e_{\phi}^{n+1} - e_{\phi}^{n} \right\|^2 \right) = -\tau\left\langle \nabla e_{\mu}^{n+1}, \nabla e_{\phi}^{n+1} \right\rangle - \tau\left\langle \eta_{1}^{n}, e_{\phi}^{n+1} \right\rangle =: \sum_{i=1}^{2}A_{i}. \label{err esti 0}
\end{align}
Secondly, taking the $L^2$ inner products of (\ref{Error equation 1}) and (\ref{Error equation 2}) with $\tau e_{\mu}^{n+1}$ and $e_{\phi}^{n+1}-e_{\phi}^{n}$, respectively, and multiplying (\ref{Error equation 3}) by $2\tilde{q}^{n+1}$, similarly as the derivation of  (\ref{oes eq})-(\ref{20240111-03}), we can   obtain the following equality:
\begin{align}
    &\frac{1}{2}\left( \left\| \nabla e_{\phi}^{n+1} \right\|^2 - \left\| \nabla e_{\phi}^{n} \right\|^2 + \left\| \nabla( e_{\phi}^{n+1}-e_{\phi}^{n} ) \right\|^2 \right) + S\left\| e_{\phi}^{n+1}-e_{\phi}^{n} \right\|^2 + \tau\left\| \nabla e_{\mu}^{n+1} \right\|^2 + 2\tilde{q}^{n+1}\left( \tilde{q}^{n+1}-q^{n} \right) \nonumber\\
    =& -2\tilde{q}^{n+1}R_{1}^{n} + 2\tilde{q}^{n+1}R_{2}^{n} + 2\tilde{q}^{n+1}R_{3}^{n} - \left\langle \eta_{2}^n, e_{\phi}^{n+1}-e_{\phi}^{n} \right\rangle - S\left\langle \eta_{3}^n, e_{\phi}^{n+1}-e_{\phi}^{n} \right\rangle  + \tau\left\langle \eta_{1}^n, e_{\mu}^{n+1} \right\rangle \nonumber\\
    =:&\sum_{i=3}^{8}A_{i}.  \label{err esti 1}
\end{align}
If we set $\tilde{q}^{0}=0$,  we can write
\begin{subequations}\label{til q}\begin{align}
    2\tilde{q}^{n+1}\left( \tilde{q}^{n+1}-q^{n} \right) =& \left| \tilde{q}^{n+1} \right|^2 - \left| \tilde{q}^{n} \right|^2 + \left| \tilde{q}^{n+1} - \tilde{q}^{n} \right|^2, \quad n=0, \\
    2\tilde{q}^{n+1}\left( \tilde{q}^{n+1}-q^{n} \right) =& 2\tilde{q}^{n+1}\left( \tilde{q}^{n+1}-\tilde{q}^{n} \right) + 2\tilde{q}^{n+1}\left( \tilde{q}^{n}-q^{n} \right) \nonumber\\
    =& \left| \tilde{q}^{n+1} \right|^2 - \left| \tilde{q}^{n} \right|^2 + \left| \tilde{q}^{n+1} - \tilde{q}^{n} \right|^2 + 2\tilde{q}^{n+1}\left( \tilde{r}^{n} - r[\phi^{n}] \right), \quad n\geq 1.
\end{align}\end{subequations}
Then by combining (\ref{err esti 0}), (\ref{err esti 1}),   (\ref{til q}) and defining $A_{9}:=2\tilde{q}^{n+1}\left( \tilde{r}^{n} - r[\phi^{n}] \right)$, we   obtain
\begin{align}
    \frac{1}{2}\left( \left\| e_{\phi}^{n+1} \right\|_{H^{1}}^2 - \left\| e_{\phi}^{n} \right\|_{H^{1}}^2 \right) + \tau\left\| \nabla e_{\mu}^{n+1} \right\|^2 + \left| \tilde{q}^{n+1} \right|^2 - \left| \tilde{q}^{n} \right|^2 \leq \sum_{i=1}^{9}\left| A_{i} \right|. \label{err esti 2}
\end{align}
Using (\ref{estimate truncation error}) and   (\ref{Assumption of exact solution 0}) under the condition $0\leq\tau\leq \tau_{1}$ (the $\tau_1$ from Proposition \ref{H2 boundness}), we have the following estimates for $\left| A_{1} \right|$ to $\left| A_{8} \right|$:
\begin{subequations} \label{estimate A 1 to 9} \begin{align}
    \left| A_{1} \right| =& \tau\left| \left\langle \nabla e_{\mu}^{n+1}, \nabla e_{\phi}^{n+1} \right\rangle \right| \leq {\frac{\tau}{5}}\left\| \nabla e_{\mu}^{n+1} \right\|^2 + \tau C\left\| \nabla e_{\phi}^{n+1} \right\|^2, \\
    \left| A_{2} \right| =& \tau\left| \left\langle \eta_{1}^{n}, e_{\phi}^{n+1} \right\rangle \right| = \tau\left| \left\langle (-\Delta)^{-\frac{1}{2}}\eta_{1}^{n}, \nabla e_{\phi}^{n+1} \right\rangle \right| \leq \tau C\left( \left\| (-\Delta)^{-\frac{1}{2}}\eta_{1}^{n} \right\|^2 + \left\| \nabla e_{\phi}^{n+1} \right\|^2 \right) \nonumber\\
    \leq& \tau C\left\| \nabla e_{\phi}^{n+1} \right\|^2 + \tau^3 C, \\
    \left| A_{3} \right| =& \left| 2\tilde{q}^{n+1}R_{1}^{n} \right| \leq \tau C\left( \tilde{q}^{n+1} \right)^2 + \frac{C}{\tau}\left| R_{1}^{n} \right|^2 \leq \tau C\left( \tilde{q}^{n+1} \right)^2 + \tau^3 C, \\
    \left| A_{4} \right| =& \left| 2\tilde{q}^{n+1}R_{2}^{n} \right| \leq \tau C\left( \tilde{q}^{n+1} \right)^2 + \frac{C}{\tau}\left| R_{2}^{n} \right|^2 \leq \tau C\left( \tilde{q}^{n+1} \right)^2 + \tau C\left( \left\| e_{\phi}^{n} \right\|^2 + \left\| \nabla e_{\phi}^{n} \right\|^2 \right) + \tau^3 C, \\
    \left| A_{5} \right| =& \left| 2\tilde{q}^{n+1}R_{3}^{n} \right| \leq \tau C\left( \tilde{q}^{n+1} \right)^2 + \frac{C}{\tau}\left| R_{3}^{n} \right|^2 \leq \tau C\left( \tilde{q}^{n+1} \right)^2 + \tau^3 C, \\
    \left| A_{6} \right| =& \left| \left\langle \eta_{2}^{n}, e_{\phi}^{n+1} - e_{\phi}^{n} \right\rangle \right| \leq \tau\left| \left\langle \eta_{2}^{n}, \Delta e_{\mu}^{n+1} - \eta_{1}^{n} \right\rangle \right| \leq \tau\left| \left\langle \nabla\eta_{2}^{n}, \nabla e_{\mu}^{n+1} \right\rangle \right| + \tau\left| \left\langle \eta_{2}^{n}, \eta_{1}^{n} \right\rangle \right| \nonumber\\
    \leq& { \frac{\tau}{5}}\left\| \nabla e_{\mu}^{n+1} \right\|^2 + \tau C\left\| \nabla\eta_{2}^{n} \right\|^2 + \tau C\int_{t_{n}}^{t_{n+1}}{ \left\| \phi_{tt}(s) \right\|_{H^{-1}}\left\| \eta_{2}^{n} \right\|_{H^{1}} }\,{\rm d}{s} \nonumber\\
    \leq& {\frac{\tau}{5}}\left\| \nabla e_{\mu}^{n+1} \right\|^2 + \tau C\left\| \eta_{2}^{n} \right\|_{H^{1}}^2 + \tau^2 C\int_{t_{n}}^{t_{n+1}}{ \left\| \phi_{tt}(s) \right\|_{H^{-1}}^2 }\,{\rm d}{s} \nonumber\\
    \leq& {\frac{\tau}{5}}\left\| \nabla e_{\mu}^{n+1} \right\|^2 + \tau C\left( \left\| e_{\phi}^{n} \right\|^2 + \left\| \nabla e_{\phi}^{n} \right\|^2 \right) + \tau^3 C, \\
    \left| A_{7} \right| =& S\left| \left\langle \eta_{3}^{n}, e_{\phi}^{n+1} - e_{\phi}^{n} \right\rangle \right| = S\tau\left| \left\langle \eta_{3}^{n}, \Delta e_{\mu}^{n+1} - \eta_{1}^{n} \right\rangle \right| \leq S\tau\left| \left\langle \nabla\eta_{3}^{n}, \nabla e_{\mu}^{n+1} \right\rangle \right| + S\tau\left| \left\langle \eta_{3}^{n}, \eta_{1}^{n} \right\rangle \right| \nonumber\\
    \leq& { \frac{\tau}{5}}\left\| \nabla e_{\mu}^{n+1} \right\|^2 + \tau S^2C\left\| \nabla\eta_{3}^{n} \right\|^2 + \tau SC \int_{t_{n}}^{t_{n+1}}{ \left\| \eta_{3}^{n} \right\|_{H^{1}}\left\| \phi_{tt}(s) \right\|_{H^{-1}} }\,{\rm d}{s}  \nonumber\\
    \leq& {\frac{\tau}{5}}\left\| \nabla e_{\mu}^{n+1} \right\|^2 + \tau S^2C\left\| \eta_{3}^{n} \right\|_{H^{1}}^2 + \tau^2 C\int_{t_{n}}^{t_{n+1}}{ \left\| \phi_{tt}(s) \right\|_{H^{-1}}^2 }\,{\rm d}{s} \leq {\frac{\tau}{5}}\left\| \nabla e_{\mu}^{n+1} \right\|^2 + \tau^3 (S^2+1)C, \\
    \left| A_{8} \right| =& \tau\left| \left\langle \eta_{1}^{n}, e_{\mu}^{n+1} \right\rangle \right| = \tau\left| \left\langle (-\Delta)^{-\frac{1}{2}}\eta_{1}^{n}, \nabla e_{\mu}^{n+1} \right\rangle \right| \leq {\frac{\tau}{5}}\left\| \nabla e_{\mu}^{n+1} \right\|^2 + \tau C\left\| (-\Delta)^{-\frac{1}{2}}\eta_{1}^{n} \right\|^2 \nonumber\\
    \leq& {\frac{\tau}{5}}\left\| \nabla e_{\mu}^{n+1} \right\|^2 + \tau^3 C.
\end{align}\end{subequations}

Next, we estimate $A_{9}$. According the definition of $r[\phi]$, it is essentially a functional of  $\phi$, which can be expanded into the following Taylor expansion:
\begin{align}
    r[\phi+\rho\psi] = & r[\phi] + \frac{{\rm d}}{{\rm d}\rho}r[\phi+\rho\psi]\Big|_{\rho=0}\cdot\rho  + \cdots + \frac{{\rm d}^k}{{\rm d}\rho^k}r[\phi+\rho\psi]\Big|_{\rho=0}\cdot\frac{\rho^k}{k!} + \frac{{\rm d}^{k+1}}{{\rm d}\rho^{k+1}}r[\phi+\rho\psi]\Big|_{\rho=\rho_{0}}\cdot\frac{\rho^{k+1}}{{(k+1)!}}, \label{Taylor expansion}
\end{align}
where $\rho\in \mathbb{R}$, $\rho_{0}\in(0,\rho)$, $\phi$ and $\psi$ are any functions of ${\bf x}$. We can calculate to get the first, the second and the third order functional derivatives of $r[\phi]$:
\begin{align*}
    &\frac{{\rm d}}{{\rm d}\rho}r[\phi+\rho \psi] = \frac{1}{2\sqrt{\mathcal{F}[\phi+\rho \psi]}}\int_{\Omega}{ f(\phi+\rho \psi) \psi }\,{\rm d}{\bf x}, \\
    &\frac{{\rm d}^2}{{\rm d}\rho^2}r[\phi+\rho \psi] = -\frac{1}{4}\big( \mathcal{F}[\phi+\rho \psi] \big)^{-\frac{3}{2}}\Big(\int_{\Omega}{ f(\phi+\rho \psi) \psi }\,{\rm d}{\bf x}\Big)^2  + \frac{1}{2}\big(\mathcal{F}[\phi+\rho \psi]\big)^{-\frac{1}{2}}\int_{\Omega}{ f'(\phi+\rho \psi) \psi^2 }\,{\rm d}{\bf x}, \\
    &\frac{{\rm d}^3}{{\rm d}\rho^3}r[\phi+\rho \psi] = \frac{3}{8}\big( \mathcal{F}[\phi+\rho \psi] \big)^{-\frac{5}{2}}\Big(\int_{\Omega}{ f(\phi+\rho \psi) \psi }\,{\rm d}{\bf x}\Big)^3 + \frac{1}{2}\big(\mathcal{F}[\phi+\rho \psi]\big)^{-\frac{1}{2}}\int_{\Omega}{ f''(\phi+\rho \psi) \psi^3 }\,{\rm d}{\bf x}  \\
    &\hspace{2.7cm} - \frac{3}{4}\big( \mathcal{F}[\phi+\rho \psi] \big)^{-\frac{3}{2}}\Big(\int_{\Omega}{ f(\phi+\rho \psi) \psi }\,{\rm d}{\bf x}\Big)\Big(\int_{\Omega}{ f'(\phi+\rho \psi) \psi^2 }\,{\rm d}{\bf x}\Big).
\end{align*}
By setting $\phi=\phi^{n-1}$, $\psi=\phi^{n} - \phi^{n-1}$, there are
\begin{align*}
    r[\phi^{n}] =& r[\phi^{n-1}+(\phi^{n}-\phi^{n-1})] \\
    =& r[\phi^{n-1}] + \frac{{\rm d}}{{\rm d}\rho} r[\phi^{n-1}+\rho (\phi^{n}-\phi^{n-1})]\Big|_{\rho=0} + \frac{1}{2}\frac{{\rm d}^2}{{\rm d}\rho^2} r[\phi^{n-1}+\rho (\phi^{n}-\phi^{n-1})]\Big|_{\rho=\rho_{0}} \\
    =& r[\phi^{n-1}] + \frac{1}{2\sqrt{\mathcal{F}[\phi^{n-1}]}}\int_{\Omega}f(\phi^{n-1})(\phi^{n}-\phi^{n-1})\,{\rm d}{\bf x} -\frac{1}{8}\left( \mathcal{F}[\xi_{0}^{n}] \right)^{-\frac{3}{2}}\left( \int_{\Omega}{ f(\xi_{0}^{n})(\phi^{n}-\phi^{n-1}) }\,{\rm d}{\bf x} \right)^2  \\
    & + \frac{1}{4}\left( \mathcal{F}[\xi_{0}^{n}] \right)^{-\frac{1}{2}} \int_{\Omega}{ f'(\xi_{0}^{n})(\phi^{n}-\phi^{n-1})^2 }\,{\rm d}{\bf x},
\end{align*}
with $\xi_{0}^{n}:=\rho_{0} \phi^{n} + (1-\rho_{0})\phi^{n-1}$ for some $\rho_{0}\in(0,1)$. By moving $r[\phi^{n-1}]$ to the left-hand side and subtracting above equality from  (\ref{iSAV-BE scheme 3}) (with $n$ replaced by $n-1$), we obtain
\begin{align*}
    \tilde{r}^{n} - r[\phi^{n}] =& \frac{1}{8}\left( \mathcal{F}[\xi_{0}^{n}] \right)^{-\frac{3}{2}}\left( \int_{\Omega}{ f(\xi_{0}^{n})(\phi^{n}-\phi^{n-1}) }\,{\rm d}{\bf x} \right)^2 - \frac{1}{4}\left( \mathcal{F}[\xi_{0}^{n}] \right)^{-\frac{1}{2}} \int_{\Omega}{ f'(\xi_{0}^{n})(\phi^{n}-\phi^{n-1})^2 }\,{\rm d}{\bf x}.
\end{align*}
Combining (\ref{iSAV-BE scheme 1}) and (\ref{iSAV-BE scheme 2}), we can obtain 
\begin{align}
    \left| \tilde{r}^{n} - r[\phi^{n}] \right|^2 \leq& C\left\| \phi^{n} - \phi^{n-1} \right\|^4 \leq \tau^4 C \left\| \Delta \mu^{n} \right\|^4 \nonumber\\
    \leq& \tau^{4} C\left( S^{4}\left\| \Delta\phi^{n-1} \right\|^4 + S^{4}\left\| \Delta\phi^{n} \right\|^4 + \left\| \Delta^2\phi^{n} \right\|^4 + \left\| \Delta f(\phi^{n-1}) \right\|^4 \right) \nonumber\\
    \leq& \tau^4 C\left( 1 + S^{4}\left\| \Delta\phi^{n-1} \right\|^4 + S^{4}\left\| \Delta\phi^{n} \right\|^4 + \left\| \Delta^2\phi^{n} \right\|^4 + \left\| \Delta^2\phi^{n} \right\|^{4\sigma} \right) \leq \tau^4(S^4+1)C, \label{20240413}
\end{align}
by using the boundedness of $\left\| \Delta\phi^{n} \right\|$ and $\left\| \Delta^2\phi^{n} \right\|$ in Propositions \ref{H2 boundness} and \ref{H4 boundness}. Thus, we have
\begin{align}
    \left| A_{9} \right| = \left| 2\tilde{q}^{n+1}\left( \tilde{r}^{n}-r[\phi^{n}] \right) \right| \leq \tau C\left( \tilde{q}^{n+1} \right)^2 + \frac{C}{\tau}\left| \tilde{r}^{n}-r[\phi^{n}] \right|^2 \leq \tau C\left( \tilde{q}^{n+1} \right)^2 + \tau^3(S^4+1)C. \label{estimate A7}
\end{align}

Substituting (\ref{estimate A 1 to 9}) and (\ref{estimate A7}) into (\ref{err esti 2}), we find
\begin{align}
    &\frac{1}{2}\left( \left\| e_{\phi}^{n+1} \right\|_{H^{1}}^2 - \left\| e_{\phi}^{n} \right\|_{H^{1}}^2 \right)+ {\frac{\tau}{5}}\left\| \nabla e_{\mu}^{n+1} \right\|^2 + \left| \tilde{q}^{n+1} \right|^2 - \left| \tilde{q}^{n} \right|^2 \nonumber\\
    \leq& \tau C\left[ \left( \tilde{q}^{n+1} \right)^2 + \left\| e_{\phi}^{n} \right\|^2 + \left\| \nabla e_{\phi}^{n} \right\|^2  \right] + \tau^3 (S^4+1)C,\quad 0\leq n\leq m. \nonumber
\end{align}
All the constants $C$ appearing in (\ref{estimate A 1 to 9}), (\ref{20240413}) and (\ref{estimate A7})  depend only on  $d,\, T,\, C_{0}$ and the norms of $\phi^{0},  \phi, f$ on $\Omega$. 
Applying the discrete Gronwall inequality to the above, we can immediately derive  
$$ 
        \frac{1}{2}\left\| e_{\phi}^{n} \right\|_{H^{1}}^2 +  \left| \tilde{q}^{n} \right|^2 + \frac{\tau}{5}\sum_{k=1}^{n}\left\| \nabla e_{\mu}^{k} \right\|^2 \leq C_2\tau^2,\quad 0\leq n \leq m+1,
        $$
with a constant $C_2>0$ depends  on  $S, d,\, T,\, C_{0}$ and the norms of $\phi^{0},  \phi, f$ on $\Omega$. Thus, the error estimate result (\ref{Error estimation formula}) is deduced once the $H^1$ bound (\ref{induction}) holds, and so is the $L^\infty$ bound (\ref{inf norm}). 

{\bf Step 2.} By the triangle inequality and the deduced error bound above, we have
\begin{align*}
\left\| \phi^{m+1} \right\|_{H^{1}} \leq \left\| \phi\left( t_{m+1} \right) \right\|_{H^{1}} + \left\| e_{\phi}^{m+1} \right\|_{H^{1}} \leq \sqrt{M} + \tau \sqrt{2C_2}.
\end{align*}
Therefore, there exists some $\tau_{2}>0$ that depends only on $C_2,M$ such that when  $0<\tau\leq\tau_{2}$, we have $\left\| \phi^{m+1} \right\|_{H^{1}}^2\leq M+1$. 
The induction on (\ref{induction}) is then done, and the proof of the theorem is completed by setting $\tau_{0}=\min\left\{ \tau_{1},\tau_{2} \right\}$.

\qed 

\begin{remark}\label{rk tau0}  We remark that $\tau\leq\tau_0$ in Theorem \ref{Theo eroor estimate} is only a technical condition for rigorous analysis.  It is particular to provide the uniform $L^\infty$ bound and so to accomplish the original energy stability property. From the process of the proof, one may see that $\tau_0$ depends on $S$. The practical computations based on our experiments (Figure \ref{fig on tau0}) never feel the essential existence of such restriction on the time step.
\end{remark}

\section{Numerical experiments}\label{sec:result}
In this section, we apply the proposed iSAV schemes to solve some gradient flows. By numerical results, we   demonstrate the accuracy and improved stability of iSAV compared  with SAV. For comparison purpose, all the considered SAV schemes here are the original ones, i.e., (\ref{SAV-BE scheme}) and (\ref{SAV-BDF scheme}), which do \textbf{not} involve any stabilization techniques as mentioned in Remark \ref{stable tech}. In all of our examples, we impose the periodic boundary conditions for (\ref{gradient flow}) and utilize the Fourier pseudo-spectral method \cite{Shen2011} for spatial discretizations. 
For the BDF schemes that involve three-time levels, we use the iSAV-BE scheme to initiate for the first time level, which does not affect the overall second-order accuracy.

For simplicity, we only consider $\mathcal{L}=-\Delta$ and the $H^{-\alpha}$ flow for (\ref{gradient flow}) in two space dimensions, i.e., $d=2,\bx=(x,y)$, $\mathcal{G}=-\gamma(-\Delta)^{\alpha}$ with a parameter $\gamma>0$ controlling  the rate of energy dissipation \cite{Shen2019NewClass}. In the case of $\alpha=0$, the gradient flow is the standard Allen–Cahn equation. For $\alpha=1$, it becomes the standard Cahn-Hilliard equation.  For the nonlinearity/potential $F(\cdot)$, we will consider two typical kinds from applications for numerical experiments  in the following. 

\subsection{Double-well potential}
We first consider the double-well type potential in (\ref{gradient flow}), i.e.,
\begin{align*}
	F(\phi) := \frac{1}{4\varepsilon^2}(\phi^{2}-1)^{2},
\end{align*}  
which is the most classical case for (\ref{gradient flow}) in applications. 
\begin{example}[Accuracy of iSAV schemes] \label{Ex 1}
    We consider this example to test the accuracy of the proposed iSAV schemes, i.e., iSAV-BE (\ref{iSAV-BE scheme}) and iSAV-BDF (\ref{iSAV-BDF scheme}). Select the computational domain $\Omega=[0,2\pi]^{2}$ and the parameters
    $\varepsilon = 1, \gamma = 0.1 $ and $S = 6$. 
    The initial data for the flow (\ref{gradient flow}) is chosen as
    \begin{align*}
        \phi(x,y,0) = 1 + 0.5\,{\rm sin}(x)\,{\rm sin}(y). 
    \end{align*}
\end{example} 
We measure the error under $H^1$ norm (as analyzed) at $t=T=0.5$. The reference solution is computed using the SAV-BDF scheme (\ref{SAV-BDF scheme}). Tables \ref{tab ex1 space convergence} and \ref{tab ex1 time convergence} display the spatial and the temporal errors of iSAV schemes, respectively. From the numerical results, we can clearly observe that: 1) both iSAV schemes achieve the spectral accuracy in the spatial direction; 2) iSAV-BE and iSAV-BDF achieve the first order and the second order accuarcy in the temporal direction, respectively, which is their optimal convergence order. In particular, the result of iSAV-BE confirms Theorem \ref{Theo eroor estimate}.

\begin{table}[!ht]
    \centering
    \caption{(Spatial error) $\| e_{\phi}^{N} \|_{H^{1}}$ of iSAV schemes at $T=0.5$, with $\tau=10^{-5}$ fixed and different spatial grid.}
    \label{tab ex1 space convergence}
    \begin{tabular}{clcccccc}
        \hline
        \multicolumn{2}{c}{Spatial grid} & $4\times4$ & $8\times8$ & $12\times12$ & $16\times16$ & $20\times20$ \\
        \hline
        \multirow{2}{*}{iSAV-BE} & $\alpha = 0$ & 1.38E-03 & 1.14E-06 & 1.44E-08 & 8.78E-11 & 2.91E-13 \\
        & $\alpha = 1$ & 5.40E-03 & 9.58E-05 & 1.80E-07 & 1.13E-09 & 2.40E-12 \\
        \hline
        \multirow{2}{*}{iSAV-BDF} & $\alpha = 0$ & 1.38E-03 & 1.14E-06 & 1.44E-08 & 8.78E-11 & 3.08E-13 \\
        & $\alpha = 1$ & 5.40E-03 & 9.58E-05 & 1.80E-07 & 1.13E-09 & 2.40E-12 \\
        \hline
    \end{tabular}
\end{table}

\begin{table}[!ht]
    \centering
    \caption{ (Temporal error) $\| e_{\phi}^{N} \|_{H^{1}}$ of iSAV schemes at $T=0.5$, with spatial grid $240\times 240$ fixed and different $\tau=T/N$.} \label{tab ex1 time convergence}
    \begin{tabular}{clccccc}
        \hline
        \multicolumn{2}{c}{$N$}   & $10$  & $20$   & $40$  & $80$    & $160$     \\ \hline
        \multirow{4}{*}{iSAV-BE} & $\alpha = 0$   & 1.57E-02 & 7.96E-03 & 4.01E-03 & 2.01E-03 & 1.01E-03   \\
        & order & \multicolumn{1}{c}{--} & \multicolumn{1}{c}{0.98} & \multicolumn{1}{c}{0.99} & \multicolumn{1}{c}{0.99} & \multicolumn{1}{c}{1.00} \\ \cline{2-7} 
        & $\alpha = 1$   & 4.90E-02 & 2.52E-02 & 1.28E-02 & 6.42E-03 & 3.22E-03   \\
        & order & \multicolumn{1}{c}{--} & \multicolumn{1}{c}{0.96} & \multicolumn{1}{c}{0.98} & \multicolumn{1}{c}{0.99} & \multicolumn{1}{c}{0.99} \\ \hline
        \multirow{4}{*}{iSAV-BDF} & $\alpha = 0$   & 2.15E-03 & 5.33E-04 & 1.33E-04 & 3.31E-05 & 8.27E-06   \\
        & order & \multicolumn{1}{c}{--} & \multicolumn{1}{c}{2.01} & \multicolumn{1}{c}{2.01} & \multicolumn{1}{c}{2.00} & \multicolumn{1}{c}{2.00} \\ \cline{2-7} 
        & $\alpha = 1$   & 6.56E-03 & 1.59E-03 & 3.91E-04 & 9.71E-05 & 2.42E-05   \\
        & order & \multicolumn{1}{c}{--} & \multicolumn{1}{c}{2.05} & \multicolumn{1}{c}{2.02} & \multicolumn{1}{c}{2.01} & \multicolumn{1}{c}{2.00} \\ \hline
    \end{tabular}
\end{table}

\begin{figure}[!ht]
    \centering
    \includegraphics[width=5.2cm, height=5cm]{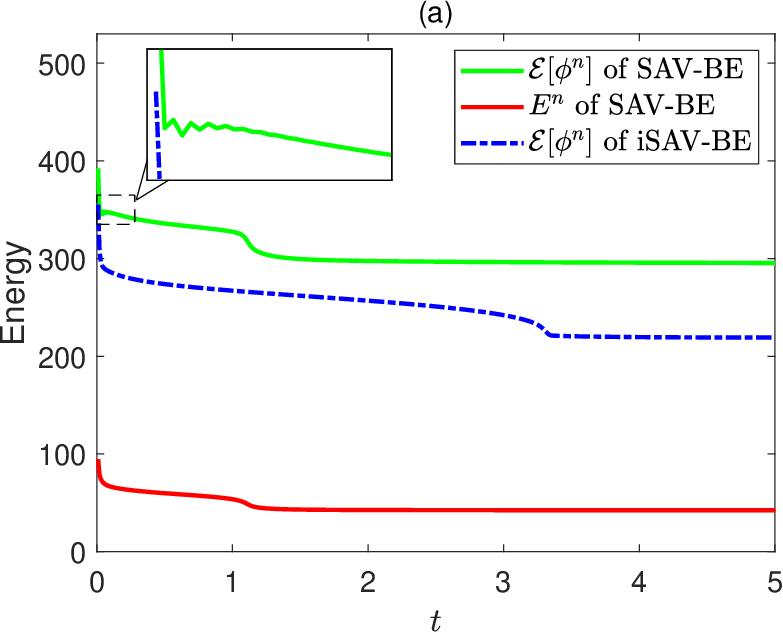}
    \includegraphics[width=5.2cm, height=5cm]{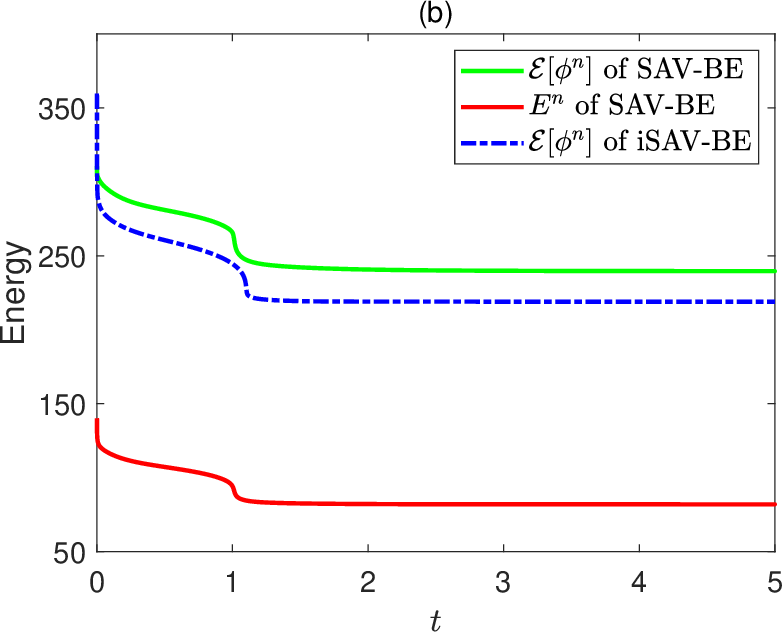}
    \includegraphics[width=5.2cm, height=5cm]{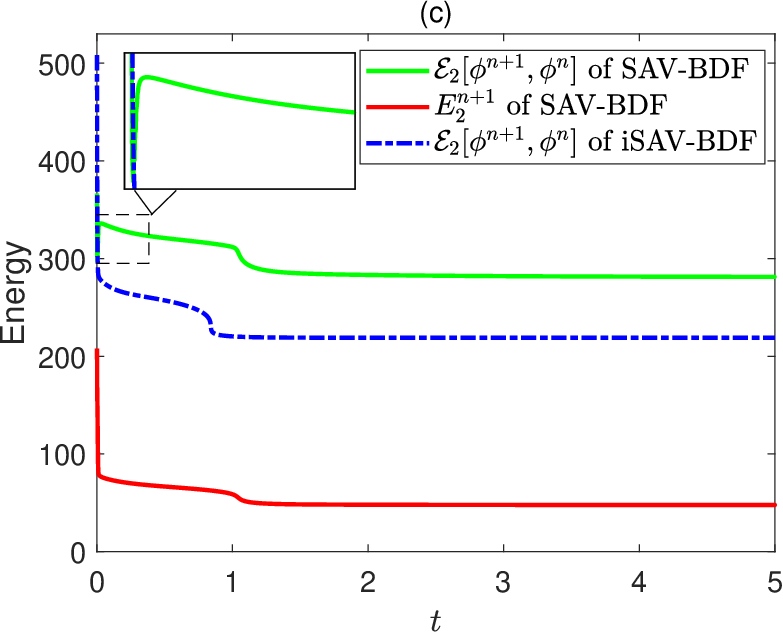}
    \caption{(Example \ref{Ex 2}): Time evolutions of the energy in SAV and iSAV. (a): BE schemes with $\tau = 0.01$; (b): BE schemes with $\tau = 0.001$; (c): BDF schemes with $\tau = 0.001$.} \label{fig ex2 energy} \vspace{3pt}
\end{figure}

\begin{figure}[!ht]
    \centering
    \includegraphics[width=5.2cm, height=5cm]{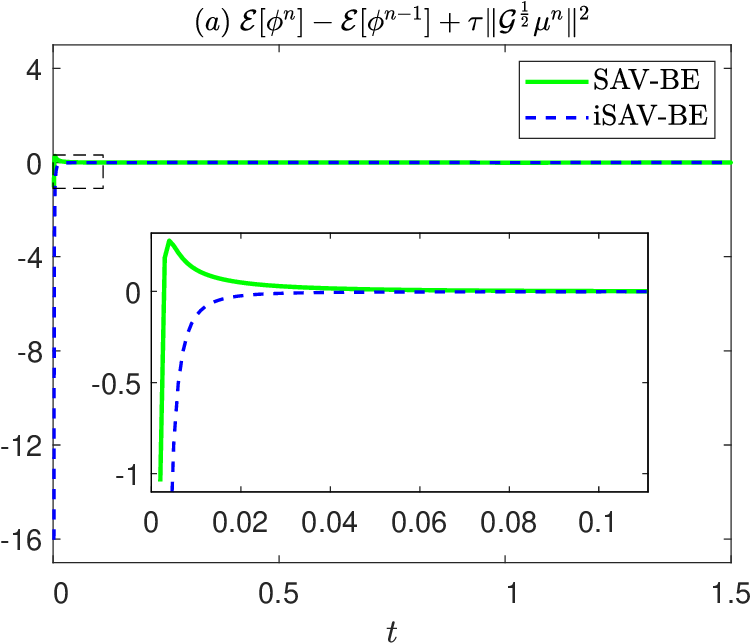}
    \includegraphics[width=5.2cm, height=5cm]{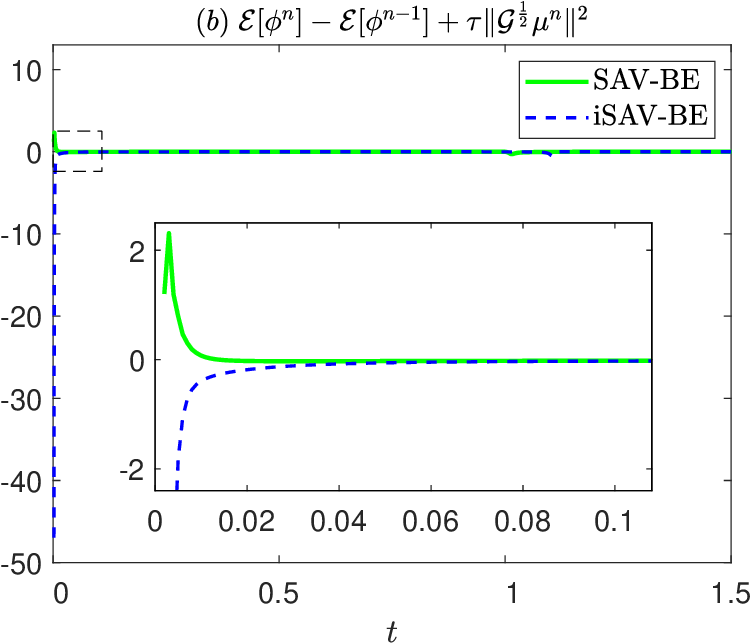}
    \includegraphics[width=5.2cm, height=5cm]{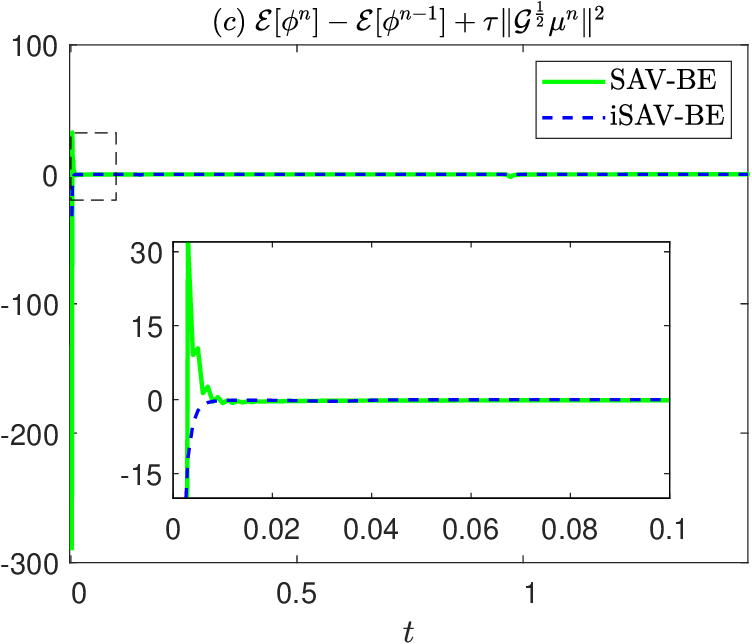}
    \caption{(Example \ref{Ex 2}): Time evolution of $\mathcal{E}[\phi^{n}]-\mathcal{E}[\phi^{n-1}]+\tau\| \mathcal{G}^{\frac{1}{2}}\mu^{n} \|^2$ in SAV-BE and iSAV-BE with $\tau = 0.001$. (a) $\varepsilon=0.1$; (b) $\varepsilon=0.04$; (c) $\varepsilon=0.01$.} \label{fig ex2 E-E BE} \vspace{3pt}
\end{figure}

\begin{figure}[!ht]
    \centering
    \includegraphics[width=5.2cm, height=5cm]{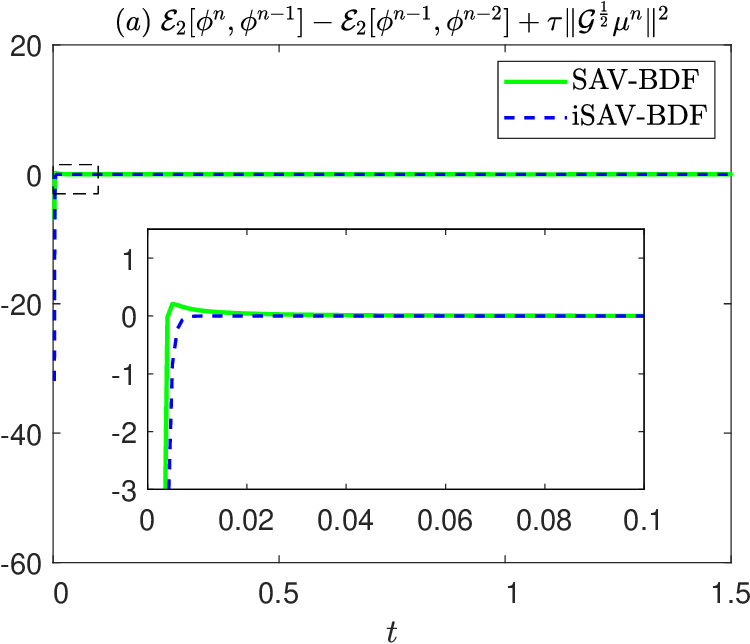}
    \includegraphics[width=5.2cm, height=5cm]{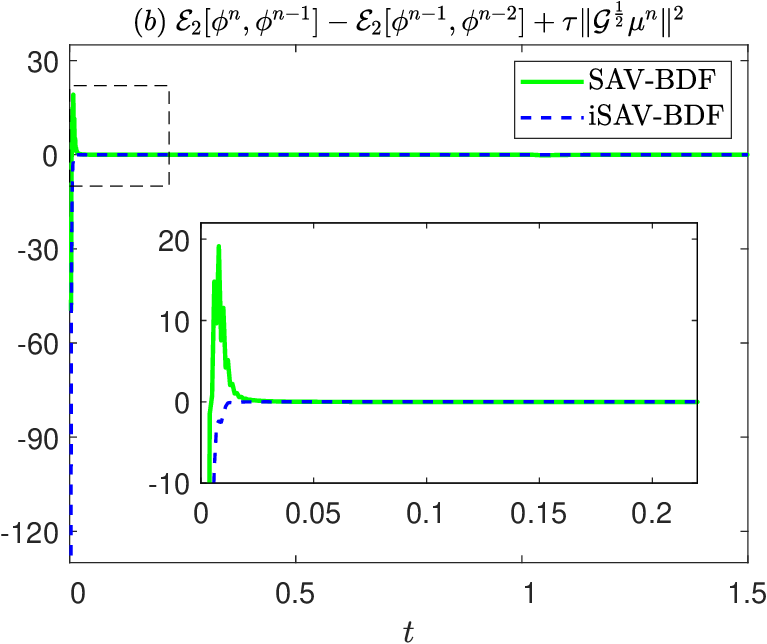}
    \includegraphics[width=5.2cm, height=5cm]{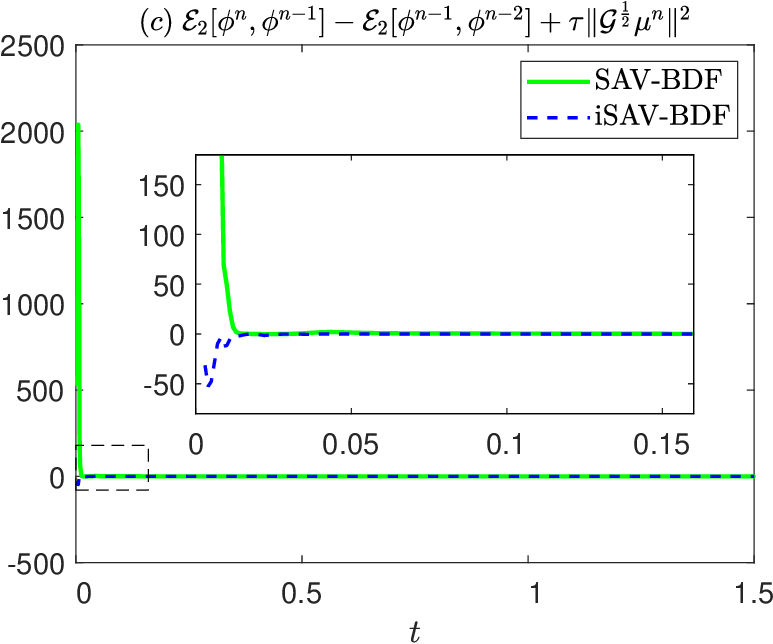}
    \caption{(Example \ref{Ex 2}): Time evolution of $\mathcal{E}_{2}[\phi^{n},\phi^{n-1}]-\mathcal{E}_{2}[\phi^{n-1},\phi^{n-2}]+\tau\| \mathcal{G}^{\frac{1}{2}}\mu^{n} \|^2$ of SAV-BDF and iSAV-BDF with $\tau = 0.001$. (a) $\varepsilon=0.1$; (b) $\varepsilon=0.04$; (c) $\varepsilon=0.01$.} \label{fig ex2 E-E BDF} \vspace{3pt}
\end{figure}

\begin{example}[Energy stability and comparison] \label{Ex 2}
    Now we take an example of the Cahn-Hilliard equation to test the energy stability of iSAV and make some comparisons with SAV. Set the computational domain $\Omega=[0,6.4]\times[0,6.4]$ with a $256\times 256$ spatial grid and the  parameters as
    $\gamma = 0.01, \, \alpha = 1.$
     The stabilizing parameter is chosen as $S = 3/\varepsilon^2$ for iSAV. The initial data of (\ref{gradient flow}) is given as 
    \begin{align*}
        \phi(x,y,0) = \left\{
        \begin{aligned}
            &\;1,\qquad (x,y)\in\Omega_{0}, \\
            &-1,\quad \, (x,y)\in\Omega/\Omega_{0}, 
        \end{aligned}
        \right.
    \end{align*}
    where $ \Omega_{0} = \big\{ (x,y)\in\Omega\, :\, (|x-3.2|\leq 1\cap |y-3.2|\leq 1) \cup (|x-5|\leq 0.36\cap |y-5|\leq 0.36) \big\} $ is a set of two squares in $\Omega$. Moreover, in order to maintain a positive lower bound for $r[\phi]$, an additional constant $ 1 $ has been added to the double-well   potential $F(\phi)$. We solve the problem by the presented SAV and iSAV schemes under the same step size.
\end{example}

 Firstly, we test the time  evolution of the `energy' under the numerical schemes, i.e., $\mathcal{E}[\phi^n],E^n$ for BE schemes and $\mathcal{E}_2[\phi^{n+1},\phi^n],E_2^{n+1}$ for BDF schemes. We fix $\varepsilon = 0.04$ here and take two different $\tau$. Figure \ref{fig ex2 energy} depicts the results for each scheme. We can 
observe that when time step of the SAV schemes is not small enough ((a)\&(c) in Figure \ref{fig ex2 energy}), the original energy ($\mathcal{E}$ or $\mathcal{E}_2$) is not monotonically decreasing where oscillations or increments could occur. In contrast, the original energy dissipation is strictly maintained by the iSAV schemes. 
This justifies Theorem \ref{Energy Law of iSAV-BE}.
Moreover, at the end of the evolution, the reached steady states by SAV possess higher energy level than iSAV.

\begin{figure}[!ht]
    \centering
    \includegraphics[width=7cm, height=5cm]{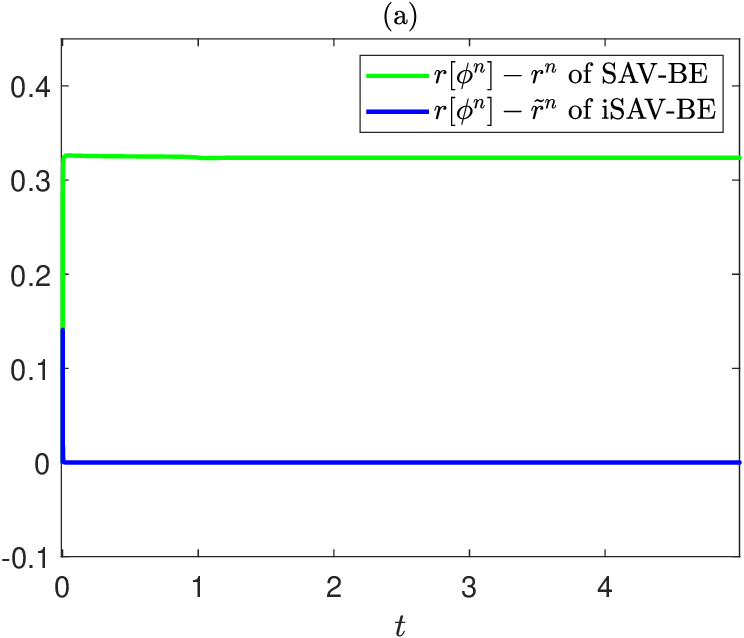}
    \includegraphics[width=7cm, height=5cm]{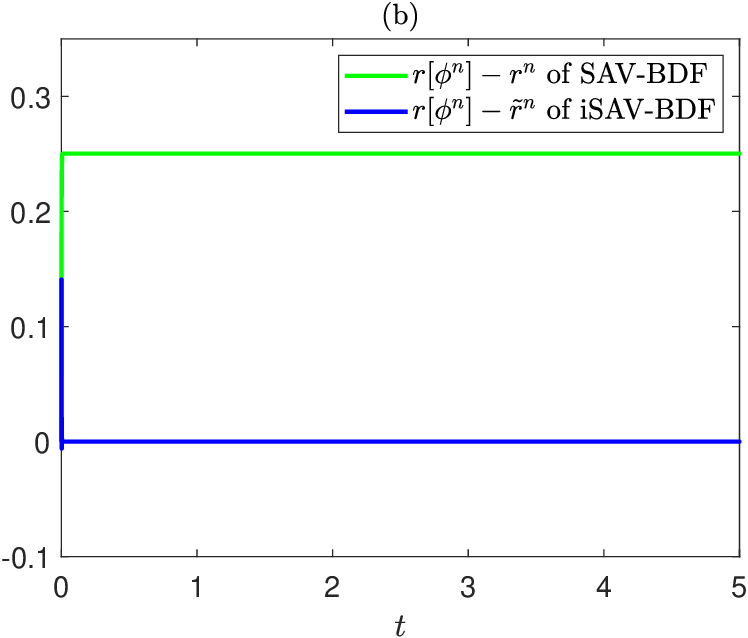}
    \caption{(Example \ref{Ex 2}): Difference between $r[\phi^{n}]$ and $r^{n}$ or $\tilde{r}^{n}$ in SAV or iSAV. (a): BE schemes with $\tau = 0.001$; (b): BDF schemes with $\tau = 0.001$.} \label{fig ex2 rn diff} \vspace{3pt}
\end{figure}
\begin{figure}[!ht]
    \centering
    \includegraphics[width=3.6cm, height=3.6cm]{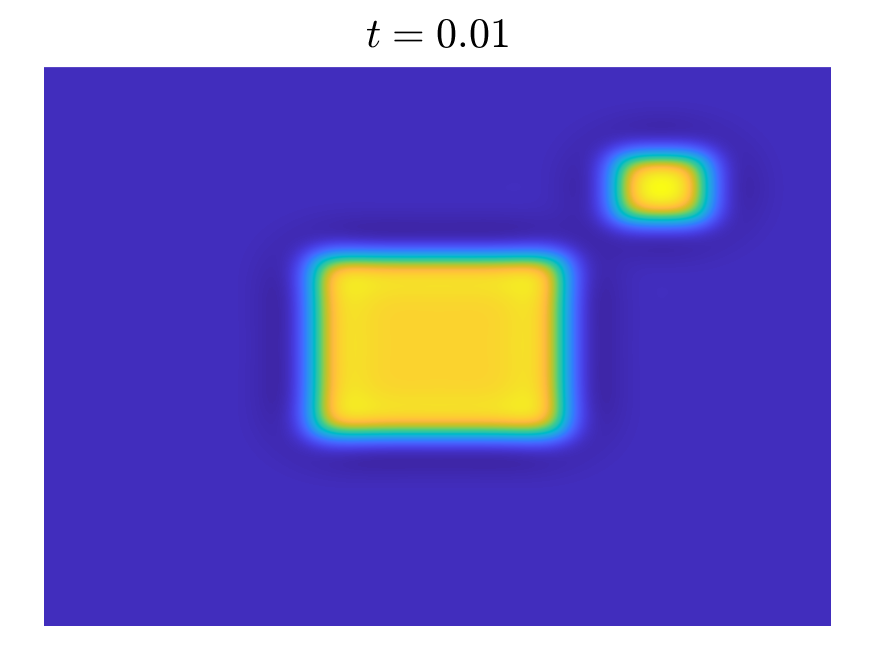}
    \includegraphics[width=3.6cm, height=3.6cm]{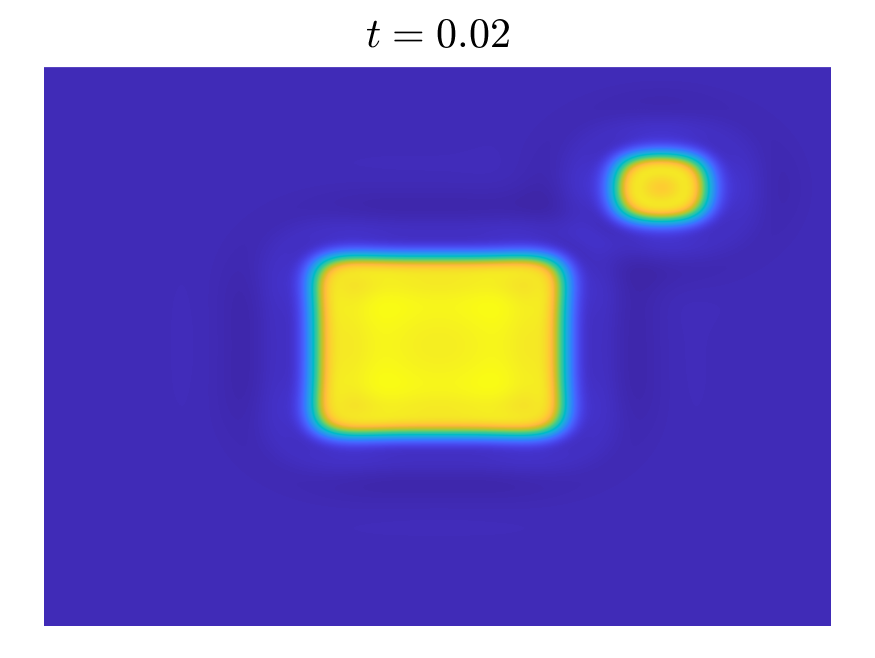}
    \includegraphics[width=3.6cm, height=3.6cm]{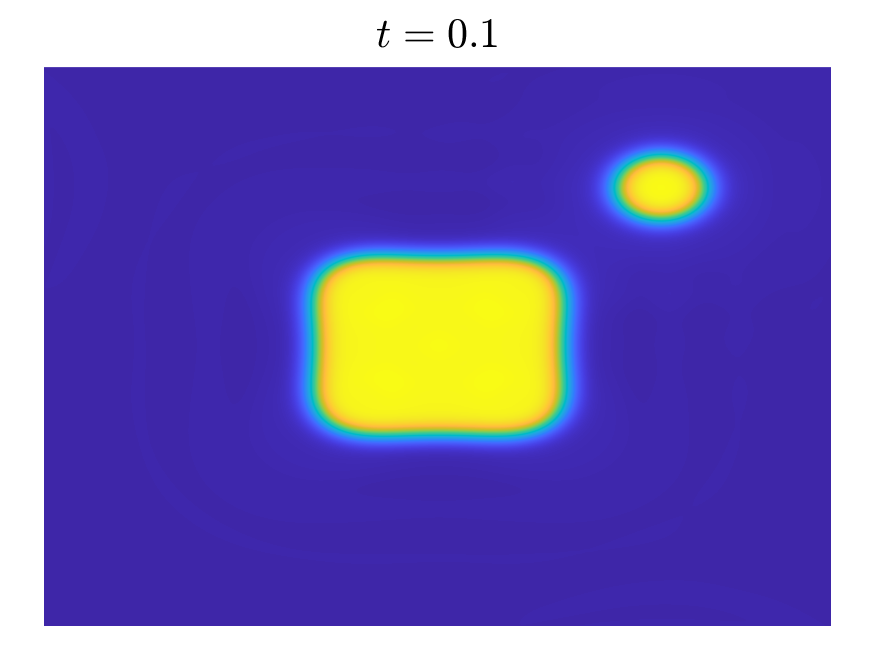}
    \includegraphics[width=3.6cm, height=3.6cm]{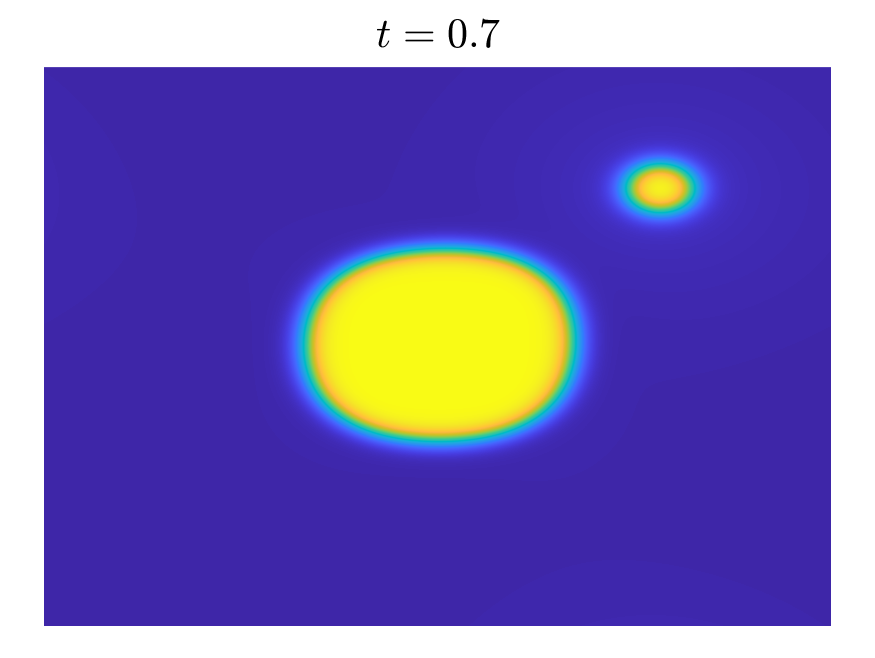} \\
    \includegraphics[width=3.6cm, height=3.6cm]{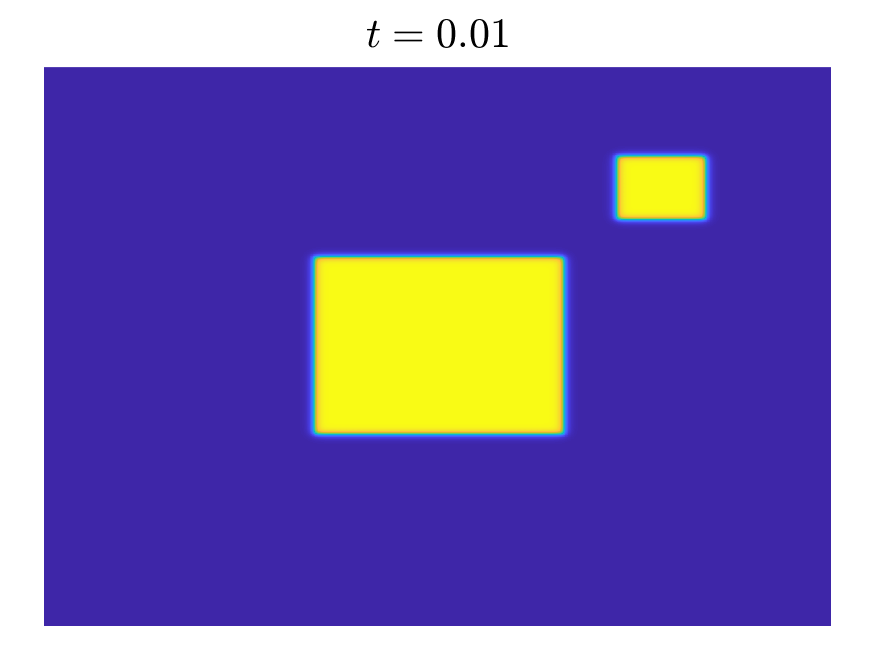}
    \includegraphics[width=3.6cm, height=3.6cm]{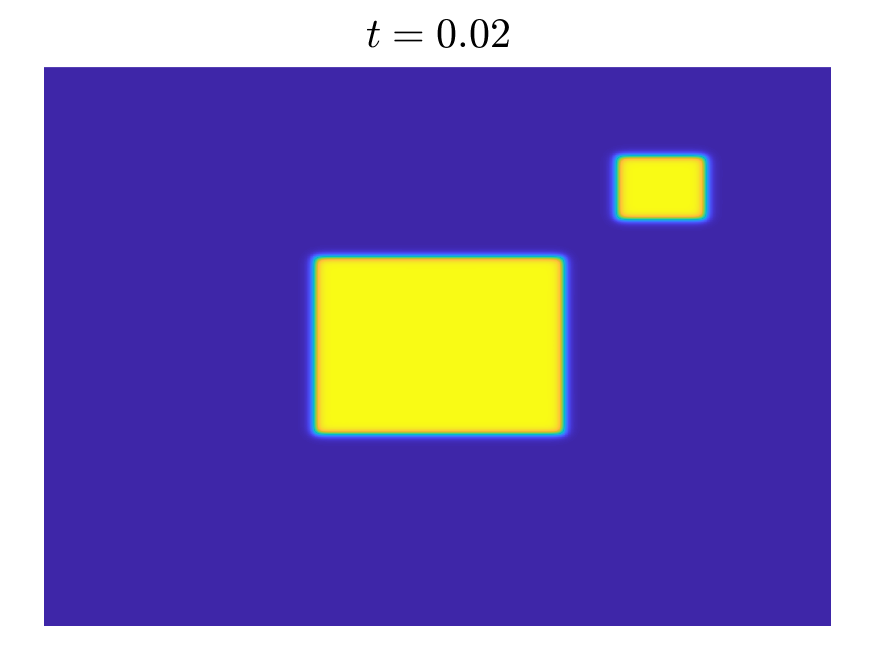}
    \includegraphics[width=3.6cm, height=3.6cm]{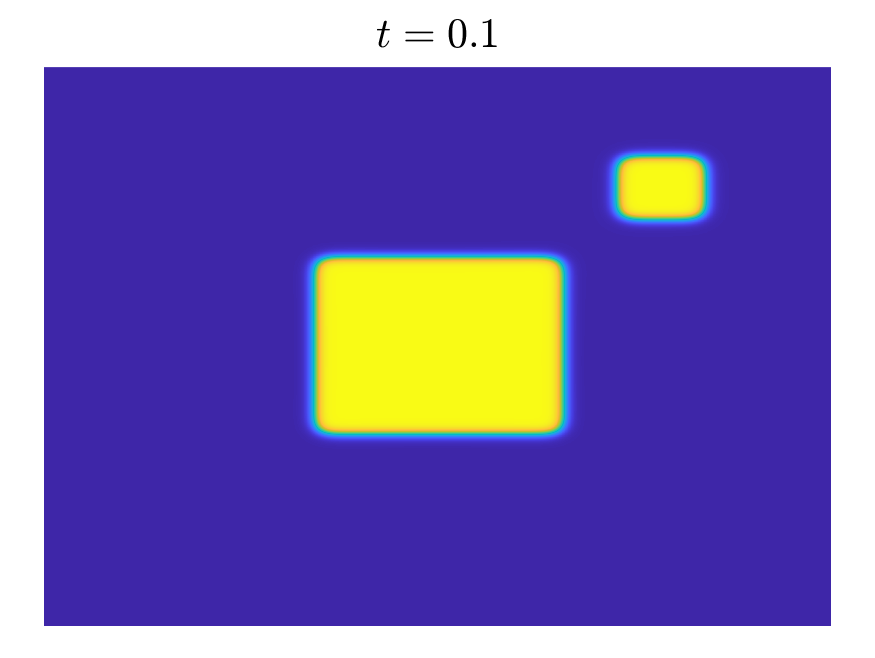}
    \includegraphics[width=3.6cm, height=3.6cm]{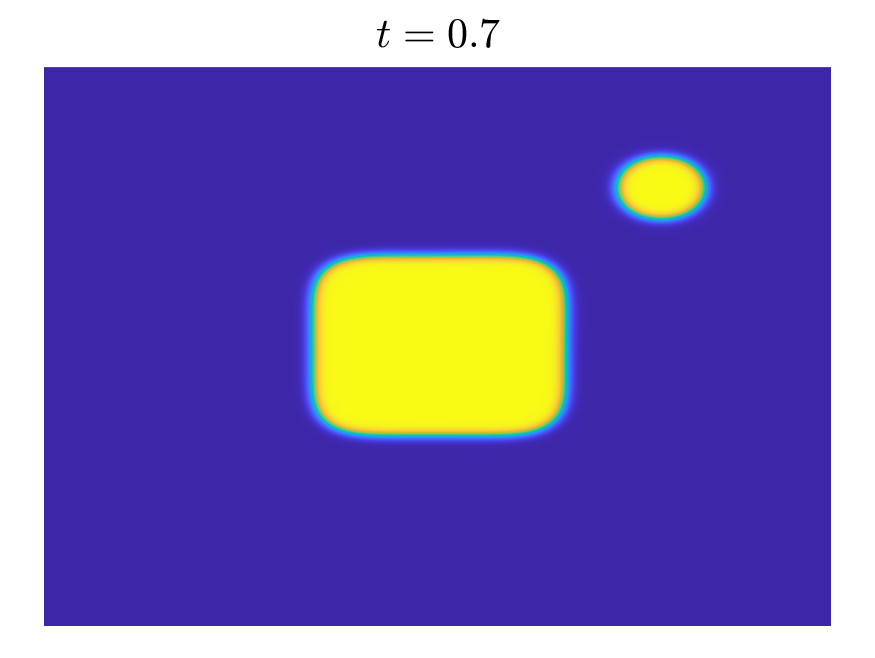}
    \caption{(Example \ref{Ex 2}) Contour plots of the numerical solution $\phi^n$ at $t = 0.01$, $0.02$, $0.1$ and $0.7$ with $\tau = 0.01$. Top: SAV-BE; Down: iSAV-BE.} \label{fig ex2 snapshots} \vspace{3pt}
\end{figure}

To further demonstrate the improvement of iSAV, we test the energy decaying rate (\ref{rate ex}) under the schemes. For the BE schemes, we aim to visualize the difference between (\ref{Energy Law of SAV-BE}) and (\ref{Original Energy Law of iSAV-BE}), especially when $\eps$ decreases.  
 Therefore, we plot the evolution of the quantity $\mathcal{E}[\phi^{n}]-\mathcal{E}[\phi^{n-1}]+\tau\| \mathcal{G}^{\frac{1}{2}}\mu^{n} \|^2$  in Figure \ref{fig ex2 E-E BE} for three different $\eps$. We can see that the result from  iSAV remains monotone and non-positive, while the result from SAV becomes positive and contains oscillations that become stronger as $\eps$ reduces.  For the BDF schemes, we test (\ref{Energy Law of iSAV-BDF}) and the results of 
 $\mathcal{E}_{2}[\phi^{n},\phi^{n-1}]-\mathcal{E}_{2}[\phi^{n-1},\phi^{n-2}]+\tau\| \mathcal{G}^{\frac{1}{2}}\mu^{n} \|^2$ are given in Figure \ref{fig ex2 E-E BDF}, where the proposed iSAV-BDF always produce non-positive results. 
  
In addition, we address Remarks \ref{Remark r-r BE} and \ref{Remark r-r BDF}  by computing the error $r[\phi^{n}]-r^{n}$ in SAV  and the error $r[\phi^{n}]-\tilde{r}^{n}$ in iSAV. 
With $\eps=0.04$ and $\tau=0.001$ fixed, the results are displayed in Figure \ref{fig ex2 rn diff}. It can be observed that in the iSAV schemes, the difference between the introduced auxiliary variable and the original variable is indeed much smaller than  that of SAV schemes.  At last, we show the contour plots of the numerical solutions from SAV-BE and iSAV-BE  in Figure \ref{fig ex2 snapshots}.  Here the SAV-BE  without stabilization terms,  as reported in \cite{Shen2019efficient,Shen2019NewClass}, exhibits spurious oscillations at the early  stage of the solution, while iSAV-BE does not. 

Last but not least, we use this example to address Remark \ref{rk tau0} on the technical condition about $\tau$.  We solve the problem under three different $\eps$ with  $S=3/\eps^2$, and we fix $\tau=0.1$ for iSAV-BE. The results given in Figure \ref{fig on tau0} show that iSAV-BE can have the original energy-stability for large $\tau$ not restricted by $S$ or $\eps$. 

Based on the above findings and comparisons, we can conclude that the proposed iSAV schemes are stable, accurate and can indeed gain practical improvements.

\begin{figure}[!ht]
    \centering
    \includegraphics[width=7.2cm, height=5.625cm]{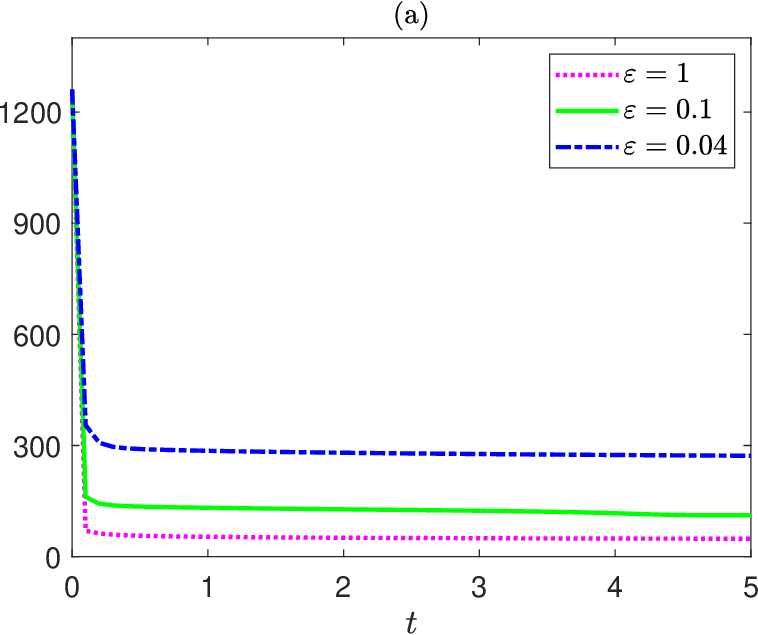}
    \includegraphics[width=7.2cm, height=5.625cm]{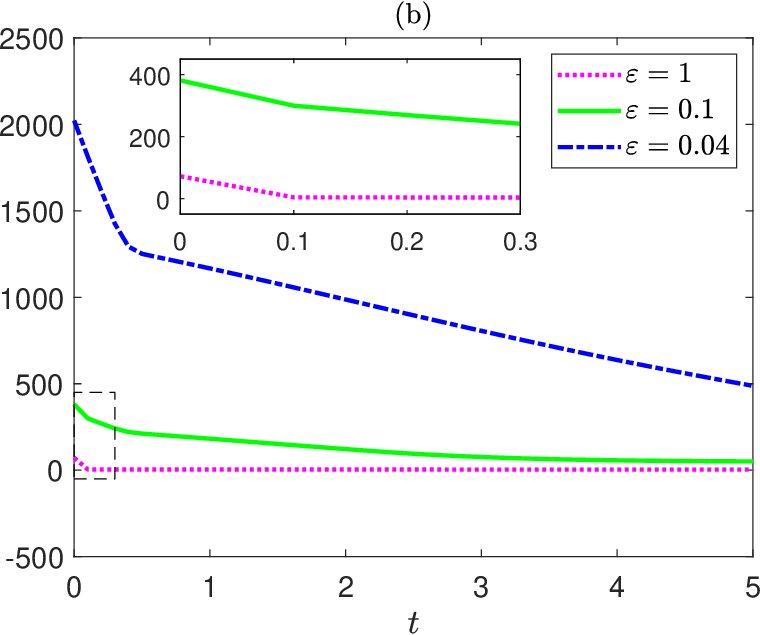}
    \caption{(Example \ref{Ex 2} and \ref{Ex 3}): Time evolution of original energy of iSAV-BE under a fixed $\tau=0.1$ for different $\varepsilon$. (a) Example \ref{Ex 2}; (b) Example \ref{Ex 3}.} \label{fig on tau0} \vspace{3pt}
\end{figure}

\subsection{Flory-Huggins (F-H) potential}
Next, we consider the Flory-Huggins type potential \cite{Binder1983collective,Fialkowski2008dynamics}, i.e.,
\begin{align*}
    F(\phi) = \frac{1}{\varepsilon^2}\left[ \phi\,{\rm ln}\phi +\left( 1-\phi \right)\,{\rm ln}(1-\phi) + \beta\left( \phi - \phi^{2} \right) \right],
\end{align*}
which is a popular class of logarithmic potentials in applications for (\ref{gradient flow}). 
In the F-H potential, the domain of definition  is an open interval $(0, 1)$, and so the numerical solution needs to be strictly confined within this domain to avoid calculation overflow, which is more challenging. A helpful regularization strategy   to the problem can be introduced by extending the domain $(0, 1)$ to $(-\infty, \infty)$ and replacing the logarithmic function for (\ref{gradient flow}) by a $C^2$-continuous, convex, and piecewisely defined  function \cite{Copetti1992,Zhang2020non}: for some  $0<\sigma\leq\frac{1}{2}$, 
\begin{align*}
    F(\phi) = \frac{1}{\varepsilon^2}
    \left\{
    \begin{array}{ll}
        \phi\,{\rm ln}\phi + \frac{\left( 1-\phi \right)^2}{2\sigma} + (1-\phi)\,{\rm ln}\sigma - \frac{\sigma}{2} + \beta\left( \phi - \phi^2 \right), & {\rm if}\;\; \phi\geq 1-\sigma, \\
        \phi\,{\rm ln}\phi +\left( 1-\phi \right)\,{\rm ln}(1-\phi) + \beta\left( \phi - \phi^{2} \right), & {\rm if}\;\; \sigma\leq\phi\leq 1-\sigma, \\
        (1-\phi)\,{\rm ln}(1-\phi) + \frac{\phi^2}{2\sigma} + \phi\,{\rm ln}\sigma - \frac{\sigma}{2} + \beta\left( \phi - \phi^2 \right), & {\rm if}\;\; \phi\leq \sigma.
    \end{array}
    \right.
\end{align*}
We shall show by this subsection that the proposed iSAV schemes (present only results of BE for simplicity) work well for (\ref{gradient flow})  with the logarithmic potential under such  strategy. 

\begin{example} \label{Ex 3}
    We consider the $H^{-\alpha}$ flow ($\alpha=0\;{\rm or}\;1$) with  parameters 
    $\varepsilon = 0.04, \,\gamma = 0.5, \,  \beta = 3, \,  \sigma = 0.01, \,S = 10/\varepsilon^2$ and the computational domain $\Omega=[0,2\pi]\times[0,2\pi]$ 
    with a $256\times 256$ spatial discrete grid. The initial value of (\ref{gradient flow}) is given as 
    \begin{align*}
        \phi(x,y,0) = \left\{
        \begin{aligned}
            &0.7,\qquad (x,y)\in\Omega_{0}, \\
            &0.3,\qquad  (x,y)\in\Omega/\Omega_{0}, 
        \end{aligned}
        \right.
    \end{align*}
    where $ \Omega_{0} = \big\{ (x,y)\in\Omega\,:\, (|x-x_{1}|^2\leq 1.4^2\cap |y-y_{1}|\leq 1.4^2) \cup (|x-x_{2}|^2\leq 0.5^2\cap |y-y_{2}|\leq 0.5^2) \big\} $ denotes two disk regions centered at $(x_{1},y_{1})=(\pi-0.8,\pi)$ and $(x_{2},y_{2})=(\pi+1.7,\pi)$. To maintain a positive lower bound for $r[\phi]$, an additional constant $ 0.06/\varepsilon^2 $ has been added in $F(\phi)$.
\end{example}

\begin{figure}[!ht]
    \centering
    \includegraphics[width=4.5cm, height=3.6cm]{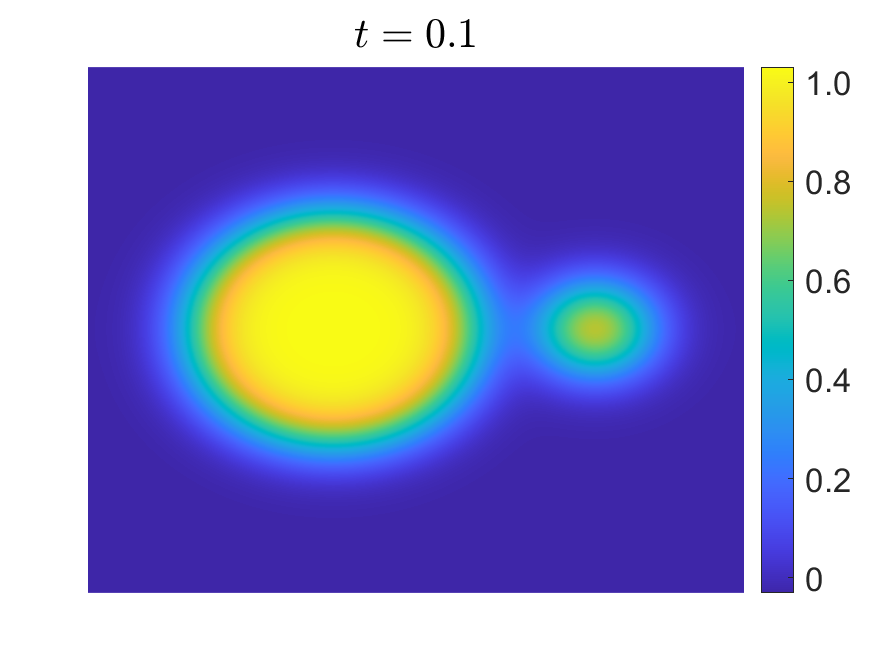}
    \includegraphics[width=4.5cm, height=3.6cm]{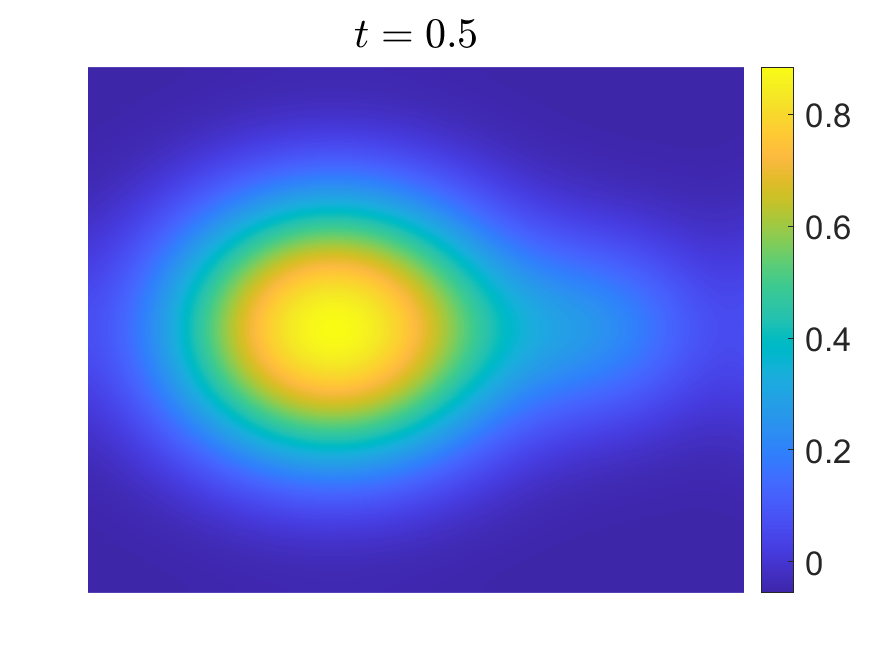}
    \includegraphics[width=4.5cm, height=3.6cm]{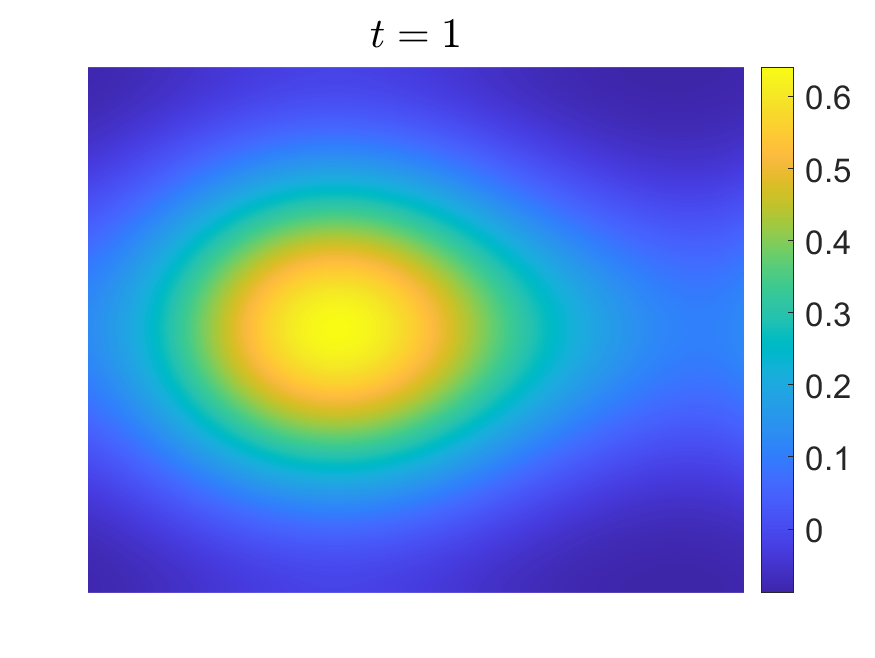} \\
    \includegraphics[width=4.5cm, height=3.6cm]{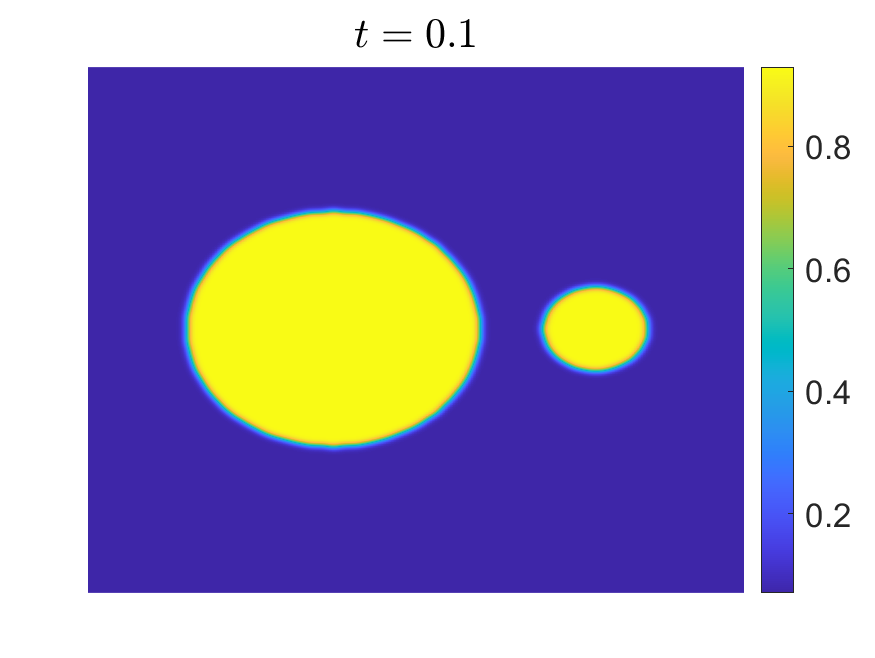}
    \includegraphics[width=4.5cm, height=3.6cm]{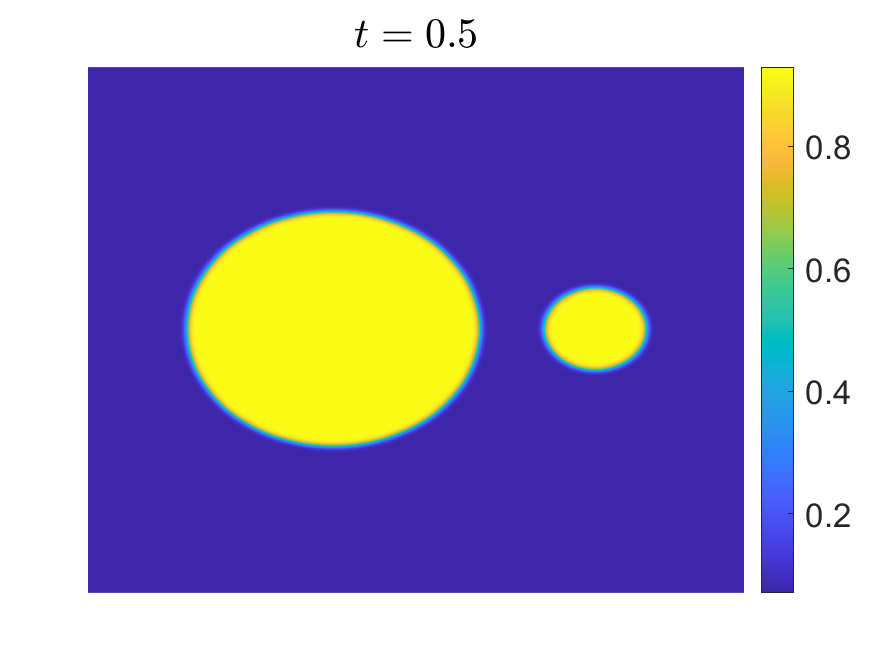}
    \includegraphics[width=4.5cm, height=3.6cm]{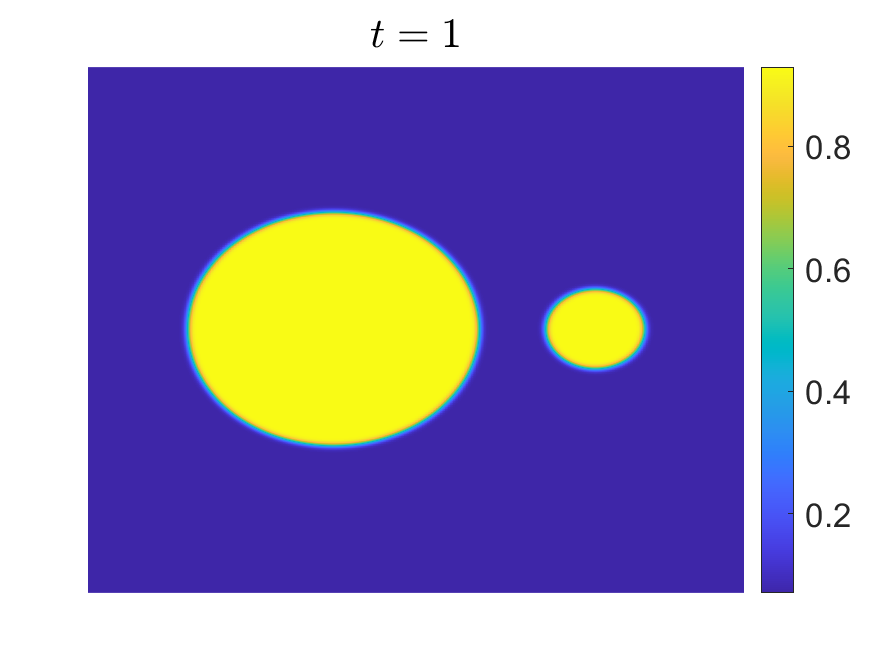}
    \caption{(Example \ref{Ex 3}, $\alpha=0$) Contour plots of the numerical solutions $\phi^n$ at $t = 0.1$, $0.5$ and $1.0$ with $\tau = 0.01$. Top: SAV-BE; Down: iSAV-BE.} \label{fig ex3 H0 snapshots} \vspace{3pt}
\end{figure}

\begin{figure}[!ht]
    \centering
    \includegraphics[width=4.5cm, height=3.6cm]{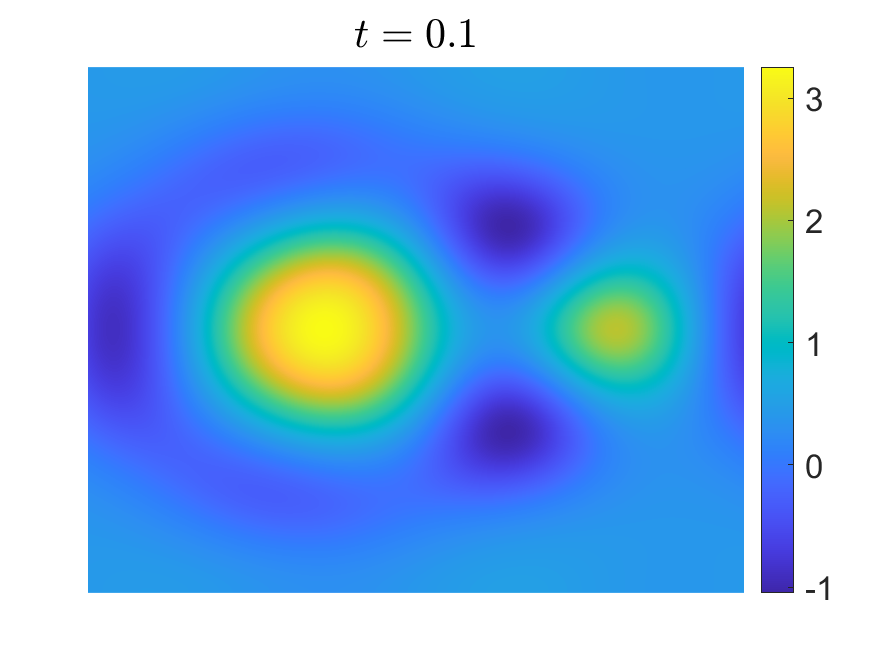}
    \includegraphics[width=4.5cm, height=3.6cm]{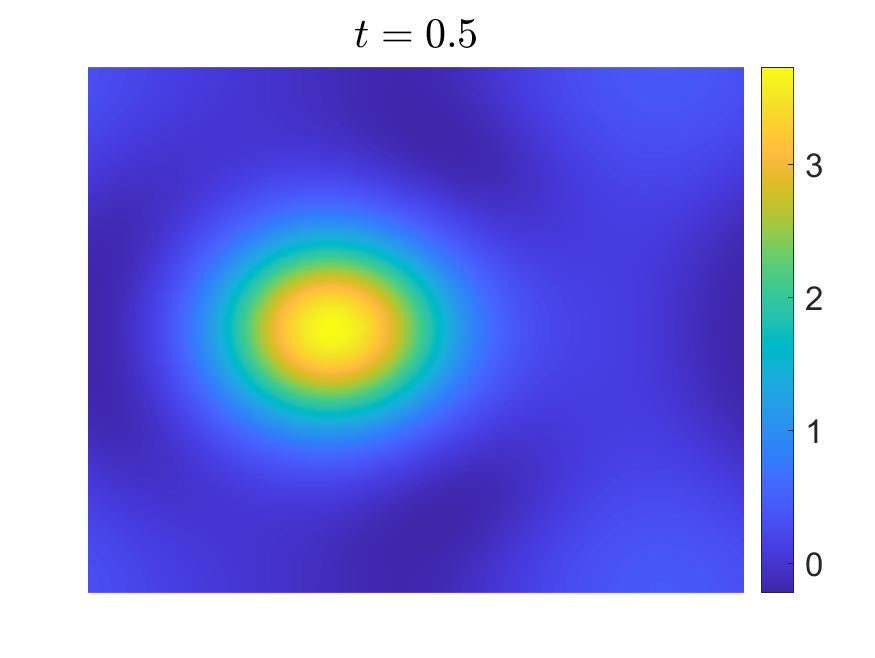}
    \includegraphics[width=4.5cm, height=3.6cm]{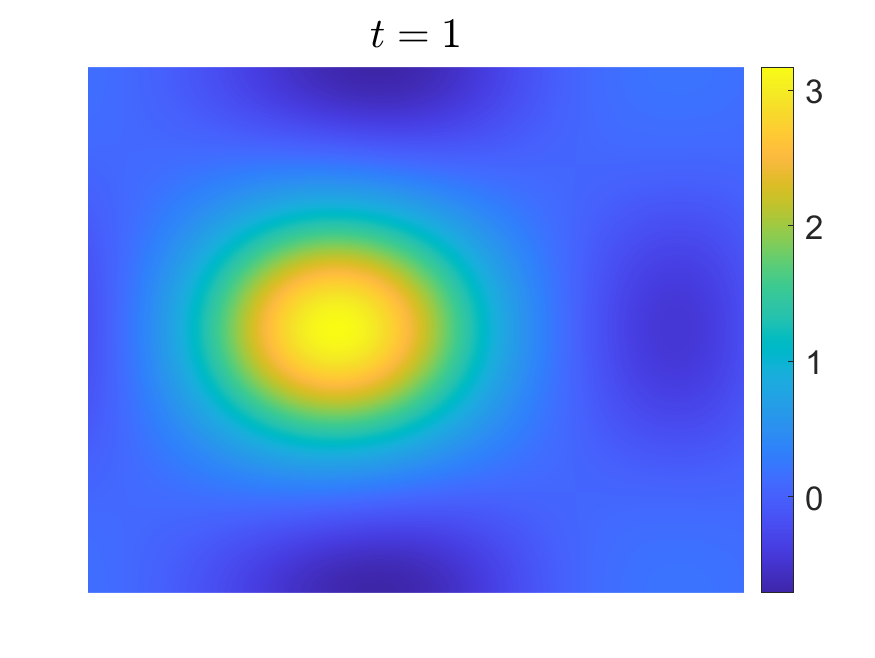} \\
    \includegraphics[width=4.5cm, height=3.6cm]{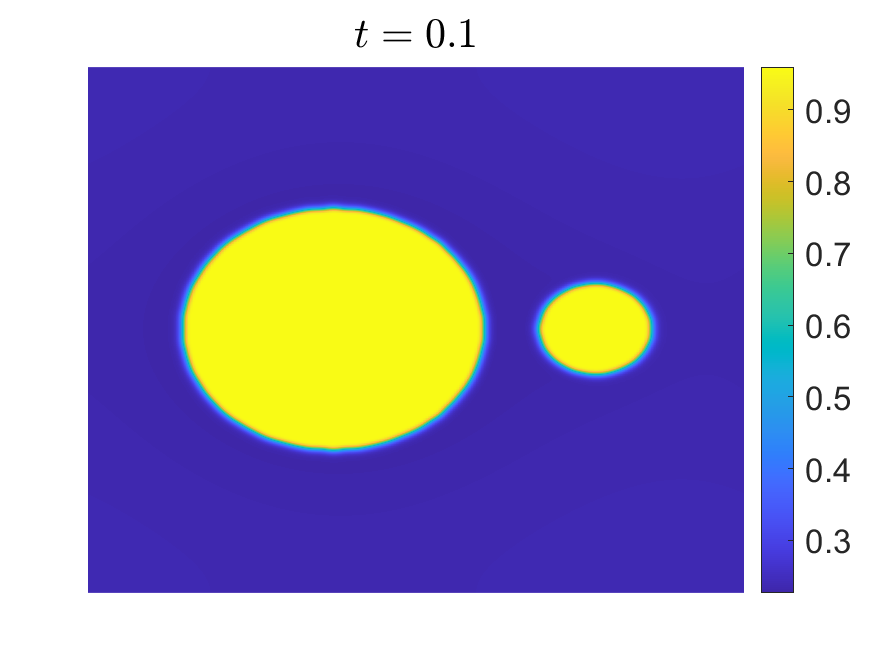}
    \includegraphics[width=4.5cm, height=3.6cm]{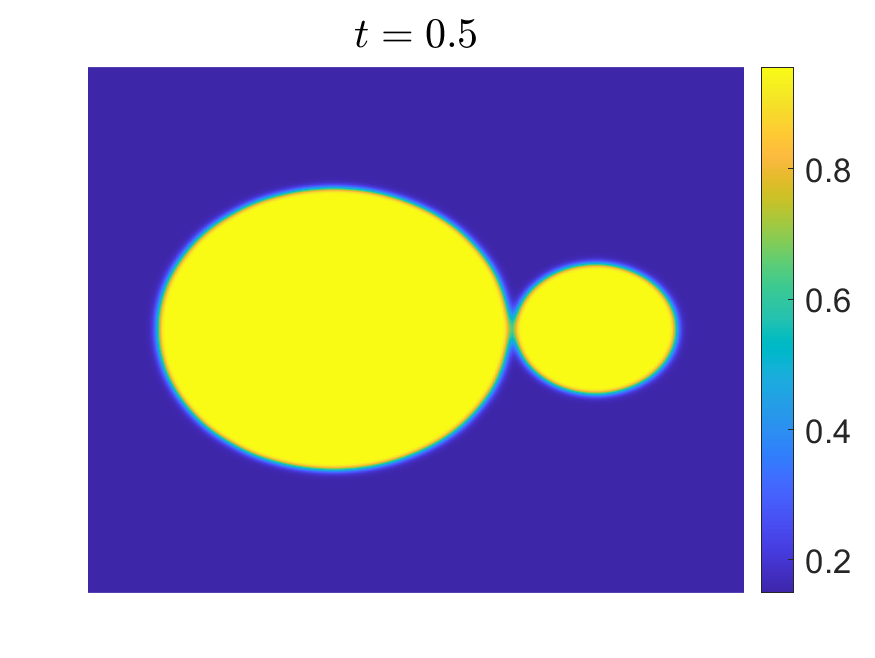}
    \includegraphics[width=4.5cm, height=3.6cm]{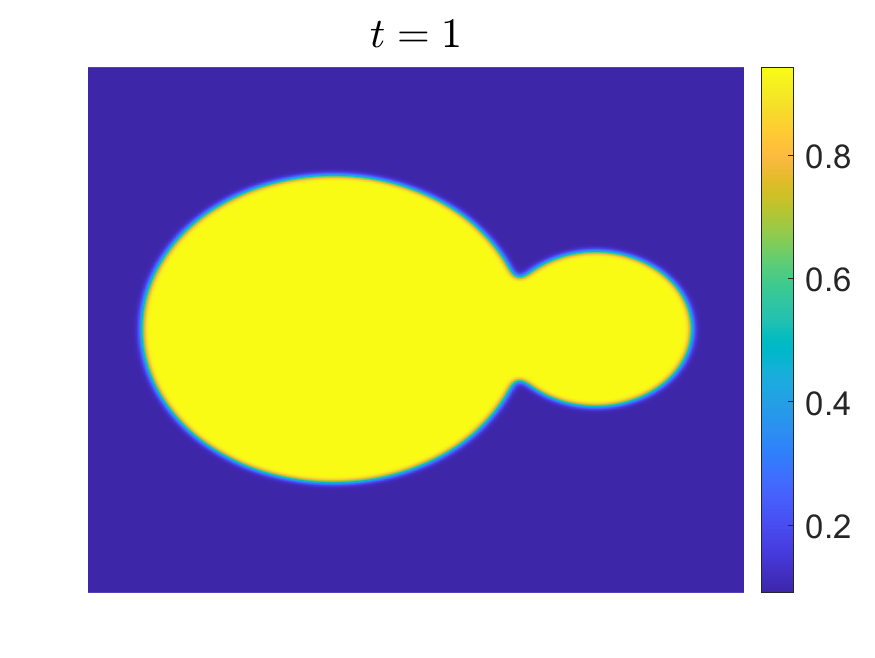}
    \caption{(Example \ref{Ex 3}, $\alpha=1$) Contour plots of  numerical solutions $\phi^n$ at $t = 0.1$,  $0.5$ and $1.0$ with $\tau = 0.01$. Top: SAV-BE; Down: iSAV-BE.} \label{fig ex3 H1 snapshots} \vspace{3pt}
\end{figure}

\begin{figure}[!ht]
    \centering
    \includegraphics[width=7.2cm, height=5.625cm]{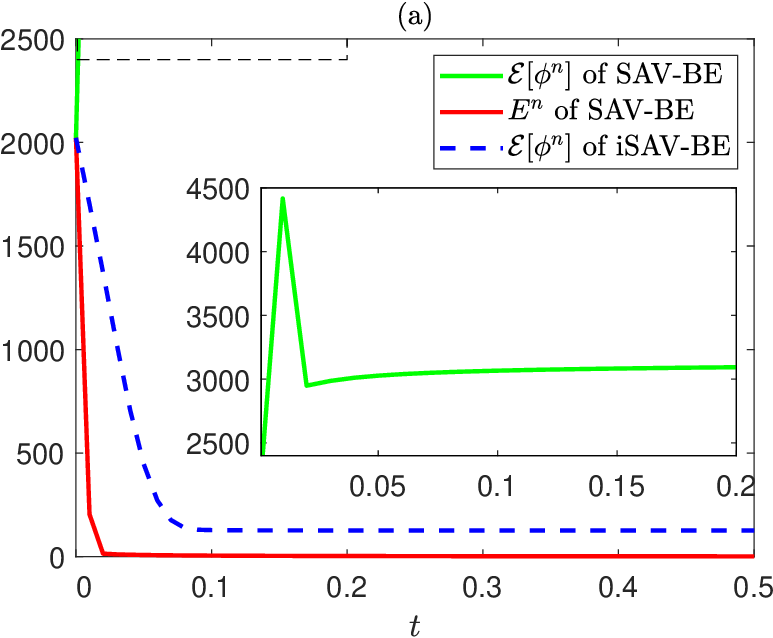}
    \includegraphics[width=7.2cm, height=5.625cm]{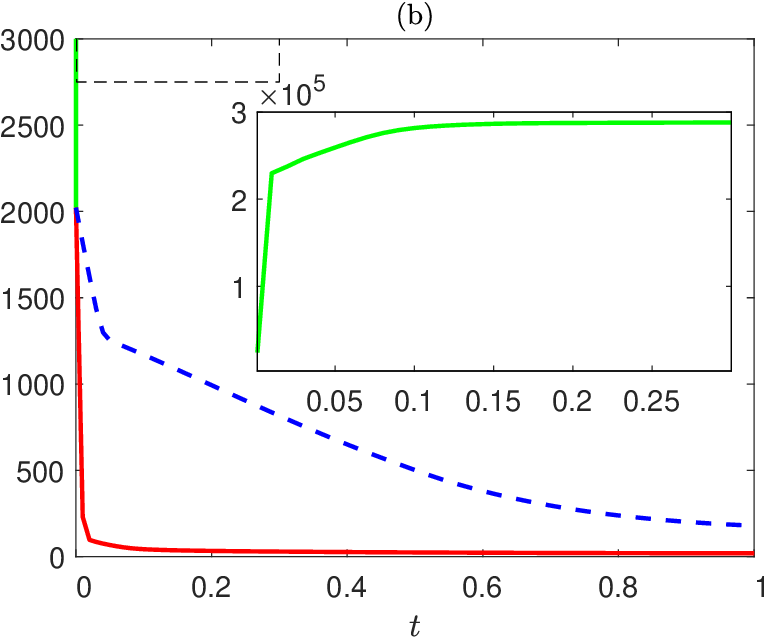}
    \caption{(Example \ref{Ex 3}): Time evolution of the energy of SAV-BE and iSAV-BE with $\tau=0.01$. (a) $\alpha=0$; (b) $\alpha=1$.} \label{fig ex3 energy} \vspace{3pt}
\end{figure}

With $\tau=0.01$ fixed, Figures \ref{fig ex3 H0 snapshots} and \ref{fig ex3 H1 snapshots} display the numerical solutions obtained by  SAV-BE and iSAV-BE  at different time for Allen-Cahn model and Cahn-Hilliard model, respectively. Figure \ref{fig ex3 energy} shows the energy evolution of each scheme. From the numerical results, we can observe that the SAV schemes (without stabilization) exhibit the similar  instability issue as before and the problem seems  severer in the case of F-H potential. Numerical solutions of SAV-BE can exceed the interval $(0,1)$ (1st line of Figure \ref{fig ex3 H1 snapshots}), resulting in false states with significant increase in the original energy (Figure \ref{fig ex3 energy}). These phenomena do not occur in iSAV. The performance of iSAV-BE  with large $\tau$ free from $\eps$ under this example is included in Figure \ref{fig on tau0}.

\begin{example} \label{Ex 4}
  With  all other parameters same as in Example \ref{Ex 3}, we now consider a random initial data:
    \begin{align*}
        \phi(x,y,0) = 0.5 + 0.2\,{\rm Rand}(x,y),
    \end{align*}
    where ${\rm Rand}(x,y)$ stands for randomly generated 
 values in $[-1,1]$ under the uniform distribution.
\end{example}

\begin{figure}[!ht]
    \centering
    \includegraphics[width=3.6cm, height=3.6cm]{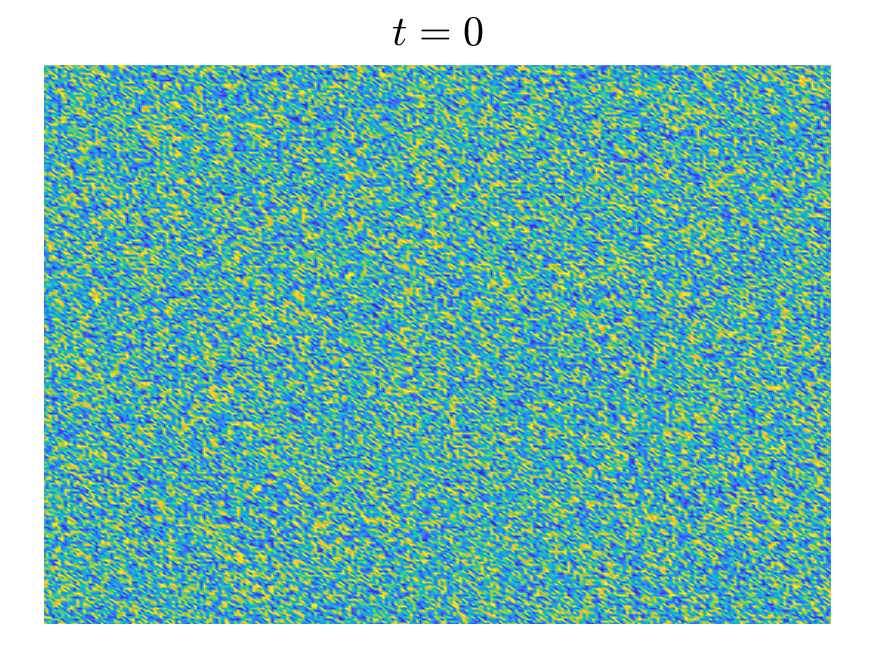}
    \includegraphics[width=3.6cm, height=3.6cm]{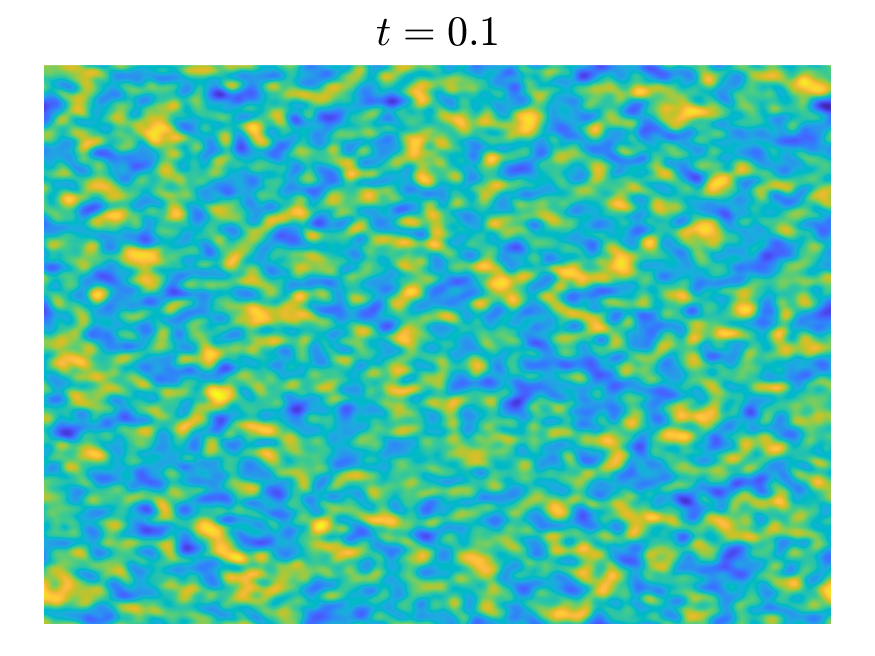}
    \includegraphics[width=3.6cm, height=3.6cm]{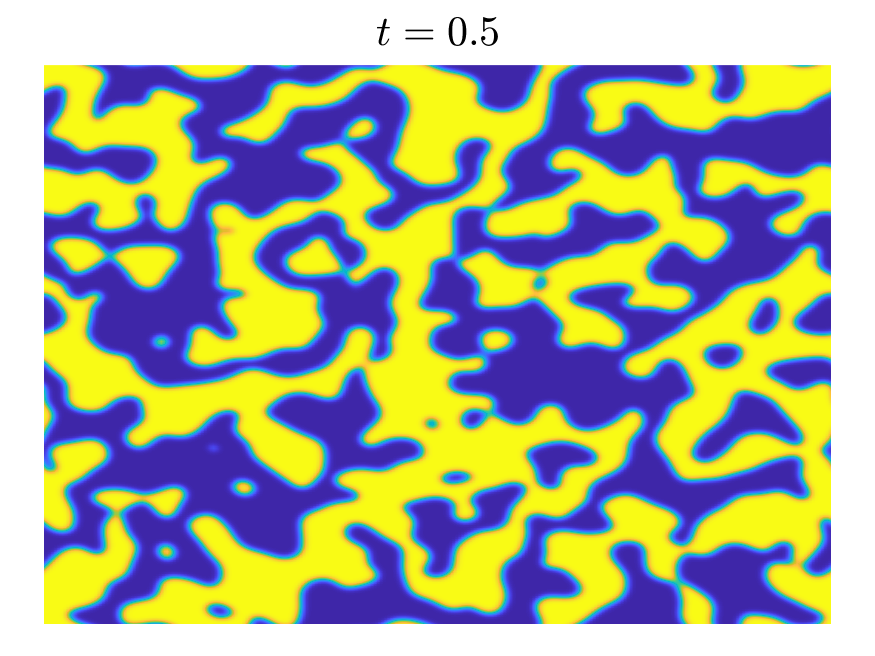}
    \includegraphics[width=3.6cm, height=3.6cm]{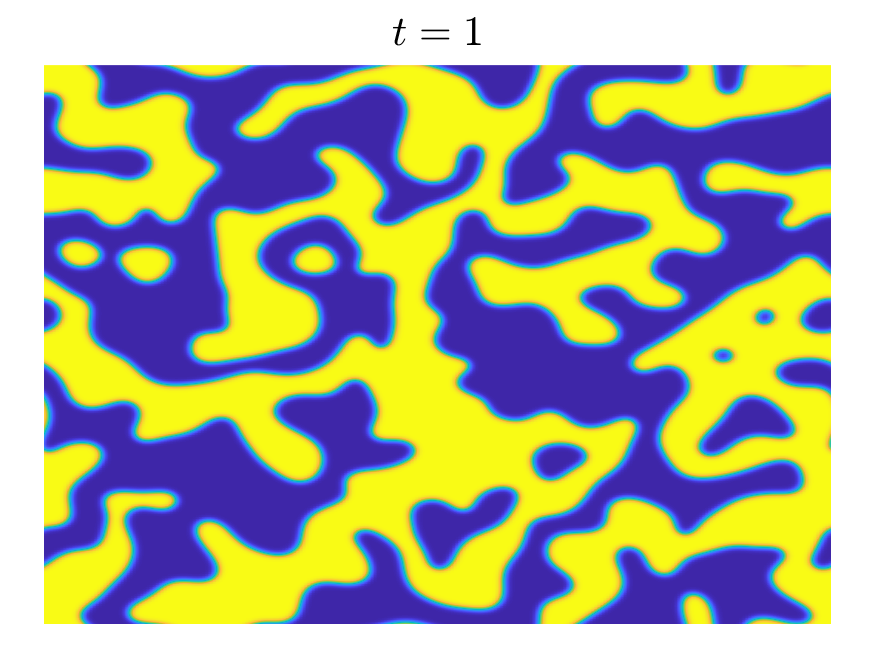}
    \caption{(Example \ref{Ex 4}, $\alpha=0$) Contour plots of the numerical solutions $\phi^n$ from iSAV-BE at $t = 0$, $0.1$, $0.5$ and $1$ with $\tau = 0.01$.} \label{fig ex4 snapshots} \vspace{3pt}
\end{figure}

Figure \ref{fig ex4 snapshots} displays the states of the $H^{0}$ flow obtained by the iSAV-BE scheme  at different time. We can observe that although the initial value for (\ref{gradient flow}) is extremely rough, the iSAV scheme  still works well, where the  coarsening process can be  captured. This indicates the smoothness  assumption in the theorem is technical.

\section{Conclusion}\label{sec:conclusion}
This paper proposed some linear numerical schemes for solving gradient flows under the framework of the scalar auxiliary variable (SAV) approach, where we aimed to improve the stability of the original energy of the problem. In particular, a first-order improved SAV (iSAV) scheme which involves no nonlinear equations, was constructed  and was rigorously proved to offer the original energy-dissipation. 
Convergence analysis with the optimal error bound has  been established for it. Numerical results validated the theoretical accuracy and original energy stability of the proposed iSAV scheme, and extensive comparisons with SAV were made to illustrate the gain. The possible second-order extension was also suggested and supported by numerical experiments.

\appendix
\section{Derivation of (\ref{Energy Law of SAV-BE}).}\label{Appendix 0}
Assume that $\phi^{n}\;(n=0,1,\cdots)$, $\phi_{t}$, $f$ and $f'$ are all bounded.
By direct calculations and (\ref{modified energy of SAV-BE}), we have
\begin{align}
    \frac{1}{\tau}\left( \mathcal{E}[\phi^{n+1}] -  \mathcal{E}[\phi^{n}] \right) =& \frac{1}{\tau}\left( E^{n+1}-E^{n} \right) + \frac{1}{\tau}\left( \mathcal{E}[\phi^{n+1}]-E^{n+1} \right) - \frac{1}{\tau}\left( \mathcal{E}[\phi^{n}] - E^{n} \right) \nonumber\\
    \leq& -\left\| \mathcal{G}^{\frac{1}{2}}\mu^{n+1} \right\|^2 + \frac{1}{\tau}\left[ \left( r[\phi^{n+1}] \right)^2 - \left( r^{n+1} \right)^2 - \left( r[\phi^{n}] \right)^2 + \left( r^{n} \right)^2 \right].\label{App 1}
\end{align}
Then by (\ref{SAV-BE scheme 3}) and (\ref{20240111-01}), we can derive 
\begin{align*}
    &\left( r[\phi^{n+1}] \right)^2 - \left( r^{n+1} \right)^2 - \left( r[\phi^{n}] \right)^2 + \left( r^{n} \right)^2 = \frac{r[\phi^{n}] - r^{n}}{\sqrt{\mathcal{F}[\phi^{n}]}}\int_{\Omega}{ f(\phi^{n})\left( \phi^{n+1} - \phi^{n} \right) }\,{\rm d}{\bf x} \\
    &\hspace{1.5cm} + \frac{1}{2}\int_{\Omega}{ f'\left( (1-\rho_{1})\phi^{n} + \rho_{1}\phi^{n+1} \right)(\phi^{n+1} - \phi^{n})^2 }\,{\rm d}{\bf x} - \frac{1}{4\mathcal{F}[\phi^{n}]}\left[ \int_{\Omega}{ f(\phi^{n})\left( \phi^{n+1} - \phi^{n} \right) }\,{\rm d}{\bf x} \right]^2 = \mathcal{O}(\tau^2),
\end{align*}
where the last equality made use of the convergence result  of  (\ref{SAV-BE scheme}) from  \cite{Shen2018ConvergenceSAV}. Plugging  the above into (\ref{App 1}), we can obtain (\ref{Energy Law of SAV-BE}).

\section{Derivation of (\ref{Energy Law of iSAV-BDF}).}\label{Appendix 1}
Firstly, by taking the $L^2$ inner product of (\ref{iSAV-BDF scheme 1}) with $\mu^{n+1}$, we have
\begin{align}
    \frac{1}{2\tau}\left\langle 3\phi^{n+1}-4\phi^{n}+\phi^{n-1},\mu^{n+1} \right\rangle = -\left\| \mathcal{G}^\frac{1}{2}\mu^{n+1} \right\|^2. \label{BDF inner product 1}
\end{align}
Secondly, by taking the $L^2$ inner product of (\ref{iSAV-BDF scheme 2}) with $(3\phi^{n+1}-4\phi^{n}+\phi^{n-1})/2\tau$, and using the identities
\begin{align*}
    &2\langle \mathcal{L}u_{1},3u_{1}-4u_{2}+u_{3} \rangle = \left\|\mathcal{L}^{\frac{1}{2}}u_{1}\right\|^{2} - \left\|\mathcal{L}^{\frac{1}{2}}u_{2}\right\|^{2}  + \left\|\mathcal{L}^{\frac{1}{2}}(2u_{1}-u_{2})\right\|^{2} -\left\|\mathcal{L}^{\frac{1}{2}}(2u_{2}-u_{3})\right\|^{2} \nonumber\\
    &\hspace{4cm} + \left\|\mathcal{L}^{\frac{1}{2}}(u_{1}-2u_{2}+u_{3})\right\|^{2}, \\
    &\langle u_{1}-2u_{2}+u_{3},3u_{1}-4u_{2}+u_{3} \rangle = \left\|u_{1}-u_{2}\right\|^2 - \left\|u_{2}-u_{3}\right\|^2 + 2\left\|u_{1}-2u_{2}+u_{3}\right\|^2, 
\end{align*}
for some general functions $u_1,u_2,u_3$, we can obtain 
\begin{align}
    &\frac{1}{2\tau}\left\langle \mu^{n+1}, 3\phi^{n+1}-4\phi^{n}+\phi^{n-1} \right\rangle \nonumber\\
    =& \frac{1}{4\tau}\left( \left\| \mathcal{L}^\frac{1}{2}\phi^{n+1} \right\|^2 - \left\| \mathcal{L}^\frac{1}{2}\phi^{n} \right\|^2 + \left\| \mathcal{L}^\frac{1}{2}(2\phi^{n+1}-\phi^{n}) \right\|^2 - \left\| \mathcal{L}^\frac{1}{2}(2\phi^{n}-\phi^{n-1}) \right\|^2 \right. \nonumber\\
    & \left.  + \left\| \mathcal{L}^\frac{1}{2}(\phi^{n+1}-2\phi^{n}+\phi^{n-1}) \right\|^2 \right) + \frac{\tilde{r}^{n+1}}{2\tau}\left\langle \frac{ f(\phi_{*}^{n+1}) }{ \sqrt{\mathcal{F}[\phi_{*}^{n+1}]} },3\phi^{n+1}-4\phi^{n}+\phi^{n-1} \right\rangle \nonumber\\
    & + \frac{S}{2\tau}\left( \left\| \phi^{n+1} - \phi^{n} \right\|^2 - \left\| \phi^{n} - \phi^{n-1} \right\|^2 + 2\left\| \phi^{n+1}-2\phi^{n}+\phi^{n-1} \right\|^2 \right). \label{BDF inner product 2}
\end{align}
Thirdly, multiplying (\ref{iSAV-BDF scheme 3}) by $\tilde{r}^{n+1}/\tau$ and again by the identity, we can obtain that
\begin{align}
    &\frac{1}{2\tau}\left[ \left( \tilde{r}^{n+1} \right)^2 - \left( r[\phi^{n}] \right)^2 + \left( 2\tilde{r}^{n+1}-r[\phi^{n}] \right)^2 - \left( 2r[\phi^{n}]-r[\phi^{n-1}] \right)^2 + \left( \tilde{r}^{n+1}-2r[\phi^{n}]+r[\phi^{n-1}] \right)^2 \right] \nonumber\\
    =& \frac{\tilde{r}^{n+1}}{2\tau}\left\langle \frac{ f(\phi_{*}^{n+1}) }{ \sqrt{\mathcal{F}[\phi_{*}^{n+1}]} },3\phi^{n+1}-4\phi^{n}+\phi^{n-1} \right\rangle.\label{BDF inner product 3}
\end{align}

Combining (\ref{BDF inner product 1}), (\ref{BDF inner product 2}) and (\ref{BDF inner product 3}), we have
\begin{align}
    \frac{1}{\tau}\left( \tilde{E}_{2}^{n+1} - \mathcal{E}_{2}[\phi^{n},\phi^{n-1}] \right) = -\left\| \mathcal{G}^\frac{1}{2}\mu^{n+1} \right\|^2 - \mathcal{N}^{n} -\frac{S}{\tau}\left\| \phi^{n+1}-2\phi^{n}+\phi^{n-1} \right\|^2,  \label{Temp 3}
\end{align}
with $\tilde{E}_{2}^{n+1}$ and $\mathcal{N}^{n}$  defined as
\begin{align*}
    &\tilde{E}_{2}^{n+1} = \frac{1}{4}\left( \left\| \mathcal{L}^{\frac{1}{2}}\phi^{n+1} \right\|^2 + \left\| \mathcal{L}^{\frac{1}{2}}(2\phi^{n+1}-\phi^{n}) \right\|^2 \right) + \frac{1}{2}\left[ \left( \tilde{r}^{n+1} \right)^2 + \left( 2\tilde{r}^{n+1}-r[\phi^{n}] \right)^2 \right] + \frac{S}{2}\left\| \phi^{n+1} - \phi^{n} \right\|^2, \\ 
    &\mathcal{N}^{n} = \frac{1}{4\tau}\left\| \mathcal{L}^{\frac{1}{2}}(\phi^{n+1}-2\phi^{n}+\phi^{n-1}) \right\|^2 + \frac{1}{2\tau}\left( \tilde{r}^{n+1}-2r[\phi^{n}]+r[\phi^{n-1}] \right)^2.
\end{align*}
On the other hand, we have 
\begin{align}
    \frac{1}{\tau}\left( \mathcal{E}_{2}[\phi^{n+1},\phi^{n}] - \tilde{E}_{2}^{n+1} \right) =& \frac{1}{2\tau}\left[ \left( r[\phi^{n+1}] \right)^2 - \left( \tilde{r}^{n+1} \right)^2 + \left( 2r[\phi^{n+1}]-r[\phi^{n}] \right)^2 - \left( 2\tilde{r}^{n+1}-r[\phi^{n}] \right)^2 \right] \nonumber\\
    =& \frac{1}{2\tau}\left( r[\phi^{n+1}] - \tilde{r}^{n+1} \right)\left( 5r[\phi^{n+1}] + 5\tilde{r}^{n+1} - 4r[\phi^{n}] \right).\label{Temp 4}
\end{align}
Then, we compute the difference between $r[\phi^{n+1}]$ and $\tilde{r}^{n+1}$. By setting $\phi = \phi_{*}^{n+1}$, $\psi = \phi^{n+1}-2\phi^{n}+\phi^{n-1}$, $\rho=1$ and $k=1$ in (\ref{Taylor expansion}), we have
\begin{align}
    r[\phi^{n+1}] =& r[\phi_{*}^{n+1}+(\phi^{n+1}-2\phi^{n}+\phi^{n-1})] \nonumber\\
    =& r[\phi_{*}^{n+1}] + \frac{1}{2\sqrt{\mathcal{F}[\phi_{*}^{n+1}]}}\int_{\Omega}{ f(\phi_{*}^{n+1})(\phi^{n+1}-2\phi^{n}+\phi^{n-1}) }\,{\rm d}{\bf x} + \mathcal{X}_{1}^{n}, \label{rphin+1 taylor expansion}
\end{align}
where 
\begin{align*}
    \mathcal{X}_{1}^{n} = \frac{1}{4}\big(\mathcal{E}_{1}[\xi_{0}^{n}]\big)^{-\frac{1}{2}}\int_{\Omega}{ f'(\xi_{0}^{n}) (\phi^{n+1}-2\phi^{n}+\phi^{n-1})^2 }\,{\rm d}{\bf x}-\frac{1}{8}\big( \mathcal{E}_{1}[\xi_{0}^{n}] \big)^{-\frac{3}{2}}\Big(\int_{\Omega}{ f(\xi_{0}^{n}) (\phi^{n+1}-2\phi^{n}+\phi^{n-1}) }\,{\rm d}{\bf x}\Big)^2,
\end{align*}
with $\xi_{0}^{n} = \rho_{0}\phi^{n+1} + 2(1-\rho_{0})\phi^{n} + (\rho_{0}-1)\phi^{n-1}$ for some $\rho_{0}\in(0,1)$. And by setting $\phi=\phi_{*}^{n+1}$, $\psi=\phi^{n} - \phi^{n-1}$, $\rho=-1,-2$ and $k=2$ in (\ref{Taylor expansion}), there are
\begin{align}
    r[\phi^{n}] =& r[\phi_{*}^{n+1}-\psi] = r[\phi_{*}^{n+1}] - \frac{1}{2\sqrt{\mathcal{F}[\phi_{*}^{n+1}]}}\int_{\Omega}{ f(\phi_{*}^{n+1})(\phi^{n}-\phi^{n-1}) }\,{\rm d}{\bf x} \nonumber\\
    &\qquad\qquad\qquad+ \frac{1}{2}\frac{{\rm d}^2}{{\rm d}\rho^2}r[\phi_{*}^{n+1}+\rho (\phi^{n}-\phi^{n-1})]\Big|_{\rho=0} - \frac{1}{6}\frac{{\rm d}^3}{{\rm d}\rho^3}r[\phi_{*}^{n+1}+\rho (\phi^{n}-\phi^{n-1})]\Big|_{\rho=\rho_{1}}, \label{rphin taylor expansion}\\
    r[\phi^{n-1}] =& r[\phi_{*}^{n+1}-2\psi] = r[\phi_{*}^{n+1}] - \frac{1}{\sqrt{\mathcal{F}[\phi_{*}^{n+1}]}}\int_{\Omega}{ f(\phi_{*}^{n+1})(\phi^{n}-\phi^{n-1}) }\,{\rm d}{\bf x} \nonumber\\
    &\qquad\qquad\qquad + 2\frac{{\rm d}^2}{{\rm d}\rho^2}r[\phi_{*}^{n+1}+\rho (\phi^{n}-\phi^{n-1})]\Big|_{\rho=0} - \frac{4}{3}\frac{{\rm d}^3}{{\rm d}\rho^3}r[\phi_{*}^{n+1}+\rho (\phi^{n}-\phi^{n-1})]\Big|_{\rho=\rho_{2}}, \label{rphin-1 taylor expansion}
\end{align}
for some $\rho_{1}\in(-1,0)$, $\rho_{2}\in(-2,0)$. Combining (\ref{rphin+1 taylor expansion}), (\ref{rphin taylor expansion}) and (\ref{rphin-1 taylor expansion}), one can obtain
\begin{align}
    &3r[\phi^{n+1}]-4r[\phi^{n}]+r[\phi^{n-1}] = \frac{1}{2\sqrt{\mathcal{F}[\phi_{*}^{n+1}]}}\int_{\Omega}{ f(\phi_{*}^{n+1})(3\phi^{n+1}-4\phi^{n}+\phi^{n-1}) }\,{\rm d}{\bf x} + \mathcal{X}_{1}^{n} + \mathcal{X}_{2}^{n} + \mathcal{X}_{3}^{n}, \label{Temp 1}
\end{align}
where the residuals $\mathcal{X}_{2}^{n}$ and $\mathcal{X}_{3}^{n}$ denote 
\begin{align*}
    \mathcal{X}_{2}^{n}=& \frac{1}{4}\left( \mathcal{F}[\xi_{1}^{n}] \right)^{-\frac{5}{2}}\left( \int_{\Omega}{ f(\xi_{1}^{n})(\phi^{n}-\phi^{n-1}) }\,{\rm d}{\bf x} \right)^3 - \frac{1}{2}\left( \mathcal{F}[\xi_{1}^{n}] \right)^{-\frac{3}{2}}\left( \int_{\Omega}{ f(\xi_{1}^{n})(\phi^{n}-\phi^{n-1}) }\,{\rm d}{\bf x} \right) \nonumber\\
    & \times\left( \int_{\Omega}{ f'(\xi_{1}^{n})(\phi^{n}-\phi^{n-1})^2 }\,{\rm d}{\bf x} \right) + \frac{1}{3}\left( \mathcal{F}[\xi_{1}^{n}] \right)^{-\frac{1}{2}} \int_{\Omega}{ f''(\xi_{1}^{n})(\phi^{n}-\phi^{n-1})^3 }\,{\rm d}{\bf x}, \nonumber\\
    \mathcal{X}_{3}^{n}=& -\frac{1}{2}\left( \mathcal{F}[\xi_{2}^{n}] \right)^{-\frac{5}{2}}\left( \int_{\Omega}{ f(\xi_{2}^{n})(\phi^{n}-\phi^{n-1}) }\,{\rm d}{\bf x} \right)^3 + \left( \mathcal{F}[\xi_{2}^{n}] \right)^{-\frac{3}{2}}\left( \int_{\Omega}{ f(\xi_{2}^{n})(\phi^{n}-\phi^{n-1}) }\,{\rm d}{\bf x} \right) \nonumber\\
    & \times\left( \int_{\Omega}{ f'(\xi_{2}^{n})(\phi^{n}-\phi^{n-1})^2 }\,{\rm d}{\bf x} \right) - \frac{2}{3}\left( \mathcal{F}[\xi_{2}^{n}] \right)^{-\frac{1}{2}} \int_{\Omega}{ f''(\xi_{2}^{n})(\phi^{n}-\phi^{n-1})^3 }\,{\rm d}{\bf x},
\end{align*}
with $ \xi_{1}^{n} = (2+\rho_{1})\phi^{n}  + (1+\rho_{1})\phi^{n-1}, \  \xi_{2}^{n} = (2+\rho_{2})\phi^{n}  + (1+\rho_{2})\phi^{n-1}$. Subtracting (\ref{iSAV-BDF scheme 3}) from (\ref{Temp 1}) and concerning the second-order accuracy of iSAV-BDF scheme, we find formally 
\begin{align}
    r[\phi^{n+1}]-\tilde{r}^{n+1} = \frac{1}{3}\left( \mathcal{X}_{1}^{n} + \mathcal{X}_{2}^{n} + \mathcal{X}_{3}^{n} \right) = \mathcal{O}(\tau^3). \label{Temp 2}
\end{align}
Combining (\ref{Temp 3}), (\ref{Temp 4}) and (\ref{Temp 2}), we can derive that
\begin{align*}
    &\frac{1}{\tau}\left( \mathcal{E}_{2}[\phi^{n+1},\phi^{n}] - \mathcal{E}_{2}[\phi^{n},\phi^{n-1}] \right) \nonumber\\
    =& \frac{1}{\tau}\left( \tilde{E}_{2}^{n+1} - \mathcal{E}_{2}[\phi^{n},\phi^{n-1}] \right) + \frac{1}{\tau}\left( \mathcal{E}_{2}[\phi^{n+1},\phi^{n}] - \tilde{E}_{2}^{n+1} \right) \\
    =& -\left\| \mathcal{G}^\frac{1}{2}\mu^{n+1} \right\|^2 - \mathcal{N}^{n} -\frac{S}{\tau}\left\| \phi^{n+1}-2\phi^{n}+\phi^{n-1} \right\|^2 + \frac{1}{2\tau}\left( r[\phi^{n+1}] - \tilde{r}^{n+1} \right)\left( 5r[\phi^{n+1}] + 5\tilde{r}^{n+1} - 4r[\phi^{n}] \right) \\
    \leq& -\left\| \mathcal{G}^\frac{1}{2}\mu^{n+1} \right\|^2 + \mathcal{O}(\tau^2),
\end{align*}
which gives (\ref{Energy Law of iSAV-BDF}).

\section*{Acknowledgements}
The work is supported by NSFC 12271413, 12001055 and the Natural Science Foundation of Hubei Province 2019CFA007. We thanks Prof. Jie Shen for helpful discussion, suggestion and encouragement.

\bibliographystyle{siam}
\bibliography{ref}

\end{document}